\documentclass[a4paper]{amsart}

\usepackage{amsmath,amstext,amssymb,mathrsfs,amscd,amsthm,indentfirst}
\usepackage{amsfonts}
\usepackage{mathtools}
\usepackage{stmaryrd}

\usepackage{booktabs}
\usepackage{verbatim}
\usepackage{enumerate}

\usepackage{graphicx}
\numberwithin{equation}{section}

\newcommand{\CircNum}[1]{\ooalign{\hfil\raise .00ex\hbox{\scriptsize #1}\hfil\crcr\mathhexbox20D}}

\newcommand{\bF}{\mathbb{F}}

\newcommand{\bN}{\mathbb{N}}

\newcommand{\bR}{\mathbb{R}}

\newcommand{\bZ}{\mathbb{Z}}

\newcommand{\MTtheta}{\mathrm{MT \theta}}

\newcommand\lra{\longrightarrow}

\newcommand\Diff{\mathrm{Diff}}
\newcommand\Emb{\mathrm{Emb}}
\newcommand\End{\mathrm{End}}
\newcommand\Bun{\mathrm{Bun}}

\newcommand\colim{\operatorname*{colim}}
\newcommand\hocolim{\operatorname*{hocolim}}

\newcommand\Ker{\operatorname*{Ker}}

\newcommand{\hcoker}{/\!\!/}

\newcommand{\cob}{\mathcal{D}}

\newcommand{\Lk}{\mathrm{Lk}}

\newcommand{\R}{\bR}

\newcommand{\Int}{\mathrm{int}}

\newcommand{\Hom}{\mathrm{Hom}}

\newcommand{\hAut}{\mathrm{hAut}}

\renewcommand{\epsilon}{\varepsilon}

\newcommand{\Mod}{\mathrm{Mod}}

\newcommand{\MM}{\mathscr{M}}
\newcommand{\NN}{\mathscr{N}}

\newcommand{\Ob}{\mathrm{Ob}}
\newcommand{\Mor}{\mathrm{Mor}}

\newcommand{\supp}{\mathrm{supp}}

\mathchardef\ordinarycolon\mathcode`\:
\mathcode`\:=\string"8000
\begingroup \catcode`\:=\active
  \gdef:{\mathrel{\mathop\ordinarycolon}}
\endgroup

\usepackage[all]{xy}
\CompileMatrices

\usepackage{amsthm}
\theoremstyle{plain}

\newtheorem{theorem}{Theorem}[section]

\newtheorem{proposition}[theorem]{Proposition}
\newtheorem{lemma}[theorem]{Lemma}

\newtheorem{corollary}[theorem]{Corollary}

\theoremstyle{definition}
\newtheorem{definition}[theorem]{Definition}

\newtheorem{claim}[theorem]{Claim}
\newtheorem{construction}[theorem]{Construction}

\theoremstyle{remark}

\newtheorem{remark}[theorem]{Remark}
\newtheorem*{remark*}{Remark}

\title[Homological stability II]{Homological stability for moduli spaces
  of high dimensional manifolds. II}

\author{S{\o}ren Galatius}
\email{galatius@stanford.edu}
\address{Department of Mathematics\\
	Stanford University\\
	Stanford CA, 94305}

\author{Oscar Randal-Williams}
\email{o.randal-williams@dpmms.cam.ac.uk}
\address{Centre for Mathematical Sciences\\
Wilberforce Road\\
Cambridge CB3 0WB\\
UK}

\dedicatory{Dedicated to Michael Weiss on the occasion of his 60th birthday}

\date{\today}

\begin{document}
\begin{abstract}
  We prove a homological stability theorem for moduli spaces of
  manifolds of dimension $2n$, for attaching handles of index at least
  $n$, after these manifolds have been stabilised by countably many
  copies of $S^n \times S^n$.

  Combined with previous work of the authors, we obtain an analogue of
  the Madsen--Weiss theorem for any simply-connected manifold of
  dimension $2n \geq 6$.
\end{abstract}
\maketitle

\section{Introduction and statement of results}

Let $W$ be a smooth compact connected manifold of dimension
$2n \geq 2$ and let $\Diff_\partial(W)$ denote the topological group
of diffeomorphisms of $W$ relative to its boundary.  This paper
concludes our study (\cite{GR-W2, GR-W3}) of the cohomology of the
classifying space $B\Diff_\partial(W\sharp g(S^n \times S^n))$ as well
as some variants, resulting in a formula for its cohomology in the
limit $g \to \infty$ in terms of a purely homotopy theoretic
construction.

The main result is a new \emph{stable homological stability} theorem,
about the dependency of (a variant of)
$B\Diff_\partial(W \sharp g(S^n \times S^n))$ on $W$ in the limit
$g \to \infty$.  Roughly speaking, it asserts that this limit depends
only on the $n$-skeleton of $W$ together with the the restriction of
the tangent bundle, up to tangential homotopy equivalence.

When combined with our earlier work on the subject, we deduce a
homotopy theoretic formula for the cohomology of
$B\Diff_\partial (W \sharp g(S^n \times S^n))$ for finite $g$ in a
range of degrees increasing with $g$.  The formula applies when
$2n \geq 6$ and $W$ is simply-connected.  This type of formula is
well-suited for calculations, and has applications to positive scalar
curvature (\cite{BER-W}), topological Pontryagin classes
(\cite{Dalian}), and homotopy equivalences of manifolds
(\cite{BerglundMadsen2}).

\subsection{Statement of the main theorem}

In order to state our homological stability results more precisely, we
must first introduce a variant of the space $B\Diff_\partial(W)$ which
incorporates certain tangent bundle information. As in our earlier
paper \cite{GR-W2} the clearest statements can be made by not fixing a
manifold $W$, but rather considering all manifolds with a specified
boundary at once.

To define this modified classifying space, we fix a map
$\theta : B \to BO(2n)$ classifying a vector bundle
$\theta^*\gamma_{2n}$ over $B$, and we shall suppose throughout that
$B$ is path connected. We shall call a bundle map
$TW \to \theta^*\gamma_{2n}$ a \emph{$\theta$-structure} on $W$.

\begin{definition}\label{defn:1.1}
  Let $W$ be a $2n$-dimensional manifold with boundary $P$, and
  $\hat{\ell}_P : \epsilon^1 \oplus TP \to \theta^*\gamma_{2n}$ be a
  bundle map. Let $\Bun^\theta_{n, \partial}(TW;\hat{\ell}_{P})$
  denote the space of all bundle maps
  $\hat{\ell} : TW \to \theta^*\gamma_{2n}$ which are equal to
  $\hat{\ell}_{P}$ when restricted to $P$ and whose underlying map
  $\ell: W \to B$ is $n$-connected. Let
  $$\NN^\theta_n(P, \hat{\ell}_P) = \coprod_{[W]}
  \Bun^\theta_{n, \partial}(TW;\hat{\ell}_{P}) \hcoker
  \Diff_\partial(W),$$
  where the disjoint union is taken over compact manifolds $W$,
  equipped with a diffeomorphism $\partial W \cong P$, one in each
  diffeomorphism class relative to $P$.  The topological group
  $\Diff_\partial(W)$ is the group of diffeomorphisms of $W$
  restricting to the identity on a neighbourhood of $\partial W$, and
  ``$\hcoker$'' denotes the homotopy orbit space.
\end{definition}

If $M$ is a cobordism between $P$ and $Q$ such that $(M,Q)$ is
$(n-1)$-connected, and is equipped with a $\theta$-structure
$\hat{\ell}_M$ restricting to $\hat{\ell}_P$ over $P$ and
$\hat{\ell}_Q$ over $Q$, then there is an induced map
\begin{equation}\label{eq:22}
  - \cup_Q (M, \hat{\ell}_M) : \NN^\theta_n(Q, \hat{\ell}_Q) \lra
  \NN^\theta_n(P, \hat{\ell}_P),
\end{equation}
which on the path components corresponding to $W$ is induced by the
$\Diff_\partial(W)$-equivariant map
$\Bun^\theta_{n, \partial}(TW;\hat{\ell}_{Q}) \to
\Bun^\theta_{n, \partial}(T(W \cup_Q M );\hat{\ell}_{P})$ defined by
extending bundle maps by $\hat{\ell}_M$, and extending diffeomorphisms
of $W$ to $W \cup_Q M$ by the identity map of $M$.

Our main result will concern the homological effect of~(\ref{eq:22})
once all manifolds in sight have been stabilised by connect-sum with
countably many copies of $S^n \times S^n$. We shall give a point-set
model for forming this stabilisation in Section
\ref{sec:DefnStatement}, but for now the following explanation of
diagram~(\ref{eq:IntroStabMap}) below will suffice.

Let $p: [0,1] \times D^{2n-1} \hookrightarrow M$ be an embedding
sending $\{0\} \times D^{2n-1}$ into $Q$ and $\{1\} \times D^{2n-1}$
into $P$. (In particular, $P$ and $Q$ must be non-empty.) We may then
form the connect-sum
$${_Q H} = ([0,1] \times Q) \sharp (S^n \times S^n)$$
inside the image of the embedding
$[0,1] \times D^{2n-1} \hookrightarrow [0,1] \times Q$ in a standard
way; we may analogously form ${_P H}$.  Unless $P$ and $Q$ are path
connected, the diffeomorphism types of ${_P H}$ and ${_Q H}$ may
depend on $p$ and not just $P$ and $Q$, but we suppress this
dependence from the notation.  Sliding along the thickened path $p$
induces a diffeomorphism
\begin{equation}\label{eq:diffeo}
  M \cup_P ({_P H}) \cong ({_Q H}) \cup_Q M
\end{equation}
relative to the common boundaries $P$ and $Q$.

\begin{definition}\label{defn:admissible}
  Let $W_{1,1} = S^n \times S^n \setminus \mathrm{int}(D^{2n})$ and
  let us say that a $\theta$-structure
  $\hat{\ell} : TW_{1,1} \to \theta^*\gamma$ is \emph{admissible} if
  there is a pair of orientation-preserving embeddings
  $e, f : S^n \times D^n \hookrightarrow W_{1,1} \subset S^n \times
  S^n$ whose cores $e(S^n \times \{0\})$ and $f(S^n \times \{0\})$
  intersect transversely in a single point, such that each of the
  $\theta$-structures $e^*\hat{\ell}$ and $f^*\hat{\ell}$ on
  $S^n \times D^n$ extend to $\bR^{2n}$ for some
  orientation-preserving embeddings
  $S^n \times D^n \hookrightarrow \bR^{2n}$.
\end{definition}  

The manifolds ${_P H}$ and ${_Q H}$ each contain an embedded copy of
$W_{1,1}$ and we shall say that $\theta$-structures on ${_P H}$ and
${_Q H}$ are admissible if they restrict to admissible
$\theta$-structures on $W_{1,1}$. Choosing an admissible
$\theta$-structure $\hat{\ell}_{{_P H}}$ on ${_P H}$ which agrees with
$\hat\ell_P$ on $\{0\} \times P$, we obtain a $\theta$-structure
$\hat{\ell}_M \cup \hat{\ell}_{{_P H}}$ on $M \cup_P ({_P H})$, and
hence via the diffeomorphism \eqref{eq:diffeo} a $\theta$-structure on
$({_Q H}) \cup_Q M$. Restricted to ${_Q H}$ this is an admissible
$\theta$-structure $\hat{\ell}_{{_Q H}}$ which agrees with
$\hat\ell_Q$ on $\{1\} \times Q$.  Restricted to $M$ it gives a new
$\theta$-structure $\hat{\ell}'_M$, which restricts to new
$\theta$-structures $\hat{\ell}'_P$ over $P$ and $\hat{\ell}'_Q$ over
$Q$. These manifolds with $\theta$-structure induce maps in the
diagram
\begin{equation}\label{eq:IntroStabMap}
  \begin{gathered}
    \xymatrix{
      {\NN_n^\theta(Q, \hat{\ell}_Q)} \ar[d]_-{- \cup_Q (M, \hat{\ell}_M)} \ar[rr]^-{- \cup_Q ({_Q H}, \hat{\ell}_{_Q H})} && {\NN_n^\theta(Q, \hat{\ell}_{Q}')} \ar[d]^-{- \cup_Q (M, \hat{\ell}_{M}')} \\
      {\NN_n^\theta(P,\hat{\ell}_P)} \ar[rr]_-{- \cup_P ({_P H}, \hat{\ell}_{_P H})} && {\NN_n^\theta(P,\hat{\ell}_{P}'),}
    }
  \end{gathered}
\end{equation}
and the diffeomorphism~\eqref{eq:diffeo} induces a homotopy between
the two compositions. Iterating this construction, we obtain an
induced map from
$$\hocolim(\NN_n^\theta(Q, \hat{\ell}_Q)  \to \NN_n^\theta(Q, \hat{\ell}'_{Q}) \to \NN_n^\theta(Q, \hat{\ell}''_{Q}) \to \cdots)$$
to the analogous homotopy colimit for $(P, \hat{\ell}_P)$.  This
limiting map is denoted $- \cup_Q (M, \hat{\ell}_M^{(\infty)})$ in the
following theorem, which is the main new result of this paper.

\begin{theorem}\label{thm:Main}
  The map
  $$- \cup_Q (M, \hat{\ell}_M^{(\infty)}) : \hocolim_{g \to \infty} \NN_n^\theta(Q, \hat{\ell}_{Q}^{(g)}) \lra \hocolim_{g \to \infty} \NN_n^\theta(P, \hat{\ell}_{P}^{(g)})$$
  induces an isomorphism on homology with all abelian systems of local
  coefficients.
\end{theorem}

\begin{remark}
  A system of local coefficients $\mathcal{L}$ on a space $X$ is
  called \emph{abelian} if for each point $x \in X$ the action of
  $\pi_1(X,x)$ on the abelian group $\mathcal{L}(x)$ factors through
  an abelian group. In particular, all constant coefficient systems
  are abelian. If $X$ is path connected and based then an abelian
  system of local coefficients is equivalent to the data of a module
  over the group ring $\mathbb{Z}[H_1(X;\mathbb{Z})]$.
\end{remark}

\subsection{Stable homology}

The homology of the limiting spaces
\begin{equation*}
  \hocolim_{g \to \infty} \NN_n^\theta(P, \hat{\ell}^{(g)}_{P})
\end{equation*}
may in many cases be deduced from the results of \cite{GR-W2}.  In
fact, in Section \ref{sec:StabHomAndGC} we shall revisit the results
of \cite{GR-W2} and see that the two sections of that paper concerning
``surgery on objects'' may be replaced by the homological stability
results of the present paper, by following a line of argument inspired
by~\cite{Tillmann}.  We shall see that in all cases the homology of
the limiting space agrees with the homology of the infinite loop space
of a certain Thom spectrum.  This approach using homological stability
does not give substantially simpler or easier proofs of the results of
\cite{GR-W2}, but it does work in more generality.  For example, we
shall be able to remove the restriction $2n \geq 6$ from \cite{GR-W2}.

Recall that we write $\MTtheta$ for the Thom spectrum of the virtual
vector bundle $-\theta^*\gamma_{2n}$ over $B$, and
$\Omega^\infty \MTtheta$ for its infinite loop space.

\begin{theorem}\label{thm:StableHomology}
  If $\NN_n^\theta(P, \hat{\ell}_{P}) \neq \emptyset$ then there is a
  map
  $$\hocolim_{g \to \infty} \NN_n^\theta(P, \hat{\ell}_{P}^{(g)}) \lra
  \Omega^{\infty} \MTtheta$$ which induces an isomorphism on homology,
  and is in fact acyclic.
\end{theorem}

In fact, we will construct a specific acyclic map, using a cobordism
category $\mathcal{C}_\theta$. The space
$\NN_n^\theta(P, \hat{\ell}_{P})$ will be modelled as a subspace of
the space of morphisms
$\mathcal{C}_\theta(\emptyset, (P, \hat{\ell}_P))$ in
$\mathcal{C}_\theta$, and hence comes with a canonical map
\begin{equation*}
  \NN_n^\theta(P,\hat{\ell}_P) \subset
  \mathcal{C}_\theta(\emptyset,(P,\hat{\ell}_P)) \lra
  \Omega_{[\emptyset,(P,\hat{\ell}_P)]} B\mathcal{C}_\theta,
\end{equation*}
where $B\mathcal{C}_\theta$ denotes the classifying space (geometric
realisation of the nerve) and $\Omega_{[\emptyset,(P,\hat{\ell}_P)]}$
denotes the space of paths between the points corresponding to the
objects $\emptyset$ and $(P,\hat{\ell}_P)$.  The main theorem of
\cite{GMTW} provides a weak equivalence from this space of paths, when
non-empty, to $\Omega^{\infty} \MTtheta$. This is the map which the
theorem is about.

\begin{remark}
  A map $f : X \to Y$ is \emph{acyclic} if for any system of local
  coefficients $\mathcal{L}$ on $Y$ the induced map
  $f_* : H_*(X;f^*\mathcal{L}) \to H_*(Y;\mathcal{L})$ is an
  isomorphism. This is equivalent to the homotopy fibre of $f$ over
  any point of $Y$ being acyclic, i.e.\ having the integral homology
  of a point.

  In the statement of Theorem \ref{thm:StableHomology} the target has
  abelian fundamental group. Thus every system of local coefficients
  is abelian, and so the map being acyclic is equivalent to inducing
  an isomorphism on homology with all abelian systems of local
  coefficients.
\end{remark}

Theorem \ref{thm:StableHomology}, and the slightly more general
Theorem \ref{thm:StableHomologyBetter} below, improves on
\cite[Theorem 1.8]{GR-W2} in a few ways. Firstly, it applies to
manifolds of all even dimensions, and not just dimensions $2n \geq 6$;
secondly, it does not require that $\theta$ be ``spherical'' (a
technical condition introduced in \cite{GR-W2}); thirdly, it requires
a simpler form of stabilisation (only copies of $W_{1,1}$ have to be
glued on, rather than a ``universal $\theta$-end''); fourthly, the
conclusion is that the map is acyclic rather than merely a homology
equivalence.  In the special case $2n = 2$ we recover the main result
of \cite{MW}.  In that case the acyclicity of the map seems to be new
(when $\pi_1 MT\theta \neq 0$), as does the extension to completely
general $\theta: B \to BO(2)$.

\subsection{Finite genus and closed
  manifolds}\label{sec:intro:finitegenus}

We can combine Theorem \ref{thm:Main} with our earlier work
\cite{GR-W3} to obtain results about the homology of the spaces
$\NN_n^\theta(P;\hat{\ell}_{P})$ before stabilising by $W_{1,1}$.  If
we write
$\MM_n^\theta(W,\hat{\ell}_W) \subset \NN_n^\theta(Q,\hat{\ell}_Q)$
for the path component containing $(W,\hat{\ell}_W)$, then the
map~(\ref{eq:22}) restricts to a map
\begin{equation}\label{eq:23}
  - \cup_Q (M, \hat{\ell}_M) : \MM^\theta_n(W, \hat{\ell}_W) \lra
  \MM^\theta_n(W', \hat{\ell}_{W'}),
\end{equation}
where $W' = W \cup_Q M$ and $\hat{\ell}_{W'}$ is the bundle map
obtained by gluing $\hat{\ell}_W$ and $\hat{\ell}_M$.  Recall from
\cite{GR-W3} that the \emph{genus} of a $\theta$-manifold is defined
as
\begin{equation*}
  g^\theta(W,\hat{\ell}_W) = \max\left\{g \in \bN
    \,\,\bigg|\,\, \parbox{18em}{there are $g$ disjoint copies of
      $W_{1,1}$ in $W$,\\ each with admissible
      $\theta$-structure}\right\}
\end{equation*}
and the \emph{stable genus} is defined as
\begin{equation*}
  \bar{g}^\theta(W,\hat{\ell}_W) = \max\{g^\theta(W\sharp W_{k,1},
  \hat{\ell}^{(k)}_W)-k\,\,|\,\, k \in \bN \},
\end{equation*}
where $W \sharp W_{k,1}$ is the manifold obtained from $W$ by cutting
out a disc and forming the boundary connected sum with $k$ copies of
$W_{1,1}$.  This glued manifold is equipped with a $\theta$-structure
$\hat{\ell}^{(k)}_W$ obtained by extending the restriction of
$\hat{\ell}_W$ by any admissible structures on the $W_{1,1}$.  In
Section~\ref{sec:finite-genus-stab} we shall explain how to deduce the
following two finite-genus results from the results in this paper and
\cite{GR-W3}.  (In fact we shall deduce slightly more general results
which imply the corollaries but which also imply some homological
stability for $H_0(\NN_n^\theta(-);\bZ)$.)

\begin{corollary}\label{cor:Main1}
  Assume $2n \geq 6$,
  $B$ is simply-connected, and $(M, \hat{\ell}_M)$ is a
  $\theta$-cobordism between $P$ and $Q$ such that $(M,Q)$ is
  $(n-1)$-connected, and denote by
  $\MM_n^\theta(W,\hat{\ell}_W)\subset \NN_n^\theta(Q,\hat{\ell}_Q)$
  the path component containing $(W,\hat{\ell}_W)$.  Write
  $g = \bar{g}^\theta(W,\hat{\ell}_W)$, and let $\mathcal{L}$ be an
  abelian coefficient system on the target of the map~(\ref{eq:23}).
  Then the induced map
  \begin{equation*}
    (- \cup_Q (M,\hat{\ell}_M))_* : H_i(\MM^\theta_n(W,
    \hat{\ell}_W);(- \cup_Q M)^* \mathcal{L}) \lra
    H_i(\MM^\theta_n(W', \hat{\ell}_{W'});\mathcal{L})
  \end{equation*}
  is an isomorphism, provided one of the following three assumptions
  hold.
  \begin{enumerate}[(i)]
  \item $2i \leq {g-3}$, $\mathcal{L}$ is constant, and $\theta$ is
    spherical;
  \item $3i \leq {g-4}$ and $\mathcal{L}$ is constant;
  \item $3i \leq {g-4}$, $g \geq 5$, and either $g \geq 7$ or $\theta$
    is spherical.
\end{enumerate}
\end{corollary}

The situation is formally quite similar to both \cite{HV} and
\cite{Wahl}, from which we learnt the following trick.  Suppose one
wishes to prove homological stability for a stabilisation map $\alpha$
which increases genus and leaves the boundary unchanged and also for
another stabilisation map $\beta$ which changes the boundary.  If
$\alpha$ and $\beta$ commute up to homotopy and stability has already
been shown for $\alpha$, then it suffices to prove stability for
$\beta$ after taking the colimit over stabilising infinitely many
times by $\alpha$.

\begin{corollary}\label{cor:Main2}
  There is a
  map
  $$\MM_n^\theta(W,\hat{\ell}_W) \lra \Omega^{\infty}_0 \MTtheta$$
  which, under the same conditions on $n$ and $B$ as in
  Corollary~\ref{cor:Main1}, is acyclic in degrees satisfying
  $3* \leq {\bar{g}^\theta(W,\hat{\ell}_W)-4}$, and induces an
  isomorphism on integral homology in degrees satisfying
  $2* \leq {\bar{g}^\theta(W,\hat{\ell}_W)-3}$ if in addition $\theta$
  is spherical.
\end{corollary}
As in Theorem~\ref{thm:StableHomology} and the discussion following
it, Corollary~\ref{cor:Main2} concerns a specific map into a component
of a path space of $B\mathcal{C}_\theta$, which by \cite{GMTW} may be
identified with the basepoint component $\Omega^\infty_0 \MTtheta$ of
the infinite loop space $\Omega^\infty \MTtheta$.

\subsection{Other tangential structures}

The definition of $\NN_n^\theta(P,\hat{\ell}_P)$ used a space
$\Bun_{n,\partial}^\theta(TW;\hat{\ell}_P)$, consisting of all
$n$-connected bundle maps with fixed boundary condition.  If we
replace this by the space of all (i.e.\ not necessarily $n$-connected)
bundle maps with fixed boundary condition, then the results of this
paper may still be used to deduce the homology of the resulting moduli
space in a stable range.  Again we state the result one path component
at a time, and for simplicity we shall restrict attention to closed
manifolds in this introduction.

Let $\theta': B' \to BO(2n)$ be a map, and let $W$ be a closed
$2n$-dimensional manifold and
$\hat{\ell}'_W: TW \to (\theta')^* \gamma$ a $\theta'$-structure.  We
shall write $\Bun^{\theta'}(TW)$ for the space of all (not necessarily
$n$-connected) $\theta'$-structures and
\begin{equation*}
  \MM^{\theta'}(W,\hat{\ell}_W) \subset \Bun^{\theta'}(TW) \hcoker \Diff(W)
\end{equation*}
for the path component containing the image of $\hat{\ell}'_W$.  Let
$\ell'_W: W \to B'$ be the map underlying $\hat{\ell}'_W$ and let
$W \to B \to B'$ be a Moore--Postnikov $n$-stage factorisation of
$\ell'_W$, i.e.\ a factorisation into an $n$-connected cofibration
$\ell_W: W \to B$ and an $n$-co-connected fibration $u: B \to B'$. Let
us write $\theta = \theta' \circ u : B \to BO(2n)$. In
Section~\ref{sec:GeneralTheta} below, we explain how to deduce the
following result from Corollary~\ref{cor:Main2}.

\begin{corollary}\label{cor:GeneralThetaCalc}
  In this situation the topological monoid $\mathrm{hAut}(u)$,
  consisting of weak equivalences $B \to B$ commuting with the map
  $u: B \to B'$, acts on the space $\Omega^\infty \MTtheta$ and there
  is a map
  \begin{equation*}
    \MM^{\theta'}(W,\hat{\ell}'_W) \lra \left(\Omega^\infty \MTtheta\right)
    \hcoker \mathrm{hAut}(u).
  \end{equation*}
  Considered as a map to the path component which it hits, if
  $2n \geq 6$ and $W$ is simply-connected then this map is acyclic in
  degrees satisfying
  $3* \leq {\bar{g}^{\theta'}(W,\hat{\ell}'_{W})-4}$, and induces an
  isomorphism in integral homology in degrees satisfying
  $2* \leq {\bar{g}^{\theta'}(W,\hat{\ell}'_{W})-3}$ if in addition
  $\theta'$ is spherical.
\end{corollary}

The proof given in Section~\ref{sec:GeneralTheta} again concerns a
specific map, and we will also treat the case
$\partial W \neq \emptyset$.  Let us briefly spell out two important
special cases.

Taking $\theta' = \mathrm{Id} : BO(2n) \to BO(2n)$ and letting
$W \overset{\ell_W}\to B \overset{u}\to BO(2n)$ be a Moore--Postnikov
$n$-stage factorisation of a Gauss map $\ell_W': W \to BO(2n)$, the
space $\MM^{\theta'}(W,\hat{\ell}'_W)$ is a model for $B\Diff(W)$ and
we obtain a map
\begin{equation}\label{eq:UnorientedComparisonMap}
  B\Diff(W) \lra
  \left(\Omega^\infty \MTtheta\right) \hcoker \hAut(u)
\end{equation}
which, as long as $2n \geq 6$ and $W$ is simply-connected, is a
homology isomorphism (respectively acyclic) in a range of degrees
depending on $g(W)$, onto the path component which it hits.

Similarly we can take $\theta' : BSO(2n) \to BO(2n)$ to be the
orientation cover and let
$W \overset{\ell_W}\to B \overset{u^+}\to BSO(2n)$ be a
Moore--Postnikov $n$-stage factorisation of an \emph{oriented} Gauss
map $\ell'_W: W \to BSO(2n)$.  Then the space
$\MM^{\theta'}(W,\hat{\ell}'_W)$ is a model for $B\Diff^+(W)$ and we
obtain a map
\begin{equation*}
  B\Diff^+(W)  \lra
  \left(\Omega^\infty \MTtheta\right) \hcoker
  \hAut(u^+)
\end{equation*}
which, as long as $2n \geq 6$ and $W$ is simply-connected, is a
homology isomorphism (respectively acyclic) in a range of degrees
depending on $g(W)$, onto the path component which it hits.

\subsection{Acknowledgements}

S.\ Galatius was partially supported by NSF grants DMS-1105058 and
DMS-1405001, the European Research Council (ERC) under the European
Union's Horizon 2020 research and innovation programme (grant
agreement No 682922), as well as the Danish National Research
Foundation through the Centre for Symmetry and Deformation (DNRF92)
and ERC-682992. O.\ Randal-Williams was partially supported by the
Herchel Smith Fund and EPSRC grant EP/M027783/1.

%%% Local Variables: 
%%% mode: latex
%%% TeX-master: "stability2"
%%% End: 
 % Introduction and statement of results
\tableofcontents
\section{Definitions}\label{sec:DefnStatement}

In the following we shall suppose that $n > 0$. We emphasise that we
allow $n=1$ and $n=2$, and we will point out when these cases need
special treatment.

\subsection{Cobordism categories}\label{sec:CobCats}

It is convenient to state the precise version of Theorem
\ref{thm:Main} using the language of cobordism categories; these are
also useful for the proof of Theorem \ref{thm:Main}, and will be
essential for the proof of Theorem \ref{thm:StableHomology}. See
\cite{GMTW, GR-W, GR-W2} for further background on these
categories. Throughout the paper we shall work with manifolds inside
$\bR \times \bR^\infty$, and will write $e_0$ for the basis vector of
the first copy of $\bR$. Apart from this, we will write
$e_1, e_2, \ldots, e_k$ for the basis vectors of
$\bR^k \subset \bR^\infty$.

\begin{definition}\label{defn:CobCat}
  Let $\mathcal{C}_{\theta}$ be the (non-unital) category with objects given by
  \begin{enumerate}[(i)]
  \item a $(2n-1)$-dimensional closed submanifold $P \subset \bR^\infty$, and
  \item a bundle map $\hat{\ell}_P : \epsilon^1 \oplus TP \to \theta^*\gamma$.
  \end{enumerate}
  A morphism from $(P, \hat{\ell}_P)$ to $(Q, \hat{\ell}_Q)$ is given by
  \begin{enumerate}[(i)]
  \item an $s \in (0, \infty)$,
  \item a $2n$-dimensional submanifold
    $W \subset [0,s] \times \bR^\infty$ which for some $\epsilon > 0$
    intersects
    $$([0,\epsilon) \times \bR^\infty) \cup  ((s-\epsilon,s] \times \bR^\infty)$$
    in $([0,\epsilon) \times P) \cup ((s-\epsilon,s] \times Q)$,
  \item a bundle map $\hat{\ell}_W : TW \to \theta^*\gamma$ which
    restricts to $\hat{\ell}_P$ on $\{0\} \times P$ and to
    $\hat{\ell}_Q$ on $\{s\} \times Q$.
\end{enumerate}

The composition of
$(s, W, \hat{\ell}_W) : (P, \hat{\ell}_P) \leadsto (Q, \hat{\ell}_Q)$
and
$(s', W', \hat{\ell}_{W'}) : (Q, \hat{\ell}_Q) \leadsto (R,
\hat{\ell}_R)$ is denoted
$(s',W',\hat{\ell}_{W'}) \circ (s,W,\hat{\ell}_W)$ and defined as
$$(s+s', W \cup (W' + s \cdot e_0), \hat{\ell}_W \cup \hat{\ell}_{W'})
: (P, \hat{\ell}_P) \leadsto (R, \hat{\ell}_R).$$
We make $\mathcal{C}_{\theta}$ into a topological (non-unital)
category using the evident isomorphism with the topological
(non-unital) category described in \cite[Definition 2.6]{GR-W2}. In
particular, the topology on the morphism space
$\mathcal{C}_\theta((P, \hat{\ell}_P),(Q, \hat{\ell}_Q))$ is that of a
subquotient of
$$(0,\infty) \times \left( \coprod_{[W]} \Emb(W, [0,1] \times \bR^\infty) \times \Bun^\theta(TW)\right),$$
where $W$ ranges over the set of all cobordisms from $P$ to $Q$, one
in each diffeomorphism class.
\end{definition}

\begin{definition}
  Let $L \subset (-\infty,0] \times \bR^{\infty-1}$ be a compact
  $(2n-1)$-dimensional submanifold which near
  $\{0\} \times \bR^{\infty-1}$ coincides with
  $(-\infty,0] \times \partial L$, and let
  $\hat{\ell}_L : \epsilon^1 \oplus TL \to \theta^*\gamma_{2n}$ be a
  $\theta$-structure.

  Let $\mathcal{C}_{\theta, L} \subset \mathcal{C}_\theta$ be the subcategory with
  \begin{enumerate}[(i)]
  \item objects those $(P, \hat{\ell}_P)$ such that $P \cap ((-\infty,0] \times \bR^{\infty-1}) = L$ and $\hat{\ell}_P\vert_L = \hat{\ell}_L$,
    
  \item morphisms those $(s,W, \hat{\ell}_W)$ such that 
    $$W \cap ([0,s] \times ((-\infty,0] \times \bR^{\infty-1})) = [0,s] \times L$$
    and $\hat{\ell}_W$ restricts to the bundle map $T([0,s] \times L) \to \epsilon^1 \oplus TL \overset{\hat{\ell}_L}\to \theta^*\gamma_{2n}$.
  \end{enumerate}
\end{definition}

We shall generally abbreviate $\theta$-manifolds by the name of the
underlying smooth manifold, say $X$, and write $\hat{\ell}_X$ for its
$\theta$-structure. Using this convention, for any object
$P \in \mathcal{C}_\theta$ and any $s \in (0,\infty)$ there is an
morphism $(s, [0,s] \times P)$, having underlying manifold
$[0,s] \times P \subset [0,s] \times \bR^\infty$ and
$\theta$-structure given by
$$\hat{\ell}_{[0,s] \times P} : T([0,s] \times P) \lra \epsilon^1 \oplus TP \overset{\hat{\ell}_P}\lra \theta^*\gamma_{2n}.$$
We shall call such morphisms ``cylindrical''.  We can make the
analogous construction for objects in $\mathcal{C}_{\theta, L}$.

\begin{remark}
  In this paper it suffices to work with the non-unital category
  $\mathcal{C}_\theta$ defined above.  If identities are desired, they
  could be added by replacing the morphism space with the pushout
  $[0,\infty) \times \Ob(\mathcal{C}_\theta) \leftarrow (0,\infty)
  \times \Ob(\mathcal{C}_\theta) \to \Mor(\mathcal{C}_\theta)$, along
  the map
  $(0,\infty) \times \Ob(\mathcal{C}_\theta) \to
  \Mor(\mathcal{C}_\theta)$ defined by sending $(t,P,\hat\ell_P)$ to
  $(t,[0,t] \times P, \hat\ell_P)$.

  This motivates calling a morphism $M : P \leadsto Q$ in
  $\mathcal{C}_\theta$ an \emph{isomorphism} if there exists a
  morphism $N : Q \leadsto P$ such that $M \circ N$ and $N \circ M$
  are in the same path components of $\mathcal{C}_\theta(Q,Q)$ and
  $\mathcal{C}_\theta(P,P)$ as the cylindrical morphisms
  $[0,s] \times Q$ and $[0,s] \times P$.  Pre- and postcomposing with
  an isomorphism $M: P \leadsto Q$ then define weak equivalences
  $\mathcal{C}_\theta(Q,R) \to \mathcal{C}_\theta(P,R)$ and
  $\mathcal{C}_\theta(R,P) \to \mathcal{C}_\theta(R,Q)$.  Similarly
  for $\mathcal{C}_{\theta,L}$.  The most important examples of
  isomorphisms will be cobordisms whose underlying manifold is
  diffeomorphic to a cylinder, but with non-cylindrical
  $\theta$-structure or embedding.
\end{remark}

\begin{definition}[\cite{GR-W2}]
  Let $\mathcal{C}_{\theta, L}^{n-1} \subset \mathcal{C}_{\theta, L}$
  be the subcategory having the same objects, and having those
  morphisms $(s, M) : P \leadsto Q$ for which the pair $(M, Q)$ is
  $(n-1)$-connected.
\end{definition}

The categories $\mathcal{C}_{\theta, L}$ do not really depend on $L$,
but rather only on $\partial L$: if $L'$ is another collared
$\theta$-manifold which is equal to $L$ near its boundary, then we
obtain an isomorphism of topological categories
$\mathcal{C}_{\theta, L} \overset{\sim}\to \mathcal{C}_{\theta, L'}$
by cutting out $L$ and gluing in $L'$ (and similarly on morphisms). In
order to make use of this, we make the following definition.

\begin{definition}\label{def:PreCat}
  If $P$ is an object of $\mathcal{C}_{\theta, L}$, let
  $P^\circ \subset [0,\infty) \times \bR^{\infty-1}$ be the
  $\theta$-manifold obtained by cutting out
  $\mathrm{int}(L)$. Similarly, if $(s,W)$ is a morphism in
  $\mathcal{C}_{\theta, L}$, let
  $W^\circ \subset [0,s] \times [0,\infty) \times \bR^{\infty-1}$ be
  the $\theta$-cobordism obtained by cutting out
  $[0,s] \times \mathrm{int}(L)$.

  Let $\mathcal{C}_{\theta, \partial L}$ be the (non-unital) category
  defined by
  \begin{align*}
    \Ob(\mathcal{C}_{\theta, \partial L}) &= \{(P^\circ, \hat{\ell}_{P^\circ}) \, \vert \, (P, \hat{\ell}_{P}) \in \Ob(\mathcal{C}_{\theta, L})\}\\
    \Mor(\mathcal{C}_{\theta, \partial L}) &= \{(s,W^\circ, \hat{\ell}_{W^\circ}) \, \vert \, (s,W, \hat{\ell}_W) \in \Mor(\mathcal{C}_{\theta, L})\}
  \end{align*}
  and made into a topological category by insisting that the functor
  $\mathcal{C}_{\theta, L} \to \mathcal{C}_{\theta, \partial L}$ given
  by $X \mapsto X^\circ$ is an isomorphism of topological categories.
\end{definition}

\begin{definition}\label{def:MainCat}
  Let
  $\mathcal{C}_{\theta, \partial L}^{n-1} \subset
  \mathcal{C}_{\theta, \partial L}$ be the subcategory having the same
  objects, and having those morphisms
  $(s, M^\circ) : P^\circ \leadsto Q^\circ$ for which the pair
  $(M^\circ, Q^\circ)$ is $(n-1)$-connected.

  This topological category shall play an especially important role
  and for brevity we shall often write
  $\cob = \mathcal{C}_{\theta,\partial L}^{n-1}$ when $\theta$ and
  $\partial L$ are fixed and understood.  When using the notation
  $\cob$, we shall implicitly assume that $\partial L \neq \emptyset$.
\end{definition}

\begin{remark}\label{rem:RemovingL}
  If
  $(s, M^\circ) : P^\circ \leadsto Q^\circ \in
  \mathcal{C}_{\theta, \partial L}^{n-1}$ then the cobordism
  $M = ([0,s] \times L) \cup M^\circ$ is $(n-1)$-connected relative to
  $Q = L \cup Q^\circ$, which defines a functor
  $\mathcal{C}_{\theta, \partial L}^{n-1} \to \mathcal{C}_{\theta,
    L}^{n-1}$. This is not in general an isomorphism of categories,
  but if $(L, \partial L)$ is $(n-1)$-connected then it is. (See
  \cite[Lemma 7.2]{GR-W2} for a related discussion.)
\end{remark}

We shall generally write $P, Q, \ldots$ for objects and $W, M, \ldots$
for morphisms of $\mathcal{C}_{\theta, \partial L}$ rather than
$\mathcal{C}_{\theta, L}$, and reserve the $(-)^\circ$ notation for
when we are directly comparing elements in $\mathcal{C}_{\theta, L}$
with their corresponding elements in
$\mathcal{C}_{\theta, \partial L}$.

For use in Section~\ref{sec:StabHomAndGC}, let us recall the following
result from \cite{GR-W2}. Together with the isomorphism
$\cob \cong \mathcal{C}_{\theta,L}^{n-1}$ from Remark
\ref{rem:RemovingL} and the equivalence
$B \mathcal{C}_\theta \simeq \Omega^{\infty-1} MT\theta$ from
\cite{GMTW}, it determines the weak homotopy type of $B\cob$ for
$(n-1)$-connected $(L,\partial L)$.
\begin{theorem}[\cite{GR-W2}]\label{thm:morphism-surgery-from-acta-paper}
  The maps
  \begin{equation}
    B\mathcal{C}_{\theta, L}^{n-1} \lra B\mathcal{C}_{\theta, L} \lra
    B\mathcal{C}_\theta,
  \end{equation}
  induced by the inclusion functors, are weak equivalences.
\end{theorem}
\begin{proof}
  This first equivalence is \cite[Theorem 3.1]{GR-W2}, one of the
  major technical results of that paper (``surgery on morphisms'').
  The second is Corollary 2.17 of op.\ cit.
\end{proof}

\subsection{Standard $\theta$-structures}\label{sec:StdThetaStr}

Let $\theta : B \to BO(2n)$ be a tangential structure, and recall that
we insist that $B$ is path connected. Choose once and for all a
``basepoint'' $\theta$-structure on the vector space $\bR^{2n}$, i.e.\
a bundle map $\tau: \bR^{2n} \to \theta^*\gamma_{2n}$ (depending on
the orientability of $\theta^*\gamma_{2n}$, there are one or two
possible such choices, up to homotopy). This induces a canonical
$\theta$-structure $\hat{\ell}^\tau_X$ on any framed $2n$- or
$(2n-1)$-manifold $X$.

\begin{definition}
  The \emph{standard framing} $\xi_{S^n \times D^n}$ on
  $S^n \times D^n$ is the framing induced by the embedding
  \begin{equation*}
    \begin{aligned}
      S^n \times D^n & \lra \bR^{n+1} \times \bR^{n-1} = \bR^{2n}\\
      (x; y_1, y_2, \ldots, y_n) &\longmapsto (2 e^{-\tfrac{y_1}{2}}x
      ; y_2, \ldots, y_n).
    \end{aligned}
  \end{equation*}
  This framing induces a $\theta$-structure
  $\hat{\ell}^\tau_{S^n \times D^n}$, and we say that a
  $\theta$-structure on $S^n \times D^n$ is \emph{standard} if it is
  homotopic to $\hat{\ell}^\tau_{S^n \times D^n}$.
\end{definition}
\begin{definition}
  Let $W_{1,1} = S^n \times S^n \setminus \mathrm{int}(D^{2n})$ be
  obtained by removing an open disc inside $D^n_+ \times D^n_+$, the
  product of the two upper hemispheres.  A $\theta$-structure
  $\hat{\ell}$ on $W_{1,1}$ is \emph{standard} if both
  $\theta$-structures $\bar{e}^*\hat{\ell}$ and $\bar{f}^*\hat{\ell}$
  on $S^n \times D^n$ are standard, where $\bar{e}$ and $\bar{f}$ are
  the embeddings defined by
  \begin{equation*}
    \begin{aligned}
      \bar{e} : S^n \times D^n &\lra W_{1,1} \subset S^n \times S^n
      \subset \bR^{n+1} \times \bR^{n+1}\\
      (x,y) &\longmapsto \left(x; \tfrac{y}{2}, - \sqrt{1- \vert
          \tfrac{y}{2}\vert^2}\right)
    \end{aligned}
  \end{equation*}
  and
  \begin{equation*}
    \begin{aligned}
      \bar{f} : S^n \times D^n &\lra W_{1,1} \subset S^n \times S^n
      \subset \bR^{n+1} \times \bR^{n+1}\\
      (x,y) &\longmapsto \left(-\tfrac{y}{2}, - \sqrt{1- \vert
          \tfrac{y}{2}\vert^2}; x\right),
    \end{aligned}
  \end{equation*}
\end{definition}

\begin{remark}\label{rem:StdVsAdmissible}
  In Definition \ref{defn:admissible} we said that a
  $\theta$-structure $\hat{\ell}$ on $W_{1,1}$ was \emph{admissible}
  if there are orientation-preserving embeddings
  ${e},{f}: S^n \times D^n \to W_{1,1}$ with cores intersecting
  transversely in one point, such that each of the $\theta$-structures
  $e^*\hat{\ell}$ and $f^*\hat{\ell}$ on $S^n \times D^n$ extend to
  $\bR^{2n}$ for some orientation-preserving embeddings
  $S^n \times D^n \hookrightarrow \bR^{2n}$.  Up to reparametrisation,
  this condition is equivalent to the above notion of \emph{standard}:
  if $\hat{\ell}$ is admissible then there is an embedding
  $\phi : W_{1,1} \hookrightarrow W_{1,1}$ such that
  $\phi^*\hat{\ell}$ is standard, see \cite[Remark 7.3]{GR-W3}
\end{remark}

The embeddings $\bar{e}$ and $\bar{f}$ are orientation-preserving and
the union of their images is isotopy equivalent to
$W_{1,1}$. Therefore there exists a framing $\xi_{W_{1,1}}$ of
$W_{1,1}$ such that $\bar{e}^*\xi_{W_{1,1}}$ and
$\bar{f}^*\xi_{W_{1,1}}$ are each homotopic to $\xi_{S^n \times D^n}$,
so the associated $\theta$-structure $\hat{\ell}^\tau_{W_{1,1}}$ is
standard.

\begin{lemma}\label{lem:StdIsUnique}
  The space of standard $\theta$-structures on $W_{1,1}$ fixed on a
  disc in the boundary is path-connected.
\end{lemma}
\begin{proof}
  This is \cite[Lemma 7.7]{GR-W3}, under a choice of identification of
  $W_{1,1}$ with the manifold $H$ of that paper.
\end{proof}

\subsection{Stable homological stability}

In this section we shall formulate our stable homological stability
theorem, which is a strengthening of Theorem \ref{thm:Main}.

\begin{definition}\label{defn:ThetaEnd}
  We say that a composable sequence of cobordisms
  $$K\vert_0 \overset{K\vert_{[0,1]}}\leadsto K\vert_1
  \overset{K\vert_{[1,2]}}\leadsto K\vert_2
  \overset{K\vert_{[2,3]}}\leadsto K\vert_3 \leadsto  \cdots$$
  in $\mathcal{C}_{\theta, \partial L}$ is a \emph{$\theta$-end in
    $\mathcal{C}_{\theta, \partial L}$} if each cobordism
  $K\vert_{[i,i+1]}$ satisfies
  \begin{enumerate}[(i)]
  \item\label{it:end:1} it is $(n-1)$-connected relative to both
    $K\vert_i$ and $K\vert_{i+1}$, and
  \item\label{it:end:2} it contains an embedded copy of $W_{1,1}$ with
    standard $\theta$-structure, in a path component intersecting
    $\partial L$.
\end{enumerate}
For a $\theta$-end $K$ in $\mathcal{C}_{\theta, \partial L}$, let
$$\mathcal{F}_i(P) \subset \mathcal{C}_{\theta, \partial L}(P, K\vert_i)$$
be the subspace consisting of those $\theta$-cobordisms $(s, M)$ such
that $\ell_M : M \to B$ is $n$-connected.
\end{definition}

\begin{lemma}
  The map
  $\mathcal{C}_{\theta,\partial L}(Q,K\vert_i) \to
  \mathcal{C}_{\theta,\partial L}(P,K\vert_i)$ induced by precomposing
  with any $M \in \cob(P,Q)$ sends $\mathcal{F}_i(Q)$ into
  $\mathcal{F}_i(P)$.  Similarly, postcomposing with
  $K\vert_{[i,i+1]}$ induces compatible maps
  $\mathcal{F}_i(P) \to \mathcal{F}_{i+1}(P)$.
\end{lemma}
In other words, $\mathcal{F}_i$ defines a subfunctor of
$\mathcal{C}_{\theta, \partial L}(-, K\vert_i) : \cob^{op} \to
\mathrm{Top}$ and postcomposition with $K\vert_{[i,i+1]}$ defines a
natural transformation $\mathcal{F}_i \Rightarrow \mathcal{F}_{i+1}$.
\begin{proof}
  If $(s, W) \in \mathcal{F}_i(Q)$ and
  $(t, M) : P \leadsto Q \in \cob$, then we must show that the map
  $\ell_M \cup \ell_W : M \cup_Q W \to B$ is $n$-connected. As
  $(M, Q)$ is $(n-1)$-connected, so is $(M \cup_Q W, W)$, and hence we
  have a factorisation
  $$\ell_W : W \lra M \cup_Q W \overset{\ell_M \cup \ell_W}\lra B$$
  of an $n$-connected map where the first map is $(n-1)$-connected: it
  follows that the second map is $n$-connected too.

  For the second claim, given $(s, W) \in \mathcal{F}_i(P)$ we must
  show that $(1+s,K\vert_{[i,i+1]} \circ W) \in \mathcal{F}_{i+1}(P)$,
  i.e.\ that ${\ell}_{K\vert_{[i,i+1]} \circ W}$ is
  $n$-connected. This follows from the same argument as above, using
  that $(K\vert_{[i,i+1]}, K\vert_i)$ is $(n-1)$-connected.
\end{proof}

Theorem \ref{thm:Main} can now be stated economically as follows.

\begin{theorem}[Stable homological stability]\label{thm:StabStab}
  For any fixed $\theta$-end in $\mathcal{C}_{\theta, \partial L}$ as
  in Definition~\ref{defn:ThetaEnd}, write
  $\mathcal{F}(P) = \hocolim\limits_{i \to \infty} \mathcal{F}_i(P)$.
  Then the map $\mathcal{F}(Q) \to \mathcal{F}(P)$ induced by any
  $M \in \cob(P,Q)$ is an abelian homology equivalence.
\end{theorem}

\subsection{Models for moduli spaces of manifolds}

In order to explain how to deduce Theorem \ref{thm:Main} from Theorem
\ref{thm:StabStab}, we will make a particular choice of model for the
Borel constructions
$\Bun_{n,\partial}^\theta(TW;\hat{\ell}_P)\hcoker \Diff_\partial(W)$
which form path components of $\NN_n^\theta(P;\hat{\ell}_P)$. This
involves a comparison between two models for spaces of manifolds which
will be of further use in Section \ref{sec:StabHomAndGC}, so we give
it in full generality.

In \cite[Proposition 2.16]{GR-W2} we have explained how to start with
a $(2n-1)$-manifold
$L \subset \{0\} \times (-\infty,0] \times \bR^{\infty-1}$ with
collared boundary
$\partial L \subset \{0\} \times \{0\} \times \bR^{\infty-1}$, and
rotate it in $(-\infty,0] \times \bR \times \bR^{\infty-1}$ around the
origin in the first two coordinate directions to obtain a new manifold
$\overline{L} \subset \{0\} \times [0,\infty) \times
\bR^{\infty-1}$. The union $D(L) = L \cup \overline{L}$ is the double
of $L$, and the subset
$V_L \subset (-\infty,0] \times \bR \times \bR^{\infty-1}$ swept out
by the rotation gives a nullbordism $V_L : \emptyset \leadsto
D(L)$. The manifold $V_L$ is diffeomorphic to $[0,1] \times L$, after
unbending corners, and in particular the inclusion
$L \hookrightarrow V_L$ is a homotopy equivalence. Therefore a
$\theta$-structure $\hat{\ell}_L$ on $\epsilon^1 \oplus TL$ may be
extended to one, $\hat{\ell}_{V_L}$, on $V_L$. We then give
$\overline{L}$ the $\theta$-structure induced by restriction from
$V_L$.

We may equally well begin with a $(2n-1)$-manifold
$P \subset \{0\} \times [0,\infty) \times \bR^{\infty-1}$ with
collared boundary
$\partial P \subset \{0\} \times \{0\} \times \bR^{\infty-1}$ and
rotate it in $(-\infty,0] \times \bR \times \bR^{\infty-1}$. We write
$\overline{P} \subset \{0\} \times (-\infty,0] \times \bR^{\infty-1}$
for the manifold so obtained, and induce a $\theta$-structure on
$\overline{P}$ from one on $P$ in the same way.

In particular, if $P \in \mathcal{C}_{\theta, \partial L}$ is an
object then
$\overline{P} \subset \{0\} \times (-\infty,0] \times \bR^{\infty-1}$
is equal (as a $\theta$-manifold) to $L$ near their boundary, so there
are isomorphisms
$$\mathcal{C}_{\theta, L} \cong \mathcal{C}_{\theta, \partial L} =
\mathcal{C}_{\theta, \partial \overline{P}} \cong \mathcal{C}_{\theta,
  \overline{P}}.$$

A similar procedure associates to a $2n$-manifold
$N \subset [0,s] \times [0,\infty) \times \bR^{\infty-1}$ with
collared boundary
$(\{0\} \times P) \cup ([0,s] \times \partial L) \cup (\{s\} \times
P)$ a rotated submanifold
$\overline{N} \subset [0,s] \times (-\infty,0] \times \bR^{\infty-1}$
together with a rotated $\theta$-structure
$\hat{\ell}_{\overline{N}}$.

\begin{definition}
  For objects $P$ and $Q$ of $\mathcal{C}_{\theta, \partial L}$, let
  $\langle P, Q \rangle = \overline{P} \cup Q \in \mathcal{C}_{\theta,
    \overline{P}}$. For a morphism $(s,M) : Q \leadsto Q'$ let
  $\langle P, M \rangle = ([0,s] \times \overline{P}) \cup M : \langle
  P, Q \rangle \leadsto \langle P, Q' \rangle$, and for a morphism
  $(s,N) : P' \leadsto P$ let
  $\langle N, Q \rangle = \overline{N} \cup ([0,s] \times Q) : \langle
  P, Q \rangle \leadsto \langle P', Q \rangle$.  Both are morphisms in
  $\mathcal{C}_\theta$.
\end{definition}

\begin{lemma}\label{lem:ChangeOfModel}
  For any $P, Q \in \mathcal{C}_{\theta, \partial L}$ the composition
  \begin{equation}\label{eq:ChangeOfModel}
    \mathcal{C}_{\theta, \partial L}(P, Q) =
    \mathcal{C}_{\theta, \partial \overline{P}}(P, Q)  \hookrightarrow
    \mathcal{C}_{\theta}(D(P), \langle P, Q \rangle) \xrightarrow{-
      \circ V_{\overline{P}}} \mathcal{C}_\theta(\emptyset, \langle P,
    Q \rangle)
  \end{equation}
  is a weak homotopy equivalence. 

  If $M : Q \leadsto Q'$ and $N : P' \leadsto P$ are morphisms in
  $\mathcal{C}_{\theta, \partial L}$ then the squares
  \begin{equation*}
    \xymatrix{
      {\mathcal{C}_{\theta, \partial L}(P, Q)} \ar[r]^-{M \circ -} \ar[d]^-\simeq& {\mathcal{C}_{\theta, \partial L}(P, Q')} \ar[d]^-\simeq & {\mathcal{C}_{\theta, \partial L}(P, Q)} \ar[r]^-{- \circ N} \ar[d]^-\simeq & {\mathcal{C}_{\theta, \partial L}(P', Q)} \ar[d]^-\simeq\\
      {\mathcal{C}_{\theta}(\emptyset, \langle P, Q\rangle)} \ar[r]^-{\langle P, M\rangle \circ -}& {\mathcal{C}_{\theta}(\emptyset, \langle P, Q'\rangle)} & {\mathcal{C}_{\theta}(\emptyset, \langle P, Q\rangle)} \ar[r]^-{\langle N, Q\rangle \circ -}& {\mathcal{C}_{\theta}(\emptyset, \langle P', Q\rangle)}
    }
  \end{equation*}
  commute up to homotopy.
\end{lemma}

\begin{proof}[Proof sketch]
  Up to smoothing corners, \eqref{eq:ChangeOfModel} is given by gluing
  on an invertible cobordism. The squares each commute because the two
  compositions are given by gluing on cobordisms which are
  diffeomorphic, by a diffeomorphism preserving $\theta$-structures up
  to homotopy. For more details, see the proof of \cite[Proposition
  7.5]{GR-W2}.
\end{proof}

For an object $P \in \mathcal{C}_\theta$ we have a homeomorphism
$$\mathcal{C}_\theta(\emptyset, P) \cong \coprod_{[W]} \left((0,\infty) \times \Emb_\partial(W, (-1,0] \times \bR^\infty) \times \Bun^\theta_{\partial}(TW;\hat{\ell}_{P})\right)/\Diff_\partial(W),$$
where the disjoint union is taken over manifolds $W$ with boundary
$P$, one in each diffeomorphism class. As the action of
$\Diff_\partial(W)$ on $\Emb_\partial(W, (-1,0] \times \bR^\infty)$ is
free and has slices \cite{MR613004}, and $(0,\infty)$ is contractible,
this quotient is a model for the Borel construction so there is a weak
equivalence
$$\mathcal{C}_\theta(\emptyset, P) \simeq \coprod_{[W]} \Bun^\theta_{\partial}(TW;\hat{\ell}_{P})\hcoker\Diff_\partial(W),$$
where the disjoint union is taken over manifolds $W$ with boundary
$P$, one in each diffeomorphism class.

The spaces $\mathcal{N}_n^\theta(P)$ and
$\mathcal{M}^\theta_n(W,\hat{\ell}_W)$ in the following definition are
point-set models for the spaces $\NN_n^\theta(P)$ and
$\MM^\theta_n(W,\hat{\ell}_W)$ in the introduction.

\begin{definition}\label{defn:NN}\mbox{}
  \begin{enumerate}[(i)]
  \item For $P \in \mathcal{C}_\theta$ let $\mathcal{N}_n^\theta(P) \subset \mathcal{C}_{\theta}(\emptyset, P)$ denote the subspace of those nullbordisms $(s, W)$ of $P$ such that $\ell_W : W \to B$ is $n$-connected.
  \item For $(W, \hat{\ell}_W) \in \mathcal{N}_n^\theta(P)$, let $\mathcal{M}^\theta_n(W,\hat{\ell}_W)$ denote the path component of $\mathcal{N}_n^\theta(P)$ containing $(W,\hat{\ell}_W)$.
  \item For a manifold $W$ with boundary $P$, let $\mathcal{M}^\theta(W;\hat{\ell}_P) \subset \mathcal{C}_{\theta}(\emptyset, P)$ be the subspace of those $(X, \hat{\ell}_X)$ such that $X$ is diffeomorphic to $W$ relative to $P$.
  \end{enumerate}
\end{definition}

\begin{proof}[Proof of Theorem~\ref{thm:Main}, using Theorem~\ref{thm:StabStab}]
  We may embed $M$ in $[0,1] \times \bR^\infty$ as a cobordism
  $Q \leadsto P$, and after changing this embedding by an isotopy we
  may suppose that $M$ intersects
  $[0,1] \times [0,\infty) \times \bR^{\infty-1}$ precisely in the
  image of $p: [0,1] \times D^{2n-1} \hookrightarrow M$ and that
  $[0,1] \times D^{2n-1} \overset{p}\to M \subset [0,1] \times
  \bR^\infty$ is given by $(t, x) \mapsto (t, e(x))$ for an embedding
  $e : D^{2n-1} \to [0,\infty) \times \bR^{\infty-1}$.

  Let $L_Q = Q \cap ((-\infty,0] \times \bR^{\infty-1})$,
  $L_P = P \cap ((-\infty,0] \times \bR^{\infty-1})$, and
  $D = P \cap ([0,\infty) \times \bR^{\infty-1})$. Then $L_P$, $L_Q$,
  and $D$ all have equal boundaries, which we call $\partial L$. By
  rotating these and the manifold
  $N = M \cap ([0,1] \times (-\infty,0] \times \bR^{\infty-1})$ in the
  first two coordinate directions, we obtain a cobordism
  $$\overline{N} : \overline{L_P} \leadsto \overline{L_Q} \in \mathcal{C}_{\theta, \partial L}.$$
  Now let
  $$D=K\vert_0 \overset{K\vert_{[0,1]}}\leadsto K\vert_1 \overset{K\vert_{[1,2]}}\leadsto K\vert_2 \leadsto \cdots$$
  be a $\theta$-end in the category
  $\mathcal{C}_{\theta, \partial L}$, where each $K\vert_{[i,i+1]}$ is
  obtained from $[0,1] \times K\vert_i$ by the boundary connect-sum at
  $\{1\} \times K\vert_i$ with $W_{1,1}$ having a standard
  $\theta$-structure. There is then a commutative diagram
  \begin{equation}\label{eq:Ladder1}
    \begin{gathered}
      \xymatrix{
        {\mathcal{C}_{\theta, \partial L}(\overline{L_Q}, K\vert_0)} \ar[r]^-{K\vert_{[0,1]} \circ -} \ar[d]^-{- \circ \overline{N}} & {\mathcal{C}_{\theta, \partial L}(\overline{L_Q}, K\vert_1)}  \ar[r]^-{K\vert_{[1,2]} \circ -} \ar[d]^-{- \circ \overline{N}} & {\mathcal{C}_{\theta, \partial L}(\overline{L_Q}, K\vert_2)} \ar[d]^-{- \circ \overline{N}} \cdots\\
        {\mathcal{C}_{\theta, \partial L}(\overline{L_P}, K\vert_0)} \ar[r]^-{K\vert_{[0,1]} \circ -} & {\mathcal{C}_{\theta, \partial L}(\overline{L_P}, K\vert_1)} \ar[r]^-{K\vert_{[1,2]} \circ -} & {\mathcal{C}_{\theta, \partial L}(\overline{L_P}, K\vert_2)} \cdots
      } 
    \end{gathered}
  \end{equation}
  which Lemma \ref{lem:ChangeOfModel} shows is equivalent to the
  homotopy commutative diagram
  \begin{equation}\label{eq:Ladder2}
    \begin{gathered}
      \xymatrix{ {\mathcal{C}_{\theta}(\emptyset, Q)} \ar[r]^-{{_Q} H
          \circ -} \ar[d]^-{M \circ -} &
        {\mathcal{C}_{\theta}(\emptyset,Q')} \ar[d]^-{M' \circ -}
        \ar[r]^-{{_{Q'}} H \circ -} & {\mathcal{C}_{\theta}(\emptyset,
          Q'')} \ar[d]^-{M'' \circ -}
        \cdots\\
        {\mathcal{C}_{\theta}(\emptyset, P)} \ar[r]^-{{_P} H \circ -}
        & {\mathcal{C}_{\theta}(\emptyset, P')} \ar[r]^-{{_{P'}} H
          \circ -} & {\mathcal{C}_{\theta}(\emptyset, P'')}\cdots }
    \end{gathered}
  \end{equation}
  used to form the map of homotopy colimits in the statement of
  Theorem \ref{thm:Main}. Now Theorem \ref{thm:StabStab} implies that,
  after passing to the subdiagram of \eqref{eq:Ladder1} consisting of
  those components represented by cobordisms $(s,W)$ such that
  $\ell_W : W \to B$ is $n$-connected, the map on homotopy colimits is
  an abelian homology equivalence.

  The corresponding subdiagram of \eqref{eq:Ladder2} consists of the
  maps
  $(M,\hat{\ell}_M^{(g)}) \circ - :
  \mathcal{N}^\theta_n(Q,\hat{\ell}_Q^{(g)}) \to
  \mathcal{N}^\theta_n(P,\hat{\ell}_P^{(g)})$, and as a diagram in the
  homotopy category it is isomorphic to the diagram made out of the
  squares~\eqref{eq:IntroStabMap}.  In principle the homotopies
  implied in the diagram~\eqref{eq:Ladder2} could be different from
  the ones described informally in the introduction and since
  different choices of homotopies can lead to non-homotopic maps of
  homotopy colimits, we may not immediately conclude that the
  stabilised maps in Theorems~\ref{thm:Main} and \ref{thm:StabStab}
  are isomorphic as arrows in the homotopy category.  However, as we
  shall explain in Lemma~\ref{lem:LadderHtpies}, if the induced map of
  homotopy colimits is an abelian homology equivalence for one choice
  of homotopies then it is so for all choices.
\end{proof}

\subsection{Elementary simplifications of Theorem \ref{thm:StabStab}}\label{sec:Simplifications}

There are several easy simplifications which can be made before
embarking on the proof of Theorem \ref{thm:StabStab}. In order to
phrase these simplifications, it is convenient to introduce the
following notation: let $\mathcal{W} \subset \cob$ be the subcategory
of those morphisms $M : P \leadsto Q$ in $\mathcal{D}$ such that the
induced map $\mathcal{F}(Q) \to \mathcal{F}(P)$ is an abelian homology
equivalence. Theorem \ref{thm:StabStab} is then equivalent to the
assertion that $\mathcal{W}=\cob$, but this has the advantage that
intermediate results can be stated: we shall show that $\mathcal{W}$
contains larger and larger classes of morphisms, until it is clear
that it coincides with the entire category $\cob$.

The first simplification is a purely formal saturation property of
abelian homology equivalences.

\begin{lemma}\label{lem:W2outof3}
  The subcategory $\mathcal{W} \subset \cob$ has the 2-out-of-3 property.
\end{lemma}
\begin{proof}
  This will follow from the fact that for maps of spaces $f : X \to Y$
  and $g : Y \to Z$, if any two of $f$, $g$, and $g \circ f$ are
  abelian homology equivalences, so is the third.

  Ordinary homology equivalences satisfy 2-out-of-3, so if any two of
  $f$, $g$, and $g \circ f$ are abelian homology equivalences then all
  three spaces have the same first homology, and so have the same
  collection of abelian local coefficient systems available. In
  particular, we may assume that all abelian local coefficient systems
  are pulled back from $Z$. The claim then follows by functoriality of
  homology with local coefficients.
\end{proof}

\begin{lemma}\label{lem:HandleStr}
  If $\mathcal{W}$ contains every morphism $M : P \leadsto Q$ in
  $\cob$ whose underlying smooth cobordism has a handle structure
  relative to $Q$ consisting of a single handle of index $k$ where
  $n \leq k \leq 2n$, then $\mathcal{W}=\cob$.
\end{lemma}
\begin{proof}
  If $2n \neq 4$ then this is easy: a cobordism $M : P \leadsto Q$ in
  $\cob$ is $(n-1)$-connected relative to $Q$ by definition, so admits
  a handle structure relative to $Q$ which only has handles of index
  at least $n$. For $2n = 2$ this is clear, by cancelling any handles
  of index 0, and for $2n \geq 6$ this may be seen by handle-trading,
  as in the proof of the $s$-cobordism theorem (see e.g.\
  \cite{KervaireSCobordism}). Cutting $M$ into the corresponding
  elementary cobordisms, each of which lie in $\mathcal{W}$ by
  assumption, it follows that $M$ does too.

  If $2n=4$ such handle-trading is not immediately available, but
  becomes available after connect-sum with many copies of
  $S^2 \times S^2$. More precisely, let us write
  $H_0: P_{-1} \leadsto P_0 = P$ for any morphism with underlying
  smooth bordism $([-1,0] \times P)\sharp (S^n \times S^n)$ and
  $\theta$ structure extending $\hat\ell_P$ on $\{0\} \times P$, and
  repeat to get composable morphisms
  $H_{-k}: P_{-k-1} \leadsto P_{-k}$ with underlying smooth bordism
  $([-k-1,-k] \times P) \sharp (S^n \times S^n)$ for all $k \geq 0$.
  Then it follows from \cite[Theorem 1.2]{Quinn4Mfld} that for
  sufficiently large $k$, the cobordism
  $$ M' : P_{-k} \overset{H_{-(k-1)}}\leadsto \cdots \overset{H_{-2}}\leadsto P_{-2} \overset{H_{-1}}\leadsto P_{-1} \overset{H_0}\leadsto P_0 = P \overset{M}\leadsto Q$$
  admits a handle structure relative to $Q$ which only has relative
  handles of index at least $2$, and so by assumption lies in
  $\mathcal{W}$. The cobordism $H_{0} \circ \cdots \circ H_{-(k-1)}$
  also admits a handle structure relative to $P$ which only has
  relative handles of index at least $2$, so also lies in
  $\mathcal{W}$; by Lemma \ref{lem:W2outof3}, it follows that $M$ also
  lies in $\mathcal{W}$.
\end{proof}

The manifold $\partial L$ has finitely many path-components, and each
$K\vert_{[i,i+1]}$ contains an embedded $W_{1,1}$ with standard
$\theta$-structure in a path component intersecting $\partial
L$. Therefore there is a path component
$\partial_0 L \subset \partial L$ such that infinitely-many
$K\vert_{[i,i+1]}$ contain an embedded $W_{1,1}$ with standard
$\theta$-structure in the path component of $\partial_0 L$. We call
the path-component of $\partial_0 L$ the \emph{basepoint
  component}. By composing some of the cobordisms $K\vert_{[i,i+1]}$,
and rescaling, we may assume that each $K\vert_{[i,i+1]}$ contains an
embedded $W_{1,1}$ with standard $\theta$-structure in the basepoint
component.

For a cobordism $M : P \leadsto Q$ which has a handle structure
relative to $Q$ consisting of a single handle of index $k \geq n$, it
makes sense to ask whether the handle is attached to the basepoint
component i.e.\ whether the attaching map has image which intersects
the basepoint component. When $k>1$ the image of such an attaching map
is always connected, so is required to lie inside the basepoint
component; when $k=n=1$ we require that at least one of the two
components lie inside the basepoint component.

\begin{lemma}
  If $\mathcal{W}$ contains every morphism $M : P \leadsto Q$ in
  $\cob$ whose underlying smooth cobordism has a handle structure
  relative to $Q$ consisting of a single handle of index
  $n \leq k < 2n$ attached to the basepoint component, then
  $\mathcal{W}=\cob$.
\end{lemma}
\begin{proof}
  Let $M : P \leadsto Q \in \cob$ be a morphism having a handle
  structure relative to $Q$ consisting of a single handle of index
  $n \leq k \leq 2n$. If we can show that $M \in \mathcal{W}$ then the
  previous lemma applies and gives the desired conclusion.

  First suppose that $k=2n$: then $M$ is obtained from $Q$ by
  attaching a single $2n$-handle to $Q$, along an entire component
  $\phi : S^{2n-1} \hookrightarrow Q$. Choose surgery data
  $\varphi : S^0 \times D^{2n-1} \hookrightarrow Q$ sending one disc
  to the basepoint component and the other into the image of $\phi$,
  and let $U : Q \leadsto R$ be the trace of the surgery along
  $\varphi$. Then the composition $U \circ M$ consists of a
  $2n$-handle and a $(2n-1)$-handle relative to $R$, and moreover
  these are cancelling handles by construction: thus $U \circ M$ is a
  cylinder and so is invertible in $\cob$ and hence lies in
  $\mathcal{W}$. On the other hand $U$ consists of a single
  $(2n-1)$-handle relative to $R$ attached to the basepoint component,
  so $U \in \mathcal{W}$ by hypothesis: thus $M \in \mathcal{W}$ by
  Lemma \ref{lem:W2outof3}.

  Now suppose that $n \leq k < 2n$. We will construct morphisms
  \begin{equation*}
    \xymatrix{
      P \ar@{~>}[r]^-M \ar@{~>}[d]_-{T'}& Q \ar@{~>}[d]^-{T}\\
      R \ar@{~>}[r]^-{M'} & S
    }
  \end{equation*}
  in $\cob$ such that $T \circ M$ and $M' \circ T'$ lie in the same
  path component of $\cob(P,S)$, and $T, M', T' \in \mathcal{W}$. It
  will then follow from Lemma \ref{lem:W2outof3} that
  $M \in \mathcal{W}$, as required.

  Suppose that $M$ consists of a single handle attached via a map
  $\phi : \partial D^k \times D^{2n-k} \hookrightarrow Q$, which does
  not land in the basepoint component. Choose surgery data
  $\varphi : S^0 \times D^{2n-1} \hookrightarrow Q$, disjoint from $L$
  and from $\phi$, such that surgery along it connects one of the path
  components intersecting the image of $\phi$ to the basepoint
  component, and let $T : Q \leadsto S$ be the trace of the surgery
  along $\varphi$. Considered relative to $S$, it consists of a single
  $(2n-1)$-handle attached to the basepoint component of $S$, so
  $T \in \mathcal{W}$.

  The cobordism $T \circ M : P \leadsto S$ has a $k$-handle and a
  $(2n-1)$-handle relative to $S$, and these are attached along
  disjoint embeddings. Thus it is isotopic to a composition
  $M' \circ T' : P \leadsto S$ where $M' : R \leadsto S$ has a
  $k$-handle relative to $S$ and $T' : P \leadsto R$ has a
  $(2n-1)$-handle relative to $R$. The $k$-handle of $M'$ is attached
  along $\phi$, but as $S$ is the result of surgery on $Q$ along
  $\varphi$ the map $\phi$ now lands in the basepoint component, and
  so $M' \in \mathcal{W}$. The $(2n-1)$-handle of $T'$ is attached to
  the basepoint component of $R$, because by construction it consists
  of a 1-handle relative to $P$ with one end attached in the basepoint
  component of $P$, so $T' \in \mathcal{W}$. This provides the
  required data for the argument given above.
\end{proof}

Thus to prove Theorem \ref{thm:StabStab} it suffices to show that
every cobordism $M : P \leadsto Q$ in $\cob$ which has a handle
structure relative to $Q$ consisting of a single handle of index
$n \leq k < 2n$ attached to the basepoint component induces an abelian
homology equivalence. Our proof occupies the following four sections,
and proceeds by induction on $k$.

%%% Local Variables: 
%%% mode: latex
%%% TeX-master: "stability2"
%%% End: 
 % Definitions
\section{Proof of Theorem \ref{thm:StabStab}: stability for
  $W_{1,1}$}\label{sec:LeftRight} 

For any $P \in \cob$, let us write
$$H_{P} : {^\prime}P \leadsto P$$
for any morphism in $\cob$ obtained from $[0,1] \times P$ by forming
the boundary connect-sum with $W_{1,1}$ at $\{0\} \times P$ inside the
basepoint component of $P$ (that is, the path component of
$\partial_0 L$). Similarly, let
$${_P H} : P \leadsto P^\prime$$
be any morphism in $\cob$ obtained from $[0,1] \times P$ by forming
the boundary connect-sum with $W_{1,1}$ at $\{1\} \times P$ inside the
basepoint component of $P$.  Both ${_P H}$ and $H_P$ are equipped with
any $\theta$-structure which is standard when restricted to the
embedded $W_{1,1}$, and is equal to $\hat{\ell}_P$ when restricted to
$P$.  We shall establish the following case of Theorem
\ref{thm:StabStab} for any choice of such morphisms $H_P$ and
${_P H}$.

\begin{theorem}\label{thm:SnSnStability}
  For each $P \in \cob$, the morphism $H_{P}$ lies in $\mathcal{W}$.
\end{theorem}

\begin{proof}
  Recall that the functor $\mathcal{F}$ is defined as the objectwise
  homotopy colimit of the $\mathcal{F}_i$, so the map
  $- \circ H_{P} : \mathcal{F}(P) \to \mathcal{F}(P^\prime)$ which we
  must show is an abelian homology equivalence is the induced map on
  horizontal homotopy colimits of the commutative diagram
  \begin{equation*}
    \begin{gathered}
      \xymatrix{
        {\mathcal{F}_0(P)} \ar[d]^{- \circ H_{P}}\ar[rr]^-{K\vert_{[0,1]} \circ -} & & {\mathcal{F}_1(P)} \ar[d]^{- \circ H_{P}}\ar[rr]^-{K\vert_{[1,2]} \circ -} && {\mathcal{F}_2(P)} \ar[d]^{- \circ H_{P} }\ar[rr]^-{K\vert_{[2,3]} \circ -} & & \cdots\\ 
        {\mathcal{F}_0(P^\prime)} \ar[rr]_-{K\vert_{[0,1]} \circ -}&& {\mathcal{F}_1(P^\prime)} \ar[rr]_-{K\vert_{[1,2]} \circ -}&& {\mathcal{F}_2(P^\prime)} \ar[rr]_-{K\vert_{[2,3]} \circ -} & & \cdots.
      }
    \end{gathered}
  \end{equation*}

  We shall show that for each square in this diagram there are maps
  $$\Delta_i^\mathrm{top}, \Delta_i^\mathrm{bottom} : \mathcal{F}_i(P^\prime) \lra \mathcal{F}_{i+1}(P)$$
  which respectively make the top and bottom triangles of the squares
  in which they lie commute up to homotopy; it will then follow from
  Proposition \ref{acyclicity-key-proposition} that the induced maps
  on horizontal homotopy colimits are abelian homology
  equivalences. In order to do this, we will pass to a homotopy
  equivalent model of this diagram. For the reader familiar with
  \cite{GR-W2}, this is completely analogous to Lemma 7.15 of that
  paper.

  Lemma \ref{lem:ChangeOfModel} shows that
  $\mathcal{F}_i(P) \simeq \mathcal{N}^\theta_n(\langle P,K\vert_i
  \rangle) \subset \mathcal{C}_\theta(\emptyset, \langle P,K\vert_i
  \rangle)$, the subspace consisting of those nullbordisms
  $(s,W) : \emptyset \leadsto \langle P,K\vert_i \rangle$ for which
  $\ell_W : W \to B$ is $n$-connected. Similarly for $P^\prime$ and
  the morphism $K\vert_{[i,i+1]}$. Hence the square in which we are
  trying to find diagonal maps (up to homotopy) may be replaced by
  \begin{equation*}
    \xymatrix{
      {\mathcal{N}^\theta_n(\langle P,K\vert_i \rangle)}
      \ar[rrr]^-{\langle P, K\vert_{[i,i+1]} \rangle \circ -}
      \ar[d]_-{\langle H_P, K\vert_i \rangle \circ -} & & & {\mathcal{N}^\theta_n(\langle P,K\vert_{i+1} \rangle)} \ar[d]^-{\langle H_P, K\vert_{i+1} \rangle \circ -}\\
      {\mathcal{N}^\theta_n(\langle P^\prime,K\vert_i \rangle)} \ar[rrr]^-{\langle P^\prime, K\vert_{[i,i+1]} \rangle \circ -} & & & {\mathcal{N}^\theta_n(\langle P^\prime,K\vert_{i+1} \rangle)}.
    }
  \end{equation*}

  We claim that $X=\langle H_P, K\vert_i \rangle$ may be
  $\theta$-embedded into $Y=\langle P, K\vert_{[i,i+1]} \rangle$
  relative to $\langle P, K\vert_i \rangle$. If this is the case, the
  complement of such an embedding gives a cobordism
  $Z : \langle P^\prime, K\vert_i \rangle \leadsto \langle P,
  K\vert_{i+1} \rangle$ which is $(n-1)$-connected relative to
  $\langle P, K\vert_{i+1} \rangle$, and so $Z \circ -$ defines a
  diagonal map in the square making the top triangle commute up to
  homotopy, as required.

  The cobordism
  $X : \langle P, K\vert_i \rangle \leadsto \langle P^\prime, K\vert_i
  \rangle$ is---by definition of $H_P$---obtained from
  $\langle P, K\vert_i \rangle$ by forming the boundary connect-sum
  with $W_{1,1}$ with a standard $\theta$-structure inside the
  basepoint component of $\{1\} \times \langle P, K\vert_i
  \rangle$. On the other hand, the cobordism
  $Y : \langle P, K\vert_i \rangle \leadsto \langle P, K\vert_{i+1}
  \rangle$ contains an embedded $W_{1,1}$ with standard
  $\theta$-structure, in the basepoint component (as we re-indexed the
  sequence of cobordisms $K\vert_{[i,i+1]}$ in order to have this
  property). There is therefore an embedding $e$ of underlying
  manifolds from $X$ to $Y$ sending the $W_{1,1}$ in $X$ to that in
  $Y$. By Lemma \ref{lem:StdIsUnique} the $\theta$-structure
  $e^*\hat{\ell}_Y$ is homotopic to $\hat{\ell}_X$ relative to
  $\langle P, K\vert_i \rangle$, and extending this homotopy to
  $\hat{\ell}_Y$ it follows that $e$ is an embedding of
  $\theta$-manifolds.

  Similarly, $\langle H_P, [0,1] \times K\vert_{i+1} \rangle$ may be
  $\theta$-embedded into
  $\langle [0,1] \times P', K\vert_{[i,i+1]} \rangle$ relative to
  $\langle P', K\vert_{i+1} \rangle$, which provides a diagonal map
  making the bottom triangle commute up to homotopy.
\end{proof}

%%% Local Variables: 
%%% mode: latex
%%% TeX-master: "stability2"
%%% End: 
 % Stability for W_{1,1}
\section{Proof of Theorem \ref{thm:StabStab}: stability for handles of
  index $n$}\label{sec:stable-stab}

In this section and the next we shall prove Theorem \ref{thm:StabStab}
in the case of a cobordism which admits a handle structure relative to
its incoming boundary consisting of a single $n$-handle attached to
the basepoint component. In terms of the subcategory
$\mathcal{W} \subset \cob$ from Section~\ref{sec:Simplifications}, the
precise statement we prove is the following.

\begin{theorem}\label{thm:StabNHandle}
  If $M : P \leadsto Q$ is a morphism in $\cob$ whose underlying
  smooth cobordism admits a handle structure relative to $Q$
  consisting of a single $n$-handle attached to the basepoint
  component of $Q$, then $M \in \mathcal{W}$.
\end{theorem}

The detailed proof is cumbersome, but the strategy can be explained
informally as follows.  It suffices to prove that the composition of
$M$ and the morphism $r(M)$, which attaches the $n$-handle ``dual'' to
that of $M$, is in $\mathcal{W}$, since this then applies to $r(M)$,
and we may deduce that all three morphisms in the composition
$(M \circ r(M)) \circ r(r(M)) = M \circ (r(M) \circ r(r(M))$ induce
isomorphisms on $H_1(-;\bZ)$ and homology with abelian coefficients.
If the handle of $M$ is attached trivially, the composition of
$M \circ r(M)$ is essentially a model for the cobordism $H_P$
considered in the previous section.  We shall then use a simplicial
resolution to reduce the general case to the trivially attached case.

\subsection{Support and interchange of
  support}\label{sec:SupportAndInterchange}

We first introduce the following notion of support of a cobordism,
which will be used not only in this section but also in Section
\ref{sec:stable-stab2}.

\begin{definition}
  The \emph{support} $\supp(W)$ of a cobordism $(s,W) : P \leadsto Q$
  in $\mathcal{C}_\theta$ is the smallest closed set
  $A \subset \bR^\infty$ such that
  $$W \cap ([0,s] \times (\bR^\infty \setminus A)) = [0,s] \times (P
  \setminus A)$$ as $\theta$-manifolds. Or, equivalently, such that
  $$W \cap ([0,s] \times (\bR^\infty \setminus A)) = [0,s] \times (Q
  \setminus A)$$ as $\theta$-manifolds.
\end{definition}

This notion of support is specific to our particular model of
cobordism categories using manifolds embedded in euclidean space; it
does not have meaning for ``abstract'' cobordisms.  We used a similar
notion in \cite[p.\ 268]{GR-W2}.

\begin{definition}
  If $(s,W) : P \leadsto Q$ and $(s',W') : Q \leadsto R$ are
  cobordisms in $\mathcal{C}_\theta$ such that
  $\supp(W) \cap \supp(W')=\emptyset$, then we let the
  \emph{interchange of support} be the $\theta$-cobordisms
  $$\mathcal{R}_{W'}(W) = (W \setminus([0,s] \times
  \supp(W')))\cup([0,s] \times (R \setminus \supp(W)))$$ and
  $$\mathcal{L}_{W}(W') = (W' \setminus([0,s'] \times \supp(W)))\cup([0,s'] \times (P \setminus \supp(W'))).$$
  The composition
  $(s,\mathcal{R}_{W'}(W)) \circ (s',\mathcal{L}_{W}(W'))$ may be
  formed in $\mathcal{C}_\theta$, and is a morphism from $P$ to
  $R$. Furthermore, there is a path
  $$t \mapsto \tau_t(W,W') : [0,1] \lra \mathcal{C}_\theta(P,R)$$
  from $(s', W') \circ (s,W)$ to
  $(s,\mathcal{R}_{W'}(W)) \circ (s',\mathcal{L}_{W}(W'))$ given by
  sliding $W \cap ([0,s] \times \supp(W))$ in the positive direction
  and $W' \cap ([0,s'] \times \supp(W'))$ in the negative direction.
\end{definition}

This manoeuvre may be described graphically as in Figure
\ref{fig:Interchange2}.

\begin{figure}[h]
  \begin{center}
    \includegraphics[bb=0 0 308 96]{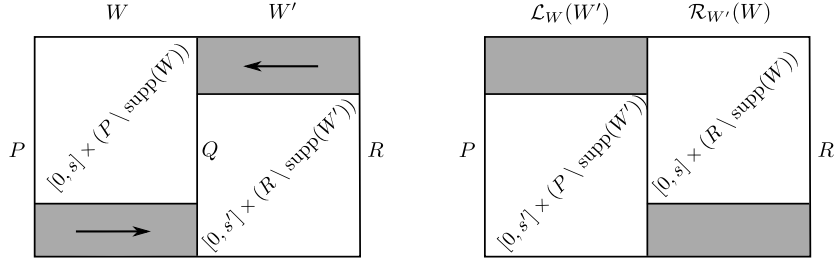}
    \caption{The interchange of support path. The grey regions indicate where the cobordisms are not cylindrical.}
    \label{fig:Interchange2}
  \end{center}
\end{figure}

\subsection{Constructing auxiliary cobordisms}\label{sec:AuxCob}

\newcommand{\phicirc}{\phi}
\newcommand{\varphicirc}{\varphi}

We begin by showing that the morphism $M : P \leadsto Q$ for which we
shall prove Theorem \ref{thm:StabNHandle} may be assumed to be of a
rather standard form. As we shall need the analogous result later,
when dealing with cobordisms having a single $k$-handle for
$n \leq k \leq 2n$, we begin by working in this generality.

\begin{construction}\label{const:Trace}
  The (reverse) trace of the surgery along
  $\partial D^k \times D^{2n-k} \subset \partial D^k \times
  \bR^{2n-k}$ gives a smooth manifold
  $T \subset [0,1] \times \bR^k \times \bR^{2n-k}$ such that
  \begin{enumerate}[(i)]
  \item $T$ agrees with $[0,1] \times \partial D^k \times \bR^{2n-k}$
    outside of $[0,1] \times D^k \times D^{2n-k}$,
  \item
    $T \cap ((1-\epsilon,1] \times \bR^{2n}) = (1-\epsilon,1]
    \times \partial D^k \times \bR^{2n-k}$ for some $\epsilon > 0$,
    and
  \item
    $T \cap ([0,\epsilon) \times \bR^{2n}) = [0,\epsilon) \times P$
    for some $\epsilon > 0$, with $P$ diffeomorphic to
    $\R^k \times \partial D^{2n-k}$,
  \end{enumerate}
  and comes equipped with embeddings of the $k$-handle and its dual
  $(2n-k)$-handle which we denote
  \begin{align*}
    \phicirc_T : D^k \times \bR^{2n-k} & \hookrightarrow T\\
    \varphicirc_T : \bR^k \times D^{2n-k} &\hookrightarrow T.
  \end{align*}
  We have
  $\phicirc_T(\partial D^{k} \times \R^{2n-k}) \subset T \cap (\{1\}
  \times \R^{2n})$ and
  $\varphicirc_T(\R^{k} \times \partial D^{2n-k}) \subset T \cap
  (\{0\} \times \R^{2n})$, and we may also arrange that
  $\phicirc_T(x,y) = (1,x,y) \in T \cap (\{1\} \times \R^{2n})$ for
  all $(x,y) \in \partial D^k \times \R^{2n-k}$.

  Suppose given a $(2n-1)$-manifold $Q \subset \bR^\infty$ and an
  embedding $\sigma: \R^k \times \R^{2n-k} \to \R^{\infty}$ such that
  $\sigma^{-1}(Q) = \partial D^k \times \bR^{2n-k}$.  Then
  $\sigma\vert_{\partial D^k \times \bR^{2n-k}}$ is a diffeomorphism
  from $\partial D^k \times \bR^{2n-k}$ onto an open subset of $Q$ and
  we may define an elementary cobordism
  $M_\sigma \subset [0,1] \times\R^\infty$ by
  \begin{equation*}
    M_\sigma = ([0,1] \times (Q \setminus \sigma(\partial D^k \times
    D^{2n-k}))) \cup ((\mathrm{Id}_{[0,1]} \times \sigma)(T)).
  \end{equation*}
  We then have
  \begin{equation*}
    \partial M_\sigma = (\{0\} \times P_\sigma) \cup (\{1\} \times Q)
  \end{equation*}
  for a closed $(2n-1)$-manifold $P_\sigma\subset \R^\infty$ obtained
  from $Q$ by surgery along $\sigma \vert_{\partial D^k \times D^k}$.
  (Even though the manifold $P_\sigma$ depends on $Q$ and $\sigma$ and
  not $P$ we shall use the notation $P_\sigma$ to emphasise that it is
  a model for $P$ in Lemma~\ref{lem:PrepCob} below.)  The embedding
  $\phicirc_T$ of a handle in $T$ then induces an embedding of a
  handle into $M$, for which we shall write
  \begin{equation*}
    \phicirc  = (\mathrm{Id}_{[0,1]} \times \sigma) \circ \phicirc_T :
    D^k \times \R^{2n-k} \to M_\sigma.
  \end{equation*}
  This embedding satisfies
  $\phicirc(\partial D^k \times \R^{2n-k}) \subset \{1\} \times Q$ and
  in fact $\phicirc(x,y) = (1,\sigma(x,y))$ for
  $(x,y) \in \partial D^k \times \R^{2n-k}$.

  If $Q$ is endowed with a $\theta$-structure $\hat{\ell}_Q$, then
  pulling back along $\sigma$ endows $\partial D^k \times \bR^{2n-k}$
  with a $\theta$-structure, and hence the subspace
  $$T_0 = (\{1\} \times \partial D^k \times \bR^{2n-k})\cup([0,1]
  \times \partial D^k \times (\bR^{2n-k} \setminus D^{2n-k})) \subset
  T$$ inherits a $\theta$-structure. The maps of pairs
  \begin{equation*}
    (D^k, \partial D^k) \times D^{2n-k}
    \xrightarrow{\phi_T\vert_{D^k}} (T, \{1\} \times \partial D^k
    \times \bR^{2n-k}) \xrightarrow{\text{incl}} (T,T_0)
  \end{equation*}
  are relative homotopy equivalences, and hence to extend the
  $\theta$-structure from $T_0$ to $T$ is the same, up to homotopy, as
  extending the bundle map
  \begin{equation}\label{eq:2}
    T(\partial D^k \times D^{2n-k})\oplus \epsilon^1 \xrightarrow{D\sigma \oplus \epsilon^1} TQ \oplus \epsilon^1 \overset{\hat{\ell}_Q}\lra \theta^*\gamma_{2n}
  \end{equation}
  to $T(D^k \times D^{2n-k})$.  In what follows we shall, whenever
  given an extension of~(\ref{eq:2}), tacitly pick an extension of
  $\theta$-structures from $T_0$ to $T$ in the corresponding homotopy
  class. Given such an extension to $T$, we obtain a
  $\theta$-structure on $M_\sigma$, and hence by restriction a
  $\theta$-structure on $P_\sigma$. The $\theta$-cobordism
  $(1,M_\sigma)$ so obtained has support in
  $\sigma(D^k \times D^{2n-k})$.
\end{construction}

\begin{lemma}\label{lem:PrepCob}
  Let $M: P \leadsto Q$ be a morphism in $\cob$ whose underlying
  smooth cobordism admits a handle structure relative to $Q$
  consisting of a single $k$-handle. Then
  \begin{enumerate}[(i)]
  \item there exists an embedding
    $\sigma: \R^k \times \R^{2n-k} \to [0,\infty) \times\R^{\infty-1}$
    such that $\sigma^{-1}(Q) = \partial D^k \times \R^{2n-k}$, and an
    extension of
    $(\sigma\vert_{\partial D^k \times D^{2n-k}})^*\hat{\ell}_Q$ to
    $T(D^k \times D^{2n-k})$, so that we may form the
    $\theta$-cobordism
    $$M_\sigma : P_\sigma \leadsto Q \in \cob,$$
  \item for a suitable choice of embedding and extension in (i), there
    exists an isomorphism $U: P \leadsto P_\sigma$ in $\cob$ such that
    $M_\sigma \circ U$ and $M$ lie in the same component of
    $\cob(P, Q)$.
  \end{enumerate}
\end{lemma}
\begin{proof}
  Suppose for simplicity that the cobordism $M$ has length 2. The
  handle structure of $M$ gives an embedding of smooth manifolds
  $M_\sigma \hookrightarrow M$ relative to $Q$, whose complement is
  diffeomorphic to $[0,1] \times P$ relative to $\{0\} \times P$, and
  we give $M_\sigma$ the $\theta$-structure from $M$.  After changing
  the embedding of $M \subset [0,2] \times \R^\infty$ relative to
  $(\{0\} \times P) \cup (\{2\} \times Q)$ we may assume that it is
  \emph{equal} to the composition of $M_\sigma$ and an isomorphism
  $U$.
\end{proof}

We now return to the particular case where $k=n$. If Theorem
\ref{thm:StabNHandle} holds for the cobordism $M_\sigma$ produced by
this lemma, then it also holds for $M$, since $U$ is an isomorphism.
Hence without loss of generality we may assume that $M$ is of the form
$M_\sigma$ with respect to some embedding
$\sigma: \R^n \times \R^n \to [0,\infty) \times\R^{\infty-1}$ with
$\sigma^{-1}(Q) = \partial D^n \times \R^n$ and some extension of
$(\sigma\vert_{\partial D^n \times D^n})^*\hat{\ell}_Q$ to
$T(D^n \times D^n)$.  We shall henceforth make this assumption, and
then also $P_\sigma = P$.  The proof of Theorem~\ref{thm:StabNHandle}
shall use certain compositions
\begin{equation*}
  R \overset{r(M)}\leadsto P \overset{M}\leadsto Q \overset{V_p} \leadsto S_p
\end{equation*}
of $M$ and cobordisms $r(M)$ and $V_p$ for $p \geq 0$, which we first
define. We shall make use of the reflection diffeomorphism
$r(t,x) = (1-t, x) : [0,1] \times \bR^\infty \to [0,1] \times
\bR^\infty$.

\begin{construction}\label{const:Mbar}
  Construct a cobordism $r(M) : R \leadsto P$ having underlying
  manifold the reflection $r(M)$ of $M=M_\sigma$, so the underlying
  manifold of $R$ is the same as that of $Q$.  It remains to describe
  a $\theta$-structure on $r(M)$. The manifold $M$ has a single
  $n$-handle relative to $P$, namely
  \begin{equation*}
    \varphicirc = (\mathrm{Id}_{[0,1]} \times \sigma) \circ
    \varphicirc_T\vert_{\R^n \times D^n} : \R^n \times D^n \lra M
  \end{equation*}
  which has
  $\varphicirc(\R^n \times \partial D^n) \subset \{0\} \times P$.  The
  reflection of $\varphicirc$ gives an $n$-handle $r(\varphicirc)$ in
  $(r(M),P)$, and any $\theta$-structure on this handle extending
  $(r(\varphicirc)\vert_{\partial D^n \times D^n})^*\hat{\ell}_P$
  induces a $\theta$-structure on $r(M)$ supported in
  $\sigma(D^n \times D^{n})$. We choose the $\theta$-structure on this
  handle by insisting that the induced $\theta$-structure on
  $$S^n \times D^n \cong r(\varphicirc)(D^n \times D^n) \cup_P {\varphicirc}(D^n \times D^n) \subset r(M) \cup_P M$$
  is diffeomorphic to the standard one described in Section
  \ref{sec:StdThetaStr}.
\end{construction}

\begin{construction}\label{const:Vp}
  We construct a cobordism $V_0 : Q \leadsto S_0$ having support in
  $\sigma(D^n \times (3e_1 + D^n))$ by letting
  $$V_0 = ([0,1] \times (Q \setminus \sigma(\partial D^n \times (3e_1 + D^n)))) \cup ((\mathrm{Id}_{[0,1]} \times \sigma)(3e_1 + r(T))),$$
  which contains the handle
  $$\psi = (\mathrm{Id}_{[0,1]} \times \sigma) \circ (3e_1 + r(\phicirc_T\vert_{D^n \times D^n})) : D^n \times D^n \lra  V_0$$
  which has $\psi(\partial D^n \times D^n) \subset \{0\} \times Q$.
  We fix the $\theta$-structure on $V_0$ by insisting that the
  $\theta$-structure on
  $$S^n \times D^n \cong \phicirc(D^n \times (3e_1 + D^n)) \cup_Q \psi(D^n \times D^n) \subset M \cup_Q V_0$$
  is the standard one described in Section \ref{sec:StdThetaStr}. 
  
  More generally, for each $p \geq 0$ we construct a cobordism
  $V_p : Q \leadsto S_p$ having support in
  $\cup_{i=0}^p\sigma(D^n \times (3(i+1)e_1 + D^n))$ by performing the
  above construction simultaneously inside each of
  $\sigma(D^n \times (3(i+1)e_1 + D^n))$ for $i=0,1,\ldots, p$. We
  write $\psi_i : D^n \times D^n \to V_p$ for the $n$-handles relative
  to $Q$, $i = 0, \dots, p$.
\end{construction}

The composable cobordisms $V_p : Q \leadsto S_p$ and
$M \circ r(M) : R \leadsto Q$ arising from these constructions have
disjoint support, so may be subjected to interchange of
support. Recall from Section \ref{sec:LeftRight} that we write
$H_{P} : P' \leadsto P$ for any morphism in $\cob$ obtained from
$[0,1] \times P$ by forming the boundary connect-sum with $W_{1,1}$ at
$\{0\} \times P$.

\begin{lemma}\label{lem:InterchangeOfSupport}
  For each $p \geq 0$, the cobordism $\mathcal{R}_{V_p}(M \circ r(M))$
  obtained by interchange of support is $H_{S_p}$ composed with an
  isomorphism.
\end{lemma}

The cobordism $M \circ r(M)$ is obtained from $Q$ by attaching an
$n$-handle and its dual, and as the support of $M \circ r(M)$ is
disjoint from that of $V_p$ the cobordism
$\mathcal{R}_{V_p}(M \circ r(M))$ is also obtained by attaching an
$n$-handle and its dual. The proof of this lemma consists of showing
that in this case the $n$-handle is trivially attached.

\begin{proof}
  In order that a cobordism $W: S \leadsto T$ be isotopic to $H_T$
  composed with an isomorphism, it is enough that it contain an
  embedded copy of $W_{1,1}$ on which the induced $\theta$-structure
  is the standard one described in Section \ref{sec:StdThetaStr}, and
  so that the smooth cobordism (without $\theta$-structure)
  $(W \setminus \mathrm{int}(W_{1,1})) \cup_{\partial W_{1,1}} D^{2n}:
  S \leadsto T$ is diffeomorphic to $[0,1] \times T$ relative to $T$.

  By construction of the $\theta$-structure on $r(M)$,
  $$S^n \times D^n \cong r(\varphicirc)(D^n \times D^n) \cup \varphicirc(D^n \times D^n) \subset r(M) \cup_P M$$
  has a standard $\theta$-structure. This is disjoint from
  $\supp(V_p)$ so can also be found inside the cobordism
  $\mathcal{R}_{V_p}(M \circ r(M))$.

  The $n$-handle $\phicirc$ of $M$ relative to $Q$ has support
  disjoint from $\supp(V_p)$, so can also be considered as an
  $n$-handle of $\mathcal{R}_{V_p}(M \circ r(M))$ relative to
  $S_p$. The manifold $S_p$ is, tautologically, the result of doing
  surgery on the embeddings
  $$\sigma\vert_{\partial D^n \times (3(i+1)e_1 + D^n)} :  \partial D^n \times D^n \hookrightarrow Q \quad \quad i=0,1,\ldots,p,$$
  and so, using only that surgery on
  $\sigma\vert_{\partial D^n \times (3e_1 + D^n)}$ has been performed,
  the embedding
  $\sigma\vert_{\partial D^n \times D^n} : \partial D^n \times D^n
  \hookrightarrow Q \setminus \supp(V_p) \subset S_p$ can be isotoped
  into a collar neighbourhood of
  $S_p \subset \mathcal{R}_{V_p}(M \circ r(M))$ and extended there to
  an embedding of $D^n \times D^n$ such that the result of joining it
  with the $n$-handle $\phicirc$ of $\mathcal{R}_{V_p}(M \circ r(M))$
  gives an embedded $S^n \times D^n$. Furthermore, by our choice of
  $\theta$-structure on $V_p$ this may be arranged to have standard
  $\theta$-structure. By construction, the core of this
  $S^n \times D^n$ intersects that of the previous paragraph
  transversely in a single point, so their union gives an embedded
  $W_{1,1}$, on which the $\theta$-structure is standard.
\end{proof}

\subsection{A semi-simplicial resolution}\label{sec:ResNHandles}
In this section we explain how to construct certain semi-simplicial
spaces $\mathcal{Y}_j(P)_\bullet$ augmented over the spaces
$\mathcal{F}_j(P)$ of cobordisms introduced in Definition
\ref{defn:ThetaEnd}, when we have fixed a $\theta$-end
$$K\vert_0 \overset{K\vert_{[0,1]}}\lra K\vert_1 \overset{K\vert_{[1,2]}}\lra K\vert_2 \overset{K\vert_{[2,3]}}\lra K\vert_3 \lra  \cdots$$
and $P$ is an object of the category $\cob$ equipped with a little extra structure. For clarity we work in slightly more generality than we will eventually need. 

\begin{definition}\label{defn:ComplexK}
  Fix an object $P \in \cob$, an element
  $W = (s,W) \in \mathcal{F}_j(P)$ for some $j \geq 0$, an embedding
  $\chi : \partial D^n \times (1,\infty) \times\bR^{n-1}
  \hookrightarrow P$, and a 1-parameter family
  $t \mapsto \hat{\ell}_t^\mathrm{std}$, $t \in (2,\infty)$, of
  $\theta$-structures on $D^n \times D^n$ such that
  $\hat{\ell}_t^\mathrm{std}\vert_{\partial D^n \times D^n} =
  \chi^*\hat{\ell}_P\vert_{\partial D^n \times (t\cdot e_1 + D^n)}$.

  Let $Y(W)_0=Y(W, \chi, \hat{\ell}_t^\mathrm{std})_0$ be the set of
  tuples $(t,c,\hat{L})$ consisting of a $t \in (2, \infty)$, an
  embedding
  $c : (D^n \times D^n, \partial D^n \times D^n) \hookrightarrow (W,
  P)$, and a path of $\theta$-structures $\tau \mapsto \hat{L}(\tau)$
  on $D^n \times D^n$, $\tau \in [0,1]$, such that
  \begin{enumerate}[(i)]
  \item there is a $\delta>0$ such that
    $c(x,v) = \chi(\tfrac{x}{\vert x \vert}, v+t \cdot e_1) + (1-\vert
    x \vert)\cdot e_0$ for $1-\vert x \vert <\delta$,
  \item the image $C = c(D^n \times D^n)$ is disjoint from
    $([0,s] \times L) \cup (\{s\} \times K\vert_j)$, and
    $c^{-1}(P) = \partial D^n \times D^n$,
  \item the restriction
    $\ell_W\vert_{W\setminus C} : W \setminus C \to B$ is
    $n$-connected,
  \item $\hat{L}(0) = c^*\hat{\ell}_W$,
    $\hat{L}(1) = \hat{\ell}_t^\mathrm{std}$, and
    $\hat{L}(\tau)\vert_{\partial D^n \times D^n}$ is independent of
    $\tau \in [0,1]$.
  \end{enumerate}
  We topologise $Y(W)_0$ as a subspace of
  $$\bR \times \Emb(D^n \times D^n, [0,\infty) \times \bR^\infty) \times \Bun^\theta(D^n \times D^n)^I.$$
  Let
  $Y(W)_p=Y(W, \chi, \hat{\ell}_t^\mathrm{std})_p \subset
  (Y(W)_0)^{p+1}$ be the subset consisting of tuples
  $(t_0, c_0, \hat{L}_0, t_1, c_1, \hat{L}_1, \ldots, t_p, c_p,
  \hat{L}_p)$ such that
  \begin{enumerate}[(i)]
  \item the images $C_i$ of the $c_i$ are disjoint,
  \item $t_0 < t_1 < \cdots < t_p$,
  \item the restriction
    $\ell_W\vert_{W \setminus (\cup_i C_i)} : W \setminus (\cup_i C_i)
    \to B$ is $n$-connected.
\end{enumerate}
We topologise $Y(W)_p$ as a subspace of the $(p+1)$-fold product of
$Y(W)_0$. The collection $Y(W)_\bullet$ has the structure of a
semi-simplicial space, where the $i$th face map forgets
$(t_i, c_i, \hat{L}_i)$.
\end{definition}

We now wish to combine all of the $Y(W)_\bullet$ for all
$W = (s,W) \in \mathcal{F}_j(P)$ into a single semi-simplicial space.

\begin{definition}
  Fix a $P \in \cob$, an embedding
  $\chi : \partial D^n \times (1,\infty) \times\bR^{n-1}
  \hookrightarrow P$, and a 1-parameter family
  $t \mapsto \hat{\ell}_t^\mathrm{std}$, $t \in (2,\infty)$, of
  $\theta$-structures on $D^n \times D^n$ such that
  $\hat{\ell}_t^\mathrm{std}\vert_{\partial D^n \times D^n} =
  \chi^*\hat{\ell}_P\vert_{\partial D^n \times (t\cdot e_1 + D^n)}$.

  Let
  $\mathcal{Y}_j(P)_p = \mathcal{Y}_j(P, \chi,
  \hat{\ell}_t^\mathrm{std})_p$ be the set of tuples $(s,W;x)$ with
  $(s,W) \in \mathcal{F}_j(P)$ and
  $x \in Y(W,\chi, \hat{\ell}_t^\mathrm{std})_p$. Topologise this set
  as a subspace of
  $$\mathcal{F}_j(P) \times (\bR \times \Emb(D^n \times D^n, [0,\infty) \times \bR^\infty) \times \Bun^\theta(D^n \times D^n)^I)^{p+1}.$$
  The collection $\mathcal{Y}_j(P)_\bullet$ has the structure of a
  semi-simplicial space augmented over $\mathcal{F}_j(P)$, where the
  $i$th face maps forgets $(t_i, c_i, \hat{L}_i)$, and the
  augmentation map just remembers the underlying $\theta$-cobordism
  $(s,W)$.
\end{definition}

The maps
$K\vert_{[j,j+1]} \circ - : \mathcal{F}_j(P) \to \mathcal{F}_{j+1}(P)$
given by the $\theta$-end lift to semi-simplicial maps
\begin{align*}
  (K\vert_{[j,j+1]} \circ -)_\bullet : \mathcal{Y}_j(P)_\bullet &\lra \mathcal{Y}_{j+1}(P)_\bullet\\
  (s,W;x) &\longmapsto (s+1;K\vert_{[j,j+1]} \circ W;x),
\end{align*}
and we let $\mathcal{Y}(P)_\bullet \to \mathcal{F}(P)$ be the
augmented semi-simplicial space obtained as the levelwise homotopy
colimit.

\begin{lemma}\label{lem:qfib}
  The map $\vert \mathcal{Y}_j(P)_\bullet\vert \to \mathcal{F}_j(P)$
  is a quasifibration, with fibre $\vert Y(W)_\bullet\vert$ over
  $(s,W) \in \mathcal{F}_j(P)$.
\end{lemma}
\begin{proof}
  This may be proved in the same way as the analogue of Lemma 6.8 in
  \cite[Section 7]{GR-W3}.
\end{proof}

The homotopy fibre of a map between mapping telescopes is weakly
homotopy equivalent to the mapping telescope of the homotopy
fibres. Hence the homotopy fibre of
$\vert \mathcal{Y}(P)_\bullet\vert \to \mathcal{F}(P)$ over a point
$(s,W) \in \mathcal{F}_j(P) \subset \mathcal{F}(P)$ is weakly homotopy
equivalent to
\begin{equation}\label{eq:ContractibilityK}
  \hocolim_{g \to \infty} \vert Y(K\vert_{[j,j+g]} \circ W, \chi, \hat{\ell}_t^\mathrm{std})_\bullet \vert.
\end{equation}
In Section \ref{sec:pf-arc-cx} we will prove the following theorem.

\begin{theorem}\label{thm:ContractibilityK}
  The space \eqref{eq:ContractibilityK} is weakly contractible, and
  hence the forgetful map
  $|\mathcal{Y}(P)_\bullet| \to \mathcal{F}(P)$ is a weak equivalence.
\end{theorem}

In the rest of this section we shall explain how to deduce
Theorem~\ref{thm:StabNHandle} from Theorem~\ref{thm:ContractibilityK}.

\subsection{Resolving composition with $M \circ r(M)$}

We now continue with the notation of Section \ref{sec:AuxCob}, in
particular we have the cobordisms $M \circ r(M) : R \leadsto Q$ and
$V_0 : Q \leadsto S_0$ and the embedding
$\sigma : \partial D^n \times \bR^n \hookrightarrow Q$, and we wish to
resolve the map
$$- \circ M \circ r(M) : \mathcal{F}(Q) \lra \mathcal{F}(R).$$
 We choose, once and for all, a 1-parameter family of embeddings
$$h_t : D^n \times D^n \hookrightarrow V_0, \quad\quad t \in (2,\infty)$$
so that $h_3$ is the map $\psi$ from Construction \ref{const:Vp}, and
$h_t\vert_{\partial D^n \times D^n}(x,v) = \sigma(x, t\cdot
e_1+v)$. We then let
$\hat{\ell}_t^\mathrm{std} = h_t^*\hat{\ell}_{V_0}$.

The $\theta$-structure on $V_0$ has been chosen (in Construction
\ref{const:Vp}) so that the induced $\theta$-structure on the sphere
$$S^n \times D^n \cong \phicirc(D^n \times (3e_1 + D^n)) \cup h_3(D^n \times D^n) \subset M \cup_Q V_0$$
is homotopic to the standard one; it follows that the same is true for
the spheres
$\phicirc(D^n \times (t e_1 + D^n)) \cup h_t(D^n \times D^n)$ for all
$t$.

The previous section then gives an augmented semi-simplicial space
\begin{equation}\label{eq:Aug1}
  \mathcal{Y}(Q)_\bullet = \mathcal{Y}(Q, \sigma\vert_{\partial D^n \times (1,\infty) \times \bR^{n-1}}, \hat{\ell}_t^\mathrm{std})_\bullet \lra \mathcal{F}(Q),
\end{equation}
and Theorem \ref{thm:ContractibilityK} shows that this augmentation
map becomes a weak homotopy equivalence after geometric realisation.

The underlying manifold of $R$ is equal to the underlying manifold of
$Q$, so the embedding $\sigma$ can also be considered as an embedding
$\sigma : \partial D^n \times \bR^n \hookrightarrow R$. The
$\theta$-structures $\sigma^*\hat{\ell}_Q$ and $\sigma^*\hat{\ell}_R$
are not necessarily equal, but they become equal when restricted to
$\partial D^n \times (1,\infty) \times \bR^{n-1}$ (as $M \circ r(M)$
is supported in $\sigma(\bR^n \times D^n)$). Hence we have defined an
augmented semi-simplicial space
\begin{equation}\label{eq:Aug2}
  \mathcal{Y}(R)_\bullet = \mathcal{Y}(R, \sigma\vert_{\partial D^n \times (1,\infty) \times \bR^{n-1}}, \hat{\ell}_t^\mathrm{std})_\bullet \lra \mathcal{F}(R).
\end{equation}
This also becomes a weak homotopy equivalence after geometric
realisation, by Theorem \ref{thm:ContractibilityK}.

We wish to cover the map
$- \circ M \circ r(M) : \mathcal{F}(Q) \to \mathcal{F}(R)$ by a map of
semi-simplicial resolutions
$\mathcal{Y}(Q)_\bullet \to \mathcal{Y}(R)_\bullet$.  The idea is that
for any embedding $c: D^n \times D^n \hookrightarrow W$ of a thickened
$n$-handle, as in Definition \ref{defn:ComplexK} but with $Q$
replacing $P$, there is an induced thickened $n$-handle embedded in
$W \circ M \circ r(M)$---provided that the boundary of the handle is
disjoint from the support of $M \circ r(M)$---defined by gluing a
cylinder $([0,1] \times \partial D^n) \times D^n$ to the handle and
embedding it into $M \circ r(M)$ in a cylindrical way.  In order to
define a map of semi-simplicial spaces in this way, we need to
reparametrise the source
\begin{equation*}
  (([0,1] \times \partial D^n) \times D^n) \cup_{\partial D^n \times
    D^n} (D^n \times D^n)
\end{equation*}
of the resulting embedding by a diffeomorphism to
$\partial D^n \times D^n$.  Of course we also need to explain what to
do with the other data ($\theta$-structures and so on) in the
simplices of $Y(W)_\bullet$.  Let us first explain the
reparametrisation. We shall have to do the same later, so the
following definition is given in greater generality than we need here:
we will only use the case $k=n$ in this section.

\begin{definition}\label{defn:Extrusion}
  Let
  $c : D^{2n-k} \times D^k \hookrightarrow [0,\infty) \times
  \bR^\infty$ be an embedding which is collared in the sense that
  there is a $\delta > 0$ such that
  $$c(x,v) = c(\tfrac{x}{\vert x \vert}, v) + (1-\vert x \vert)\cdot e_0 \quad \quad \text{ if } 1- \vert x \vert < \delta.$$
  We may then define subsets
  \begin{align*}
    A_t &= [0,t] \times c(\partial D^{2n-k} \times D^k)\\
    B_t &= c(D^{2n-k} \times D^k) + t \cdot e_0
  \end{align*}
  whose union is a smooth submanifold of
  $[0,\infty) \times \bR^\infty$. Define the \emph{extrusion} of $c$
  of length $t \in [0,\infty)$,
  $\epsilon_t(c) : D^{2n-k} \times D^k \hookrightarrow [0,\infty)
  \times \bR^\infty$, by
  $$(x, v) \longmapsto\begin{cases}
    c((1+t) \cdot x, v) + t \cdot e_0 & {\vert x \vert \leq \tfrac{1}{1+t}}\\
    c(\tfrac{x}{\vert x \vert}, v) + (1+t)(1 - \vert x \vert) \cdot e_0 & {\vert x \vert \geq \tfrac{1}{1+t}}.
  \end{cases}$$
  In particular $\epsilon_0(c) = c$.  We have
  \begin{align*}
    A_t &= \epsilon_t(c)((D^{2n-k} \setminus \tfrac{1}{1+t} \mathrm{int}(D^{2n-k})) \times D^k)\\
    B_t &= \epsilon_t(c)(\tfrac{1}{1+t} D^{2n-k} \times D^k),
  \end{align*}
  so that $A_t \cup B_t = \epsilon_t(c)(D^{2n-k} \times D^k)$.

  Let $\hat{\ell}$ be a $\theta$-structure on $D^{2n-k} \times
  D^k$. The diffeomorphism
  $$(x,v) \mapsto c(x,v) + t \cdot e_0 : D^{2n-k} \times D^k \lra B_t$$
  endows $B_t$ with a $\theta$-structure, and we may extend this
  uniquely to a $\theta$-structure on $A_t \cup B_t$ which is
  cylindrical over $A_t$. We define the \emph{extrusion} of
  $\hat{\ell}$ of length $t \in [0,\infty)$, $\epsilon_t(\hat{\ell})$,
  to be the $\theta$-structure on $D^{2n-k} \times D^k$ obtained by
  pulling back the $\theta$-structure just constructed along the
  diffeomorphism
  $\epsilon_t(c) : D^{2n-k} \times D^k \to A_t \cup B_t$.
\end{definition}

Composition with $M \circ r(M)$ gives a map
$\mathcal{F}_j(Q) \to \mathcal{F}_j(R)$, and we may lift it to a
semi-simplicial map
$(- \circ M \circ r(M))_\bullet : \mathcal{Y}_j(Q)_\bullet \to
\mathcal{Y}_j(R)_\bullet$ by the formula
$$(W; t, c, \hat{L}) \longmapsto (W \circ M \circ r(M); t, \epsilon_2(c), \epsilon_2(\hat{L}))$$
on 0-simplices, and the analogous formula on the higher
simplices. This construction commutes (strictly) with the
semi-simplicial maps
$(K\vert_{[j,j+1]} \circ -)_\bullet : \mathcal{Y}_j(Q)_\bullet \to
\mathcal{Y}_{j+1}(Q)_\bullet$ and
$(K\vert_{[j,j+1]} \circ -)_\bullet : \mathcal{Y}_j(R)_\bullet \to
\mathcal{Y}_{j+1}(R)_\bullet$, and so induces a semi-simplicial map
$(- \circ M \circ r(M))_\bullet : \mathcal{Y}(Q)_\bullet \to
\mathcal{Y}(R)_\bullet$ on the stabilisations.  The diagrams
\begin{equation*}
  \xymatrix{
    {\mathcal{Y}(Q)_p}\ar[rr]^-{- \circ M \circ r(M)}\ar[d] && {\mathcal{Y}(R)_p}\ar[d]\\
    {\mathcal{F}(Q)}\ar[rr]^-{- \circ M \circ r(M)} && {\mathcal{F}(R)}
  }
\end{equation*}
commute for all $p \geq 0$, and by Theorem~\ref{thm:ContractibilityK}
we may then view $\mathcal{Y}(Q)_\bullet \to \mathcal{Y}(P)_\bullet$
as a simplicial resolution of $\mathcal{F}(Q) \to \mathcal{F}(P)$.

\begin{proposition}\label{prop:pSxAbEq}
  For each $p \geq 0$, the map
  $$(- \circ M \circ r(M))_p : \mathcal{Y}(Q)_p \lra \mathcal{Y}(R)_p$$
  is an abelian homology equivalence.
\end{proposition}
\begin{proof}
  The idea, cf.\ diagram~\eqref{eq:ReqDiagram} below, is to identify
  this map up to homotopy with a map gluing on a cobordism $H_S$ as
  studied in Section~\ref{sec:LeftRight}.

  Recall from Construction \ref{const:Vp} that the cobordism
  $V_p : Q \leadsto S_p$ contains canonical handles
  $$\psi_i : D^n \times D^n \lra V_p \quad\quad 0 \leq i \leq p,$$
  and the $\theta$-structure on $V_p$ is determined (up to homotopy
  relative to $Q$) by its restriction to these handles. We may
  therefore suppose that the $\theta$-structure on $V_p$ is such that
  $\psi^*_i\hat{\ell}_{V_p} = \hat{\ell}^\mathrm{std}_{3(i+1)} =
  h_{3(i+1)}^*\hat{\ell}_{V_0}$ for each $i$, because the
  $\theta$-structure
  $$h_{t}^*\hat{\ell}_{V_0}\cup (\phicirc \vert_{D^n \times (te_1 + D^n)})^*\hat{\ell}_M : T(S^n \times D^n) \lra \theta^*\gamma_{2n}$$
  is standard for $t=3$ (by definition of $\hat{\ell}_{V_0}$), and
  hence for all $t$ (as being standard is homotopy invariant).

  Let $\hat{L}_i$ be the constant path
  $\hat{\ell}^\mathrm{std}_{3(i+1)}$. This determines an injection
\begin{equation}\label{eq:pSxModel}
  \begin{aligned}
    \mathcal{F}_j(S_p) & \lra \mathcal{Y}_j(Q)_p\\
    (s,X) & \longmapsto (s+1,X \circ V_p; 3, \psi_0, \hat{L}_0, 6, \psi_1, \hat{L}_1, \ldots, 3(p+1), \psi_p, \hat{L}_p).
  \end{aligned}
\end{equation}
We claim that this injection is a weak homotopy equivalence. In order
to see this, let $E$ denote the space of triples $((s,W),e,\hat{L})$
consisting of an $(s,W) \in \mathcal{F}_j(Q)$, an embedding
$e : V_p \hookrightarrow W$ relative to $Q$, and a path of
$\theta$-structures $\hat{L}$ from $e^*\hat{\ell}_W$ to
$\hat{\ell}_{V_p}$ which is constant over $Q \subset V_p$. We
topologise $E$ as a subspace of
$\mathcal{F}_j(Q) \times \Emb(V_p, [0,\infty) \times \bR^\infty;Q)
\times \Bun^\theta(TV_p)^I$. The map
\begin{align*}
  E & \lra \mathcal{Y}_j(Q)_p\\
  (s,W,e) & \longmapsto (s,W; 3, e \circ \psi_0, \psi_0^*\hat{L}, 6, e \circ \psi_1, \psi_1^*\hat{L}, \ldots, 3(p+1), e \circ \psi_p, \psi_p^*\hat{L})
\end{align*}
is then a weak homotopy equivalence, because the embedding
$$([0,\epsilon] \times Q) \cup \left (\coprod_{i=0}^p \psi_i(D^n \times D^n)\right) \lra V_p$$
is an isotopy equivalence for sufficiently small $\epsilon$. On the
other hand, the map
\begin{align*}
  E & \lra \Emb(V_p, [0,\infty) \times \bR^\infty;Q)\\
  (s,W,e) & \longmapsto e
\end{align*}
is a fibration (by the parameterised isotopy extension theorem) and
the base space is contractible (by the parameterised form of Whitney's
embedding theorem). The fibre of this map over the canonical embedding
$V_p \subset [0,1] \times \bR^\infty$ is the space of those
$N \in \mathcal{F}_j(Q)$ which contain the cobordism $V_p$, and this
is weakly homotopy equivalent to $\mathcal{F}_j(S_p)$. This proves the
claim that \eqref{eq:pSxModel} is a weak homotopy equivalence.

The cobordisms $V_p : Q \leadsto S_p$ and
$M \circ r(M) : R \leadsto Q$ are composable and have disjoint
support. We may perform interchange of support on these cobordisms,
giving composable cobordisms
$$\mathcal{L}_{M \circ r(M)}(V_p) : R \leadsto S'_p \quad\quad\quad \mathcal{R}_{V_p}(M \circ r(M)) : S'_p \leadsto S_p$$
and the embeddings $\psi_i$ have image inside
$\mathcal{L}_{M \circ r(M)}(V_p)$ and satisfy
$\psi^*_i\hat{\ell}_{\mathcal{L}_{M \circ r(M)}(V_p)} =
\hat{\ell}^\mathrm{std}_{3(i+1)}$. Thus
$\mathcal{L}_{M \circ r(M)}(V_p)$ has the same relationship to $R$ as
$V_p$ does to $Q$, and in particular the inclusion
\begin{equation*}
  \begin{aligned}
    \mathcal{F}_j(S'_p) & \lra \mathcal{Y}_j(R)_p\\
    (s,X) & \longmapsto (s+1, X \circ \mathcal{L}_{M \circ r(M)}(V_p); 3, \psi_0, \hat{L}_0, 6, \psi_1, \hat{L}_1, \ldots, 3(p+1), \psi_p, \hat{L}_p).
  \end{aligned}
\end{equation*}
is a weak homotopy equivalence for the same reason as
\eqref{eq:pSxModel} is.

Consider the diagram
\begin{equation}\label{eq:ReqDiagram}
  \begin{gathered}
    \xymatrix{
      {\mathcal{F}(S_p)} \ar[rr]^-{- \circ \mathcal{R}_{V_p}(M \circ r(M))} \ar[d]^-\simeq& &
      {\mathcal{F}(S'_p)} \ar[d]^-\simeq\\
      {\mathcal{Y}(Q)_p} \ar[rr]^-{(- \circ M \circ r(M))_p} && {\mathcal{Y}(R)_p},
    }
  \end{gathered}
\end{equation}
in which the vertical maps are weak homotopy equivalences by taking
the limit $j \to \infty$ of the weak homotopy equivalences established
above. By Lemma \ref{lem:InterchangeOfSupport} the cobordism
$\mathcal{R}_{V_p}(M \circ r(M))$ is isotopic to $H_{S_{p}} $ composed
with an isomorphism in $\cob$, and so by Theorem
\ref{thm:SnSnStability} the map
$- \circ \mathcal{R}_{V_p}(M \circ r(M))$ is an abelian homology
equivalence. Hence if the square commutes up to homotopy then we have
proved the proposition.

The two directions around the square are given by
$$X \mapsto (X \circ \mathcal{R}_{V_p}(M \circ r(M)) \circ \mathcal{L}_{M \circ r(M)}(V_p); 3, \psi_0, \hat{L}_0, 6, \psi_1, \hat{L}_1, \ldots, 3(p+1), \psi_p, \hat{L}_p)$$
for the upper composition and
$$X \mapsto (X \circ V_p \circ M \circ r(M); 3, \epsilon_2(\psi_0),
\epsilon_2(\hat{L}_0), \ldots, 3(p+1), \epsilon_2(\psi_p),
\epsilon_2(\hat{L}_p))$$ for the lower composition. The path
$t \mapsto \tau_t$ in $\cob(R,S_p)$ from $V_p \circ M \circ r(M)$ to
$\mathcal{R}_{V_p}(M \circ r(M)) \circ \mathcal{L}_{M \circ
  r(M)}(V_p)$ given by interchange of support (cf.\ Section
\ref{sec:SupportAndInterchange}) may be described as follows. For a
cobordism $(s,W) : A \leadsto B$ we write
$$\widehat{W} = ((-\infty,0] \times A) \cup W \cup ([s,\infty) \times B)$$
for its extension, then we let $\tau_t$ denote the $\theta$-manifold
$$(t\cdot e_0 + \widehat{M \circ r(M)} \setminus \supp(V_p)) \cup ((2-2t)\cdot e_0 + \widehat{V_p}\setminus \supp(M \circ r(M)))$$
intersected with $[0,3] \times \bR^\infty$, which is schematically
shown in Figure \ref{fig:Interchange}.

\begin{figure}[h]
  \begin{center}
    \includegraphics[bb=0 0 308 116]{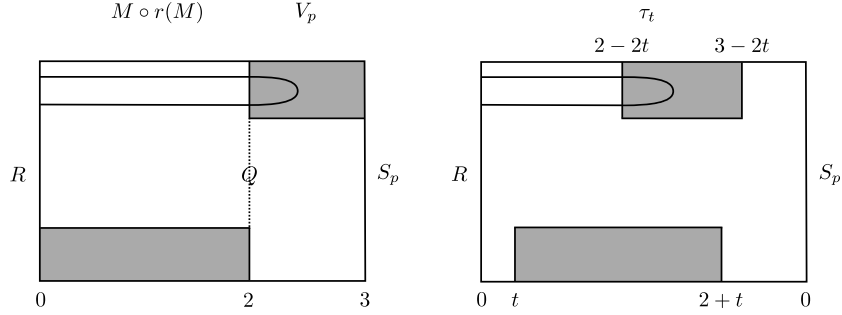}
    \caption{Interchanging the supports of $V_p$ and ${M} \circ r(M)$. The grey regions indicate where the cobordisms are not cylindrical.}
    \label{fig:Interchange}
  \end{center}
\end{figure}

The extruded embedding $\epsilon_{2t}(\psi_i)$ has image inside
$\tau_{1-t}$, by inspection of the formulae for the two
objects. Furthermore the bundle map
$\epsilon_{2t}(\psi_i)^*\hat{\ell}_{\tau_{1-t}} : T(D^n \times D^n)
\to \theta^*\gamma_{2n}$ agrees with the extruded bundle map
$\epsilon_{2t}(\hat{L}_i) =
\epsilon_{2t}(\hat{\ell}^\mathrm{std}_{3(i+1)})$, by
construction. Hence sending $X$ to
$$(X \circ \tau_{1-t}; 3, \epsilon_{2t}(\psi_1), \epsilon_{2t}(\hat{L}_1), 6, \epsilon_{2t}(\psi_2), \epsilon_{2t}(\hat{L}_2),\ldots,3(1+p), \epsilon_{2t}(\psi_p), \epsilon_{2t}(\hat{L}_p))$$
gives a homotopy in $\mathcal{Y}(R)_p$ between the two compositions in
\eqref{eq:ReqDiagram}, as required.
\end{proof} 

Now, for any system $\mathcal{L}$ of abelian local coefficients on
$\mathcal{F}(R)$ we have a map of spectral sequences
\begin{equation*}
  \xymatrix{
    E^1_{p,q}(P) = H_q(\mathcal{Y}(Q)_p;\mathcal{L}) \ar[d] \ar@{=>}[r]& H_{p+q}(\vert\mathcal{Y}(Q)_\bullet\vert;\mathcal{L}) \ar[d] \ar[r]^-\sim& H_{p+q}(\mathcal{F}(Q);\mathcal{L}) \ar[d]\\
    E^1_{p,q}(R) = H_q(\mathcal{Y}(R)_p;\mathcal{L})  \ar@{=>}[r]& H_{p+q}(\vert\mathcal{Y}(R)_\bullet\vert;\mathcal{L}) \ar[r]^-\sim & H_{p+q}(\mathcal{F}(R);\mathcal{L})
  }
\end{equation*}
where we have continued to write $\mathcal{L}$ for this system of
local coefficients pulled back along any of the maps occurring in this
set-up. By Proposition \ref{prop:pSxAbEq} the map of $E^1$-pages is an
isomorphism, as $\mathcal{L}$ is an \emph{abelian} system of local
coefficients. Thus the middle vertical map is also an isomorphism. The
two rightmost horizontal maps are isomorphisms because the
augmentation maps \eqref{eq:Aug1} and \eqref{eq:Aug2} are weak
homotopy equivalences, and hence the rightmost vertical map is
too. This shows that
\begin{equation*}
- \circ M \circ r(M) : \mathcal{F}(Q) \lra \mathcal{F}(R)
\end{equation*}
induces an isomorphism on homology with abelian coefficients, so
$M \circ r(M) \in \mathcal{W}$.

\begin{proof}[Proof of Theorem~\ref{thm:StabNHandle}, assuming
  Theorem~\ref{thm:ContractibilityK}]
  We must show that $M \in \mathcal{W}$. However, the discussion above
  holds for any cobordism which admits a handle structure relative to
  its outgoing boundary consisting of a single $n$-handle attached to
  the basepoint component. In particular, it applies to $r(M)$ too, so
  $r(M) \circ r(r(M)) \in \mathcal{W}$.  By the argument of
  Lemma~\ref{lem:W2outof3}, all three morphisms in the composition
  $M \circ r(M) \circ r(r(M))$ are in $\mathcal{W}$.  This establishes
  Theorem \ref{thm:StabNHandle}.
\end{proof}

%%% Local Variables: 
%%% mode: latex
%%% TeX-master: "stability2"
%%% End: 
 % Stability for handles of index n
\section{Contractibility of the higher-dimensional arc complex}
\label{sec:pf-arc-cx}

In this section we will give the proof of Theorem
\ref{thm:ContractibilityK}. It is convenient to work with two more
flexible approximations to the semi-simplicial space of Definition
\ref{defn:ComplexK}. In the first we relax the condition that the
handles be embedded: we allow them to be immersed, as long as their
cores $c(D^n \times\{0\})$ are still embedded.

\begin{definition}\label{defn-K-bar-bullet}
  Fix data $(W, \chi, \hat{\ell}_t^\mathrm{std})$ as in Definition
  \ref{defn:ComplexK}. Let us write
  $\overline{Y}(W)_0 = \overline{Y}(W, \chi,
  \hat{\ell}_t^\mathrm{std})_0$ for the set of triples
  $(t,c, \hat{L})$ where $t \in (2,\infty)$,
  $c : D^n \times D^n \looparrowright W$ is an immersion, and
  $\hat{L}(\tau)$ is a path of $\theta$-structures on
  $D^n \times D^n$, $\tau \in [0,1]$ such that
  \begin{enumerate}[(i)]
  \item\label{it:1} $c\vert_{D^n \times \{0\}}$ is an embedding, and
    there is a $\delta>0$ such that
    $c(x,v) = \chi(\tfrac{x}{\vert x \vert}, v+t \cdot e_1) + (1-\vert
    x \vert)\cdot e_0$ for $1-\vert x \vert <\delta$,
  \item\label{it:2} $c(D^n \times D^n)$ lies outside of
    $([0,s] \times L) \cup (\{s\} \times K\vert_j)$, and
    $c^{-1}(P) = \partial D^n \times D^n$,
  \item\label{it:3} the restriction
    $\ell_W\vert_{W \setminus c(D^n \times \{0\})} : W \setminus c(D^n
    \times \{0\}) \to B$ is $n$-connected,
  \item\label{it:4} $\hat{L}(0) = c^*\hat{\ell}_W$,
    $\hat{L}(1) = \hat{\ell}_t^\mathrm{std}$, and
    $\hat{L}(\tau)\vert_{\partial D^n \times D^n}$ is independent of
    $\tau \in [0,1]$.
  \end{enumerate}
  We topologise $\overline{Y}(W)_0$ as in Definition
  \ref{defn:ComplexK}, but using the $C^\infty$ topology on the space
  of immersions rather than embeddings. Let
  $\overline{Y}(W)_p \subset (Y(W)_0)^{p+1}$ be the subspace
  consisting of ordered tuples
  $((t_0, c_0, \hat{L}_0), \ldots, (t_p,c_p, \hat{L}_p))$ such that
  \begin{enumerate}[(i)]
  \item the sets $c_i(D^n \times \{0\})$ are disjoint,
  \item $t_0 < t_1 < \cdots < t_p$,
  \item $\ell_W$ restricts to an $n$-connected map
    $W \setminus (\cup_i c_i(D^n \times \{0\})) \to B$.
  \end{enumerate}
  We shall also write $\overline{Y}{}^\delta(W)_p$ for the set
  $\overline{Y}(W)_p$ equipped with the discrete topology.
\end{definition}

In the most flexible approximation, we further relax the condition
that the cores be embedded, as long as they are immersed and in
general position. We also forget condition (\ref{it:3}), and
topologise it discretely.

\begin{definition}
  Fix data $(W, \chi, \hat{\ell}_t^\mathrm{std})$ as in Definition
  \ref{defn:ComplexK}. Let us write
  $\widehat{Y}^\delta(W)_0 = \widehat{Y}^\delta(W, \chi,
  \hat{\ell}_t^\mathrm{std})_0$ for the set of triples
  $(t,c, \hat{L})$ where $t \in (2,\infty)$,
  $c : D^n \times D^n \looparrowright W$ is an immersion, and
  $\hat{L}$ is a path of $\theta$-structures on $D^n \times D^n$, such
  that
  \begin{enumerate}[(\ref{it:1}$^\prime$)]
  \item the immersion $c\vert_{D^n \times \{0\}}$ is self-transverse
    and has no triple points, and there is a $\delta>0$ such that
    $c(x,v) = \chi(\tfrac{x}{\vert x \vert}, v+t \cdot e_1) + (1-\vert
    x \vert)\cdot e_0$ for $1-\vert x \vert <\delta$.
  \end{enumerate}
  as well as (\ref{it:2}) and (\ref{it:4}) of
  Definition~\ref{defn-K-bar-bullet}. Note that the data is \emph{not}
  required to satisfy (\ref{it:3}). Let
  $\widehat Y^\delta(W)_p \subset (\widehat Y^\delta(W)_0)^{p+1}$ be
  the subset consisting of ordered tuples
  $((t_0, c_0, \hat{L}_0), \ldots, (t_p,c_p, \hat{L}_p))$ such that
  \begin{enumerate}[(i)]
  \item the immersions $c_i\vert_{D^n \times \{0\}}$ are in general
    position (i.e.\ pairwise transverse and without triple
    intersections),
  \item $t_0 < t_1 < \cdots < t_p$.
  \end{enumerate}
  As usual, this data defines a semi-simplicial set
  $\widehat Y^\delta(W)_\bullet$.
\end{definition}

\begin{lemma}\label{lem:shrinking}
  The map
  $\vert Y(W)_\bullet\vert \to \vert \overline{Y}(W)_\bullet\vert$ is
  a weak homotopy equivalence.
\end{lemma}
\begin{proof}
  The inclusion ${Y}(W)_\bullet \subset \overline{Y}(W)_\bullet$ is a
  levelwise weak homotopy equivalence, by precomposing an immersion
  $h : D^n \times D^n \looparrowright W$ which is an embedding of the
  core, and hence of a neighbourhood of the core, with an isotopy from
  the identity map of $D^n \times D^n$ to an embedding into a small
  neighbourhood of $(\partial D^n \times D^n) \cup (D^n \cup \{0\})$.
\end{proof}

By this lemma, in order to prove Theorem \ref{thm:ContractibilityK} it
is enough to show that
$$\hocolim_{g \to \infty} \vert \overline{Y}(K\vert_{[j,j+g]} \circ W, \chi, \hat{\ell}_t^\mathrm{std})_\bullet \vert$$
is weakly contractible. Since homotopy colimit commutes with geometric
realisation, we may equivalently prove weak contractibility of the
realisation of the semi-simplicial space
$$[p]
\longmapsto \hocolim_{g \to \infty} \overline{Y}(K\vert_{[j,j+g]}
\circ W, \chi, \hat{\ell}_t^\mathrm{std})_p.$$ Since each map
$(K\vert_{[i,i+1]} \circ -)_p$ forming the colimit diagram is the
inclusion of a subspace, we may replace the homotopy colimit by the
actual colimit.  We shall therefore prove weak contractibility of the
semi-simplicial space whose space of $p$-simplices is
\begin{equation*}
  \overline{Y}(K\vert_{[j,\infty)} \circ W)_p = \colim_{g \to \infty} \overline{Y}(K\vert_{[j,j+g]} \circ W, \chi, \hat{\ell}_t^\mathrm{std})_p.
\end{equation*}
(The proof would work for the homotopy colimit, but it is notationally
convenient to work with the actual colimit.)  We shall first prove
that the underlying semisimplicial \emph{set}
$\overline{Y}^\delta(K\vert_{[j,\infty)} \circ W)_\bullet$ has
contractible realisation, and then use the techniques of \cite{GR-W3}
to deduce weak contractibility in the topologised case.

\begin{lemma}\label{lem:K-hat}
  The realisation $|\widehat Y^\delta(W)_\bullet|$ is contractible.
\end{lemma}
\begin{proof}
  We first show that $\widehat Y^\delta(W)_0 \neq \emptyset$. Pick
  some $t \geq 2$ and consider the commutative square
  \begin{equation*}
    \xymatrix{
      \partial D^n \times D^n \ar@{^(->}[d]  \ar[rr]^-{\chi(-, -+t\cdot e_1)}& & P \ar@{^(->}[r] & W \ar[d]^-{\ell_W}\\
      D^n \times D^n \ar[rrr]^-{\ell^\mathrm{std}_t} \ar@{-->}[rrru]^-{g} & & & B.
    }
  \end{equation*}
  As $\ell_W$ is $n$-connected, and the pair
  $(D^n \times D^n, \partial D^n \times D^n)$ is homotopy equivalent
  to $(D^n, \partial D^n)$, there exists a dashed diagonal map $g$
  making the top triangle commute and the bottom triangle commute up
  to homotopy. As $\ell_W$ and $\ell_t^\mathrm{std}$ are covered by
  bundle maps $\hat{\ell}_W$ and $\hat{\ell}_t^\mathrm{std}$, this
  provides a bundle map $\hat{g} : T(D^n \times D^n) \to TW$ such that
  $\hat{\ell}_W \circ \hat{g}$ is homotopic to
  $\hat{\ell}^\mathrm{std}_t$ through bundle maps, via a homotopy
  which is constant over $\partial D^n \times D^n$. By Smale--Hirsch
  theory, the pair $(g, \hat{g})$ may be homotoped relative to
  $\partial D^n \times D^n$ to a pair of the form $(c, Dc)$, for $c$
  an immersion and $Dc$ its differential. Then
  $c^* \hat{\ell}_W = \hat{\ell}_W \circ Dc$ is still homotopic to
  $\hat{\ell}_t^\mathrm{std}$, and if we choose such a homotopy,
  $\hat{L}$, then we have constructed an element
  $(t,c,\hat{L}) \in \widehat Y^\delta(W)_0$.

  To prove contractibility assuming non-emptiness, we must prove that
  for each $k \geq 1$, any map
  $f: \partial I^k \to |\widehat Y^\delta(W)_\bullet|$ extends to a
  map from $I^k$.  By the simplicial approximation theorem, we can
  assume that $f$ is simplicial with respect to some PL triangulation
  of $\partial I^k$.  For each vertex $v_i \in \partial I^k$ in the
  triangulation, there is then given an element
  $f(v_i) = (t_i,c_i, \hat{L}_i) \in \widehat Y^\delta(W)_0$.  By a
  suitable application of Thom's transversality theorem, we may let
  $(t,c, \hat{L}) \in \widehat Y^\delta(W)_0$ be a slight perturbation
  of one of the $(t_i,c_i, \hat{L}_i)$ such that $t \neq t_j$ and
  $c(D^n \times \{0\}) \pitchfork c_j(D^n \times \{0\})$ for all $j$.
  A further perturbation will remove triple intersections between the
  cores.  Then $f(\partial I^k)$ is contained in the star of
  $(t,c, \hat{L})$, and therefore $f$ extends to the cone
  $C(\partial I^k) \cong I^k$.
\end{proof}

The (rather lengthy) proof of the following proposition makes
essential use of the passage to the limit $g \to \infty$.
\begin{proposition}\label{prop:K-hat-to-K}
  The natural maps
  $\overline{Y}{}^\delta(K\vert_{[j,\infty)} \circ W)_\bullet \to
  \widehat Y^\delta(K\vert_{[j,\infty)} \circ W)_\bullet$ induce a
  weak equivalence on geometric realisation.
\end{proposition}

\begin{proof}
  To ease notation throughout the proof, let us write
  $M=K\vert_{[j,\infty)} \circ W$.  Consider a map
  \begin{equation*}
    f: (I^k, \partial I^k) \lra (|\widehat Y^\delta(M)_\bullet|, |\overline{Y}{}^\delta(M)_\bullet|)
  \end{equation*}
  which we may assume simplicial with respect to some PL triangulation
  of $I^k$.  We shall explain how to homotope $f$ to a map with image
  in $|\overline{Y}{}^\delta(M)_\bullet|$.  If the image of some
  simplex $\sigma < I^k$ is a $p$-simplex
  $f(\sigma) = ((t_0, c_0, \hat{L}_0), \dots, (t_p,c_p, \hat{L}_p))$,
  then there can be three reasons why $f(\sigma)$ is not in this
  subcomplex: firstly, the cores $D_i = c_i(D^n \times \{0\})$ may not
  be embedded; secondly, they may not be pairwise disjoint; thirdly,
  the restriction
  $\ell\vert_{M \setminus \cup D_i} : M \setminus \cup D_i \to B$ may
  not be $n$-connected.  All three problems will be fixed using the
  infinite supply of embedded copies of
  $W_{1,1} = S^n \times S^n \setminus \Int(D^{2n})$ given by the
  $\theta$-end, using a well-known local move. Let us first explain
  it.

  Suppose that $h_0, h_1 : D^n \times D^n \looparrowright M$ are two
  immersed handles, such that each of the cores
  $h_i\vert_{D^n \times \{0\}}$ is self-transverse and has no triple
  points, and they also meet transversely. Around a point of
  intersection $x_0$ between the two cores, we may find a coordinate
  chart $\varphi : D^n \times D^n \hookrightarrow M$ inside which the
  cores are $\varphi(D^n \times \{0\})$ and
  $\varphi(\{0\} \times D^n)$ respectively. By scaling the handles in
  the meridian direction near the discs
  $h_0^{-1}\varphi(D^n \times \{0\})$ and
  $h_1^{-1}\varphi(\{0\} \times D^n)$, and making a change of
  coordinates, we may obtain new immersed handles $h_i'$ which agree
  with the old $h_i$ outside a small neighbourhood of
  $\varphi(D^n \times D^n)$, and which intersect
  $\varphi(D^n \times D^n)$ in $\varphi(D^n \times \tfrac{1}{2}D^n)$
  and $\varphi(\tfrac{1}{2}D^n \times D^n)$ respectively.  This
  preliminary move is illustrated in
  Figure~\ref{fig:ShrinkingHandles}.

  \begin{figure}[h]
    \begin{center}
      \includegraphics[bb=0 0 343 93]{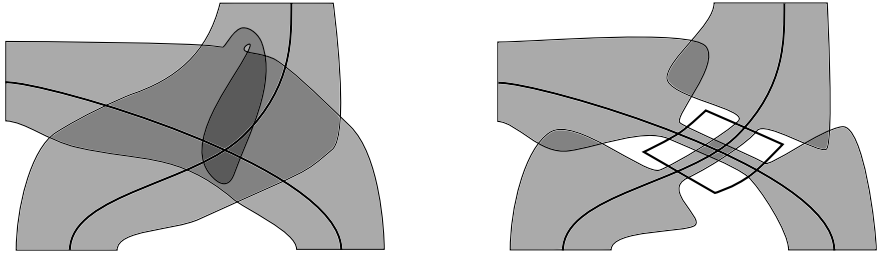}
      \caption{Shrinking and straightening handles near an intersection point so that they intersect a coordinate chart cleanly.}
      \label{fig:ShrinkingHandles}
    \end{center}
  \end{figure}
  
  We may then find an embedded path in $M$ from the boundary of
  $\varphi(D^n \times D^n)$ to an embedded copy of $W_{1,1}$ having
  standard $\theta$-structure, ensure that it is disjoint from the
  cores of all the $h_i'$, and thicken it up to obtain an embedding
  $\varphi' : (D^n \times D^n) \natural W_{1,1} \hookrightarrow M$, so
  that the handles intersect the image of $\varphi'$ in precisely
  $\varphi'(D^n \times \tfrac{1}{2}D^n)$ and
  $\varphi'(\tfrac{1}{2}D^n \times D^n)$.

  \begin{figure}[h]
    \begin{center}
      \includegraphics[bb=0 0 342 93]{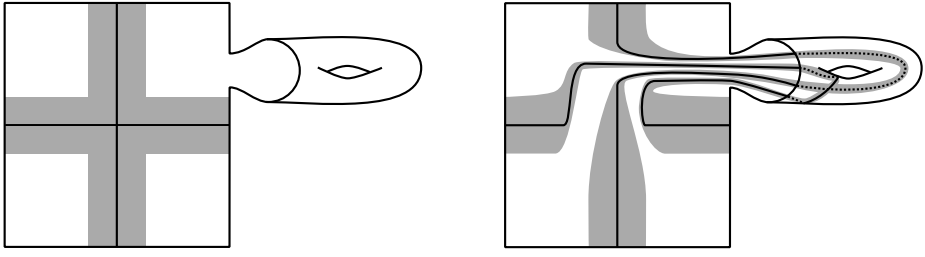}
      \caption{Extending the coordinate chart $\varphi$ to an embedding of $(D^n \times D^n) \natural W_{1,1}$, and using it to remove a point of intersection.}
      \label{fig:StableWhitneyTrick}
    \end{center}
  \end{figure}

  Now, as illustrated in Figure~\ref{fig:StableWhitneyTrick}, there
  are \emph{disjoint} embeddings
  \begin{align*}
    j_0: D^n \times \tfrac{1}{2}D^n \lra (D^n \times D^n) \natural W_{1,1},\\
    j_1: \tfrac{1}{2}D^n \times D^n \lra (D^n \times D^n) \natural W_{1,1},
  \end{align*}
  which near $\partial(D^n \times D^n)$ are given by
  $j_i(x,v) = (x, v)$. (These may be constructed as follows:
  $(D^n \times D^n) \natural W_{1,1}$ is diffeomorphic to
  $S^n \times S^n \setminus \mathrm{int}(D^n_- \times D^n_-)$, the
  manifold obtained by removing points having both coordinates in the
  lower hemisphere. This diffeomorphism can be taken to be the
  identity on that part of the boundary of
  $(D^n \times D^n) \natural W_{1,1}$ away from where the connect-sum
  was formed. Inside
  $S^n \times S^n \setminus \mathrm{int}(D^n_- \times D^n_-)$ the
  discs $D^n_+ \times \tfrac{1}{2}D^n_-$ and
  $\tfrac{1}{2}D^n_- \times D^n_+$ are disjoint, and satisfy the
  required boundary condition.) Replacing the map $h_0'$ on
  $(h'_0)^{-1}(\varphi(D^n \times D^n))$ by
  $$(h'_0)^{-1}(\varphi(D^n \times D^n)) \overset{\varphi^{-1} \circ h'_0}{\underset{\sim}{\lra}} D^n \times \tfrac{1}{2}D^n \overset{j_0} \lra (D^n \times D^n)\natural W_{1,1} \overset{\varphi'}\lra M,$$
  and the map $h_1'$ on $(h'_1)^{-1}(\varphi(D^n \times D^n))$ by
  $$(h'_1)^{-1}(\varphi(D^n \times D^n)) \overset{\varphi^{-1} \circ h'_1}{\underset{\sim}{\lra}} \tfrac{1}{2}D^n \times D^n \overset{j_1} \lra (D^n \times D^n)\natural W_{1,1} \overset{\varphi'}\lra M,$$
  we obtain new handles $h_i''$ which no longer intersect at $x_0$,
  and whose cores outside of
  $\varphi'(D^n \times D^n \natural W_{1,1})$ are unchanged.

  Finally, we claim that the $\theta$-structures $(h_0')^*\ell_M$ and
  $(h_0'')^*\ell_M$ on $D^n \times D^n$, which are already equal
  outside of $(h'_0)^{-1}(\varphi(D^n \times D^n))$, are homotopic
  relative to the complement of this subset.  This will make use of
  the embedded copy of $W_{1,1}$ having \emph{standard}
  $\theta$-structure.  If we write $\ell = (\varphi')^*\ell_M$ it
  suffices to see that the two embeddings
  $$D^n \times \tfrac{1}{2}D^n \lra (D^n \times D^n)\natural W_{1,1},$$
  given by the standard embedding and by $j_0$, pull back $\ell$ to
  $\theta$-structures which are homotopic relative to
  $\partial D^n \times \tfrac{1}{2} D^n$. The two embeddings differ by
  forming the ambient connected-sum (inside
  $\varphi'(D^n \times D^n \natural W_{1,1})$) with the sphere
  $$S^n \times\{0\} \subset S^n \times D^n_+ \subset S^n \times S^n \setminus \mathrm{int}(D^n_- \times D^n_-) \approx W_{1,1}.$$
  As the $\theta$-structure on this sphere is standard, it extends
  over the contractible space
  $$S^n \times D^n_+ \approx (D^{n+1} \setminus \tfrac{1}{2} \mathring{D}^{n+1}) \times D^{n-1} \subset D^{n+1} \times D^{n-1},$$
  and so the $\theta$-structure obtained after forming the ambient
  connected-sum is homotopic to the original one, as required. The
  analogous claim holds for $(h_1')^*\ell_M$ and $(h_1'')^*\ell_M$.

  Note that as the argument above took place locally near the
  intersection point, it works equally well for a transverse
  self-intersection of a core of an immersed handle.

  We now explain how to implement these moves. We first explain the
  argument in the case $n>1$, where the cores (having codimension at
  least 2) cannot separate. In the case $n=1$ the argument must be
  reorganised slightly, and we shall explain that at the end.
  
  \vspace{1ex}

  \noindent \textbf{Step 1}. We first explain how the map $f$ may be
  changed by a homotopy so that for each vertex $v \in I^k$ with
  $f(v)=(t,c, \hat{L})$, the manifold $D = c(D^n \times \{0\})$ is
  embedded. If $v \in \Int(I^k)$ with $f(v)=(t,c, \hat{L})$ is a
  vertex which does not satisfy this condition, then the map
  $c\vert_{D^n \times \{0\}}$ is an immersion with a finite number of
  transverse self-intersections.

  Firstly, let $c'$ be obtained by perturbing $c$ a small amount
  $\epsilon$ in the normal direction $e_1 \in \bR^n$. Precisely, we
  have a field of vectors tangent to $M$ defined on
  $\chi(\partial D^n \times (1,\infty) \times D^{n-1})$ given by the
  unit vector in the $(1,\infty)$ direction. We may extend this to a
  compactly supported smooth vector field on $M$ which is nowhere
  tangent to $c(D^n \times\{0\})$, and (choosing a Riemannian metric)
  obtain a 1-parameter family $\varphi_t$ of compactly-supported
  diffeomorphisms of $M$. Then $c'$ is the map obtained by applying
  $\varphi_\epsilon$ to $c$ for some small $\epsilon$; if $\epsilon$
  is small enough, the core of the immersion $c'$ has as many points
  of self-intersection as that of $c$, and is transverse to the core
  of $c$.

  Secondly, let $c''$ be obtained from $c'$ by applying the move using
  $W_{1,1}$ described above to remove a point of self-intersection. We
  make sure that the copy of $W_{1,1}$, and the path used to form the
  connect-sum, are disjoint from the core $c'(D^n \times \{0\})$ so
  that no new self-intersections are created when performing this
  move.  We may also ensure that no non-transverse intersections with
  other cores are created.

  To obtain a zero-simplex $(t + \epsilon, c'', \hat{L}'')$, it
  remains to find the path $\hat{L}''$.  The $\theta$-structures
  $(c')^*\hat{\ell}_{M}$, $(c'')^*\hat{\ell}_{M}$, and
  $\hat{\ell}^\mathrm{std}_{t+\epsilon}$ are all equal on
  $\partial D^n \times D^n$.  The isotopy $\varphi_s$,
  $s \in [0,\epsilon]$, gives paths
  $(c')^*\hat{\ell}_{M} \leadsto c^*\hat{\ell}_{M}$ and
  $\hat{\ell}^\mathrm{std}_{t+\epsilon} \leadsto
  \hat{\ell}^\mathrm{std}_{t}$ which become equal when restricted to
  $\partial D^n \times D^n$. Conjugating the path
  $\hat{L} : c^*\hat{\ell}_{M} \leadsto \hat{\ell}^\mathrm{std}_{t}$
  by these determines a path
  $(c')^*\hat{\ell}_{M} \leadsto \hat{\ell}^\mathrm{std}_{t+\epsilon}$
  which restricts to a nullhomotopic loop over
  $\partial D^n \times D^n$.  We may then use homotopy extension to
  find a path
  $\hat{L}' : (c')^*\hat{\ell}_{M} \leadsto
  \hat{\ell}^\mathrm{std}_{t+\epsilon}$ which is constant over
  $\partial D^n \times D^n$. Finally, by our discussion of the move
  above, $(c')^*\hat{\ell}_{M}$ and $(c'')^*\hat{\ell}_{M}$ are
  homotopic relative to $\partial D^n \times D^n$, so we may find a
  path
  $\hat{L}'' : (c'')^*\hat{\ell}_{M} \leadsto
  \hat{\ell}^\mathrm{std}_{t+\epsilon}$ which is constant over
  $\partial D^n \times D^n$.

  By construction the vertices $(t,c, \hat{L})$ and
  $(t+\epsilon,c'', \hat{L}'')$ span a 1-simplex, and if
  $f(w)=(t_i, c_i, \hat{L}_i)$ for $w$ a vertex adjacent to $v$, then
  $(t+\epsilon,c'', \hat{L}'')$ and $(t_i, c_i, \hat{L}_i)$ are
  adjacent too (when $\epsilon$ is small enough, the relative order of
  $t+\epsilon$ and $t_i$ is the same as that of $t$ and $t_i$, and if
  $c$ and $c_i$ are in general position so are $c''$ and $c_i$). Hence
  we may define a simplicial map
  \begin{equation*}
    F: [0,1] \times I^k \lra |\widehat Y^\delta(M)_\bullet|
  \end{equation*}
  by $F(0,v)=f(v) = (t,c, \hat{L})$,
  $F(1,v)=(t+\epsilon, c'', \hat{L}'')$, and $F(-, w)=f(w)$ for
  $w \neq v$. This is a homotopy which is constant on $\partial I^k$,
  and $F(1,-)$ sends $v$ to an immersion whose core has one fewer
  self-intersection. After finitely many applications of this
  technique, we may change $f$ by a homotopy so that each disc is
  embedded.

  \vspace{1ex}
  
  \noindent \textbf{Step 2}. We now explain how to modify $f$ so that
  it sends adjacent vertices in $I^k$ to embeddings with disjoint
  cores. For each pair of adjacent vertices $(v,u)$ in $I^k$ mapping
  to a 1-simplex $((t_0, c_0, \hat{L}_0),(t_1, c_1, \hat{L}_1))$ for
  which $D_0 = c_0(D^n \times \{0\})$ intersects
  $D_1 = c_1(D^n \times \{0\})$, we create embeddings $c_0''$ and
  $c_1''$ by the process described in Step 1: we perturb each $c_i$ a
  little in the $e_1$-direction, and then use the move to eliminate a
  point of intersection (using a path from the intersection point to a
  $W_{1,1}$ which is disjoint from the $D_i$ and from any adjacent
  cores). We obtain $(t_0+\epsilon, c_0'', \hat{L}''_0)$ disjoint from
  $(t_0, c_0, \hat{L}_0)$ and $(t_1+\epsilon, c_1'', \hat{L}_1'')$
  disjoint from $(t_1, c_1, \hat{L}_1)$, with $c_0''$ and $c_1''$
  having one fewer point of intersection than $c_0$ and $c_1$
  did. Hence we may define a simplical map
  \begin{equation*}
    F: [0,1] \times I^k \lra |\widehat Y^\delta(M)_\bullet|
  \end{equation*}
  by $F(0,u)=f(u) = (t_0,c_0, \hat{L}_0)$,
  $F(1,u)=(t_0+\epsilon, c''_0, \hat{L}''_0)$, similarly for $v$, and
  $F(-, w)=f(w)$ for $w \neq u$ or $v$. Since $c_i''$ is disjoint from
  $c_i$ this gives a homotopy of pairs
  $(I^k, \partial I^k) \lra (|\widehat Y^\delta(M)_\bullet|,
  |\overline{Y}{}^\delta(M)_\bullet|)$ after which the total number of
  intersection points has been reduced by one. After finitely many
  applications of this technique, we obtain a map $f$ sending adjacent
  vertices in $I^k$ to embeddings with disjoint cores.

  \vspace{1ex}

  \noindent \textbf{Step 3}. Finally, we fix simplices $\sigma < I^k$
  with
  $f(\sigma) = ((t_0,c_0, \hat{L}_0), \dots, (t_p,c_p, \hat{L}_p))$
  for which the restriction
  $\ell_{M \setminus \cup D_i}: M \setminus \cup D_i \to B$ is not
  $n$-connected.  The discs $D_i$ that we have cut out have
  codimension $n$, so the inclusion
  $M \setminus \cup D_i \hookrightarrow M$ is $(n-1)$-connected. Thus
  $\ell_{M \setminus \cup D_i}$ can fail to be $n$-connected either
  because it is not surjective on $\pi_n$ or because it is not
  injective on $\pi_{n-1}$.  Thus we set
  \begin{align}\label{eq:25}
    \begin{aligned}
      K_\sigma &= \mathrm{Ker}(\pi_{n-1}(M \setminus \cup D_i) \to \pi_{n-1}(B))\\
      I_\sigma &= \mathrm{Im}(\pi_n(M \setminus \cup D_i) \to \pi_n(B))
    \end{aligned}
  \end{align}
  and aim to kill the groups $K_\sigma$ and $\pi_n(B)/I_\sigma$.  As
  the map $\ell_M : M \to B$ is $n$-connected, it follows from the
  long exact sequence of a triple that the map
  $$\pi_n(M, M \setminus \cup D_i) \lra \pi_n(B, M \setminus \cup D_i)$$
  is surjective. If $n>2$, so $M$, $M \setminus \cup D_i$, and $B$
  have a common fundamental group, $\pi$, then by the Hurewicz theorem
  $\pi_n(M, M \setminus \cup D_i)$ is generated as a
  $\mathbb{Z}[\pi]$-module by the $(p+1)$ meridian spheres of the
  handles which have been cut out, so it follows that
  $\pi_n(B, M \setminus \cup D_i)$ is a finitely generated
  $\mathbb{Z}[\pi]$-module, and hence that $K_\sigma$ is also a
  finitely generated $\mathbb{Z}[\pi]$-module; in fact, it is
  generated by the meridian spheres of the handles that have been
  removed. If $n=2$ then the Seifert--van Kampen theorem shows that
  analogous claim is true: $K_\sigma$ is normally generated by the
  meridian circles of the handles that have been removed. Furthermore,
  \emph{if} $K_\sigma$ is trivial, then (if $n=2$ the map
  $\pi_1(M \setminus \cup D_i) \to \pi_1(M) = \pi$ is an isomorphism
  and) there is an exact sequence
  $$\pi_n(M \setminus \cup D_i) \lra \pi_n(B) \lra \pi_n(B, M \setminus \cup D_i) \lra 0$$
  with rightmost term a finitely generated
  $\mathbb{Z}[\pi]$-module. Thus $\pi_n(B)$ is generated as a
  $\mathbb{Z}[\pi]$-module by $I_\sigma \subset \pi_n(B)$ along with
  finitely many elements (this makes sense, and holds, for
  $n \geq 2$).  We will explain how to kill these finitely many extra
  elements.

  Let us describe a general construction, and some of its
  properties. Let $\{v\} < I^k$ be an interior vertex with
  $f(v) = (t,c,\hat{L})$. Suppose that we modify $c$ by choosing a
  standard $W_{1,1}$, and a path from a point in
  $D = c(D^n \times \{0\})$ to the $W_{1,1}$, both disjoint from all
  the cores of all vertices in $\Lk(t, c, \hat{L})$ (which is possible
  as $n \geq 2$), and then form a new embedding $c'$ by perturbing $c$
  a small amount $\epsilon$ in the $e_1$-direction, and then forming
  the connect-sum of its core with the core of
  $\bar{e}(S^n \times D^n) \subset W_{1,1}$ along the path. The
  $\theta$-structure $(c')^*\hat{\ell}_M$ is homotopic to
  $c^*\hat{\ell}_M$ extending the standard homotopy on the boundary,
  because the $\theta$-structure on
  $\bar{e}(S^n \times D^n) \subset W_{1,1}$ is standard, and hence by
  the same argument as in Step 1 there is a path
  $\hat{L}' : (c')^*\hat{\ell}_M \leadsto
  \hat{\ell}^\mathrm{std}_{t+\epsilon}$.  The following claim
  describes how to use this construction to modify
  $f: I^k \to |\hat{Y}^\delta(M)_\bullet|$, and the effect of the
  modification on the groups $I_\sigma$ and $K_\sigma$ defined
  in~\eqref{eq:25} above.

  \begin{claim}\label{clm:Connectivity}
    In the situation described above, let
    $f': I^k \to |\hat{Y}^\delta(M)_\bullet|$ be the simplicial map
    obtained from $f$ by changing its value at the single interior
    vertex $v \in I^k$: $f'(w) = f(w)$ for vertices $w \neq v$, but
    $f'(v) = (t + \epsilon, c',\hat{L}')$.  Then the maps $f$ and $f'$
    are homotopic relative to $\partial I^k$, and furthermore
    \begin{enumerate}[(i)]
    \item\label{it:Connectivity:2} For any simplex
      $\sigma = (v, w_1, \ldots, w_p)$, if
      $\sigma' = (v', w_1, \ldots, w_p)$ then there is a surjection
      $K_\sigma \to K_{\sigma'}$, sending the class of the meridian of
      the core of $w_i$ to the class of the meridian of the core of
      $w_i$, and whose kernel contains the class of the meridian of
      the core of $v$.      
    \item\label{it:Connectivity:3} For any simplex
      $\sigma = (v, w_1, \ldots, w_p)$, if
      $\sigma' = (v', w_1, \ldots, w_p)$, then
      $I_\sigma \subset I_{\sigma'}$.
    \end{enumerate}
  \end{claim}
  \begin{proof}
    The statement about the relative homotopy is immediate, using a
    homotopy defined as in the previous steps.

    For (\ref{it:Connectivity:2}), consider the region $X$ formed by
    the union of $c(D^n \times D^n)$, the standard $W_{1,1}$, and a
    thickening of the chosen path between them, inside which the
    (perturbed) core $D = c(D^n \times \{0\})$ is connect-summed with
    $\bar{e}(S^n \times \{0\}) \subset W_{1,1}$ to form the modified
    core $D' = c'(D^n \times \{0\})$. If
    $\sigma = (v, w_1, \ldots, w_p)$ is a simplex then all the cores
    $D_i = w_i(D^n \times \{0\})$ are disjoint from $X$. Thus there
    are maps
    $$K_{\sigma'} \overset{l} \longleftarrow\mathrm{Ker}(\pi_{n-1}(M \setminus (X \cup_i D_i)) \to \pi_{n-1}(B)) \overset{r}\lra K_\sigma,$$
    which are both surjective by transversality, as $X$ and the $D_i$
    have $n$-dimensional cores. We claim that $r$ is also injective.
    If $g : S^{n-1} \to M \setminus (X \cup_i D_i)$ becomes trivial in
    $K_\sigma$, then there is a nullhomotopy
    $\bar{g} : D^n \to M \setminus (D \cup_i D_i)$. After possibly
    perturbing it, if this nullhomotopy intersects the core of $X$
    then it must do so by intersecting $W_{1,1} \subset X$. But
    $\partial W_{1,1} \cong S^{2n-1}$ is $n$-connected, so $\bar{g}$
    can be rechosen to miss $W_{1,1}$, and hence to lie in
    $M \setminus (X \cup_i D_i)$. Thus $r$ is an isomorphism. On the
    other hand, under $l \circ r^{-1}$ the meridian of $D$ in
    $K_\sigma$ maps to the meridian of $D'$ in $K_{\sigma'}$, which is
    nullhomotopic as the meridian of
    $\bar{e}(S^n \times \{0\}) \subset W_{1,1}$ is.

    For (\ref{it:Connectivity:3}), let
    $\psi : S^n \to M \setminus (D \cup_i D_i)$. After perturbing it,
    if this map intersects the core of $X$ then it must do so by
    intersecting $W_{1,1} \subset X$. But if $\psi'$ is obtained as
    above by rechoosing the part of the image of $\psi$ in $W_{1,1}$,
    the homotopy class of $\psi'$ in $M$ is obtained from that of
    $\psi$ by the addition of a class in the $\bZ[\pi]$-submodule
    generated by $\bar{e}, \bar{f} \in \pi_n(W_{1,1})$. As the
    $\theta$-structure on $W_{1,1}$ is standard,
    $\ell_M \circ \bar{e}$ and $\ell_M \circ \bar{f}$ are
    nullhomotopic, and so
    $[\ell_M \circ \psi] = [\ell_M \circ \psi'] \in \pi_n(B)$. Thus
    $I_{\sigma'} \supset I_\sigma$.
  \end{proof}

  Let us now explain how to use the Claim above to inductively fix
  simplices $\sigma < I^k$ with
  $f(\sigma) = ((t_0,c_0, \hat{L}_0), \dots, (t_p,c_p, \hat{L}_p))$
  for which the restriction
  $\ell_{M \setminus \cup D_i}: M \setminus \cup D_i \to B$ is not
  $n$-connected. Firstly, apply the construction described above once
  for each interior vertex $v \in I^k$. This changes the map $f$ by a
  homotopy, and for every simplex $\sigma$ we now have that $K_\sigma$
  is generated by the meridians of the cores of its vertices, but also
  by Claim \ref{clm:Connectivity} (\ref{it:Connectivity:2}) that the
  classes of all these meridians are trivial. Thus $K_\sigma=0$ for
  every simplex $\sigma < I^k$.

  It remains to explain how to achieve $I_\sigma = \pi_n(B)$ for all
  simplices $\sigma < I^k$.  By the discussion before Claim
  \ref{clm:Connectivity}, when $K_\sigma=0$ it follows that the
  fundamental group of $M \setminus \cup D_i$ is also $\pi$ and that
  $\pi_n(B)$ is generated as a $\mathbb{Z}[\pi]$-module by $I_\sigma$
  along with finitely-many additional elements. We may thus choose
  finitely-many maps $\{g_\alpha : S^n \to M\}_{\alpha \in J}$ such
  that $[\ell_M \circ g_\alpha] \in \pi_n(B)$ are the additional
  generators, and we may suppose that the $g_\alpha$ are transverse to
  the discs $D_i$.  Using the same move to eliminate intersection
  points as in all the previous steps, with embedded paths and copies
  of $W_{1,1}$ chosen disjoint from all the $c_i$, the $g_\alpha$, and
  any cores in the link of any vertex of $\sigma$, we iteratedly
  remove points of intersection between the $c_i$ and the $g_\alpha$
  until they are disjoint.  As in Steps 1 and 2, this changes the map
  $f$ by a simplicial homotopy. The effect of this move on each $c_i$
  is to connect-sum its core with the core of
  $\bar{e}(S^n \times D^n) \subset W_{1,1}$, and so is described by
  Claim \ref{clm:Connectivity}. In particular, we must still have
  $K_\tau=0$ for every simplex $\tau < I^k$ by Claim
  \ref{clm:Connectivity} (\ref{it:Connectivity:2}), and must have
  $I_\sigma \subset I_{\sigma'}$ by Claim \ref{clm:Connectivity}
  (\ref{it:Connectivity:3}). But in addition, the map obtained from
  $g_\alpha$ is now disjoint from the cores of the vertices of
  $\sigma$, so gives a class
  $[g'_\alpha] \in \pi_n(M \setminus \cup_i D_i')$. Furthermore,
  although $[g_\alpha] \neq [g'_\alpha] \in \pi_n(M)$, we have
  $[\ell_M \circ g_\alpha] = [\ell_M \circ g'_\alpha]$, because the
  spheres with which we took the connected sum of $g_\alpha$ were
  inside copies of $W_{1,1}$ with standard $\theta$-structure, and so
  had nullhomotopic maps to $B$. Thus the classes
  $[\ell_M \circ g_\alpha] \in \pi_n(B)$ also lie in $I_{\sigma'}$,
  but this means that $I_{\sigma'} = \pi_n(B)$. We have therefore
  arranged that
  $\ell_{M \setminus \cup D_i}: M \setminus \cup D_i \to B$ is
  $n$-connected. By a further application of Claim
  \ref{clm:Connectivity} (\ref{it:Connectivity:2}) and
  (\ref{it:Connectivity:3}), this cannot be undone as we go on to fix
  other simplices.

  \vspace{1ex}

  \noindent \textbf{The case $n=1$}. In this case we use the above
  techniques, but with some reorganisation. We sketch the necessary
  changes.

  We first show that $\overline{Y}{}^\delta(M)_0$ is non-empty. To do
  this, choose a vertex in the non-empty set
  $\widehat{Y}{}^\delta(M)_0$, represented by a tuple $(t,c, \hat{L})$
  with $c\vert_{D^1 \times \{0\}}$ a self-transverse immersed arc
  having no triple points. This immersed arc may separate $M$, but its
  complement will always have a (unique) non-compact path component,
  which we shall call the ``non-compact region''. Step 1 can now be
  performed as follows: choose a self-intersection point which is in
  the closure of the non-compact region, and choose a path in the
  non-compact region to a copy of $W_{1,1}$, then use the move
  described above to remove this intersection point. Continuing in
  this way, we may remove all self-intersection points, and hence
  suppose that $c\vert_{D^1 \times \{0\}}$ is an embedding. We then
  perform the move in Step 3, forming the connect-sum of
  $D= c(D^1 \times \{0\})$ with a copy of
  $S^1 \times \{x\} \subset (S^1 \times S^1) \setminus \Int(D^{2}) =
  W_{1,1}$ along a path in the non-compact region: this ensures that
  the arc $D$ is non-separating (which in this dimension is the
  analogue of $K_\sigma$ being trivial). We then proceed as in Step 3,
  to ensure that $\ell_M\vert_{M \setminus D} : M \setminus D \to B$
  is 1-connected (as $D$ is non-separating, we may choose all the
  necessary paths to be disjoint from it).  We have produced the data
  of a vertex of $\overline{Y}{}^\delta(M)$.

  We now show that $\vert \overline{Y}{}^\delta(M)_\bullet \vert$ is
  contractible. Let
  $f : S^{k-1} \to \vert \overline{Y}{}^\delta(M)_\bullet \vert$ be a
  map, which we may assume is simplicial with respect to some PL
  triangulation $\vert K \vert \approx S^{k-1}$, and let
  $(t_0, c_0, \hat{L}_0) \in \overline{Y}{}^\delta(M)_0$ be a vertex
  (for example, one constructed as above). As $K$ has finitely many
  vertices, we may perturb the data $(t_0, c_0, \hat{L}_0)$ so that
  its core $D = c_0(D^1 \times\{0\})$ is transverse to the core
  $D_v = c_v(D^1 \times \{0\})$ of $f(v) = (t_v, c_v, \hat{L}_v)$ for
  every $v \in K$. Furthermore, by stretching we may ensure that $D$
  is not disjoint from the non-compact region of
  $M \setminus \cup_{v \in K} D_v$. Now we may proceed as in Step 2,
  using paths in the non-compact region of
  $(M \setminus \cup_{v \in K} D_v) \setminus D$ to remove
  intersections between $D$ and the $D_v$. This can be done so that it
  changes $f$ by a homotopy, by ensuring that when forming a parallel
  copy of $D_v$ it remains disjoint from any $D_{v'}$ that it was
  already disjoint from. This move also changes
  $(t_0, c_0, \hat{L}_0)$ and hence $D$, but does not change the fact
  that it intersects the boundary of the non-compact region of
  $(M \setminus \cup_{v \in K} D_v) \setminus D$. After finitely-many
  applications we may therefore suppose that all $D_v$ are disjoint
  from $D$.

  Finally, we choose a $W_{1,1}$ and a path from $D$ to it in the
  non-compact region, and change $c_0$ by forming the connect-sum of
  its core with that of $\bar{e}(S^1 \times D^1) \subset W_{1,1}$. For
  any simplex $\sigma < S^{k-1}$ the map
  $M \setminus \cup_{v \in \sigma} D_v \to B$ is 1-connected. The arc
  $D$ is non-separating in $M \setminus \cup_{v \in \sigma} D_v$, as
  $\bar{e}(S^1 \times \{0\})$ is non-separating in $W_{1,1}$, and
  furthermore $D$ has a dual circle $C$ in $W_{1,1}$. It follows that
  $(M \setminus \cup_{v \in K} D_v) \setminus D$ is path connected,
  and that $\pi_1(M \setminus \cup_{v \in K} D_v)$ is generated by
  $\pi_1((M \setminus \cup_{v \in K} D_v) \setminus D)$ and the
  conjugacy class of $C$. But $C$ is nullhomotopic in $B$, as the
  $W_{1,1}$ had standard $\theta$-structure, and so the map
  $(M \setminus \cup_{v \in K} D_v) \setminus D \to B$ is 1-connected
  and hence $f(\sigma)$ spans a simplex with $(t_0, c_0,
  \hat{L}_0)$. This means that $f$ has image in the star of
  $(t_0, c_0, \hat{L}_0)$, so is nullhomotopic.
\end{proof}

\begin{lemma}
  The space $|\overline{Y}(K\vert_{[j,\infty)} \circ W)_\bullet|$ is
  weakly contractible.
\end{lemma}
\begin{proof}
  We have proved that
  $|\overline{Y}{}^\delta(K\vert_{[j,\infty)} \circ W)_\bullet|$ is
  contractible, so it remains to show that the map
  \begin{equation*}
    |\overline{Y}{}^\delta(K\vert_{[j,\infty)} \circ W)_\bullet| \lra |\overline{Y}(K\vert_{[j,\infty)} \circ W)_\bullet|
  \end{equation*}
  induced by the identity is a weak homotopy equivalence.  To this
  end, we proceed as in the proof of \cite[Theorem 5.6]{GR-W3}.  In
  more detail, we first observe that for any
  $\sigma = ((t_0, c_0, \hat{L}_0), \dots, (t_p,c_p, \hat{L}_p)) \in
  \overline{Y}(K\vert_{[j,\infty)} \circ W)_p$, there is a subcomplex
  $F(\sigma)_\bullet \subset \overline{Y}{}^\delta(K\vert_{[j,\infty)}
  \circ W)_\bullet$ whose $q$-simplices are those
  $\sigma' = ((t_0',c_0', \hat{L}'_0),\dots,(t_q',c_q', \hat{L}'_q))$
  satisfying
  \begin{equation}\label{eq:4}
    ((t_0, c_0, \hat{L}_0), \dots, (t_p,c_p, \hat{L}_p), (t_0',c_0',\hat{L}'_0),
    \dots, (t_q',c_q', \hat{L}'_q)) \in \overline{Y}{}^\delta(K\vert_{[j,\infty)} \circ W)_{p+q+1}.
  \end{equation}
  The same argument that we used to prove contractibility of
  $|\overline{Y}{}^\delta(K\vert_{[j,\infty)} \circ W)_\bullet|$
  applies to $F(\sigma)_\bullet$ and proves it has contractible
  realisation. (That space is a slight enlargement of
  $\overline{Y}{}^\delta((K\vert_{[j,\infty)} \circ W) \setminus
  \cup_i c_i(D^n \times\{0\}))_\bullet$, and the restriction of
  $\ell_{K\vert_{[j,\infty)} \circ W}|_{K\vert_{[j,\infty)} \circ W
    \setminus \cup_i c_i(D^n \times\{0\})}$ is still $n$-connected.)
  Then define the subspace
  $D_{p,q} \subset \overline{Y}(K\vert_{[j,\infty)} \circ W)_p \times
  \overline{Y}{}^\delta(K\vert_{[j,\infty)} \circ W)_q$ consisting of
  those $(\sigma,\sigma')$ satisfying~\eqref{eq:4}. The obvious
  forgetful maps make $D_{\bullet, \bullet}$ into a bi-semi-simplicial
  space, augmented in both directions. The augmentation
  $D_{p,q} \to \overline{Y}(K\vert_{[j,\infty)} \circ W)_p$ then
  induces a map
  $|D_{p,\bullet}| \to \overline{Y}(K\vert_{[j,\infty)} \circ W)_p$
  with fibre $|F(\sigma)_\bullet|$ over $\sigma$, which is
  contractible. The argument of \cite[Theorem 5.6]{GR-W3} shows that
  this augmentation map is a weak equivalence (and in fact a Serre
  fibration).  By the argument of \cite[Lemma 5.8]{GR-W3}, the
  resulting weak equivalence
  $|D_{\bullet,\bullet}| \to |\overline{Y}(K\vert_{[j,\infty)} \circ
  W)_\bullet|$ factors up to homotopy through the contractible space
  $|\overline{Y}{}^\delta(K\vert_{[j,\infty)} \circ W)_\bullet|$,
  proving that all three spaces are weakly contractible.
\end{proof}

%%% Local Variables:
%%% mode: latex
%%% TeX-master: "stability2"
%%% End:
 % Contractibility of the higher-dimensional arc complex
\section{Proof of Theorem \ref{thm:StabStab}: stability for handles of index $k$, $n < k < 2n$}\label{sec:stable-stab2}

In this section we shall prove the following instance of Theorem
\ref{thm:StabStab}.

\begin{theorem}\label{thm:StabHigherHandle}
  If $M : P \leadsto Q$ is a morphism in $\cob$ whose underlying
  smooth cobordism admits a handle structure relative to $Q$
  consisting of a single $k$-handle, with $n \leq k < 2n$, attached to
  the basepoint component of $Q$, then $M \in \mathcal{W}$.
\end{theorem}

The proof shall be by induction on $k$, where the case $k=n$ has
already been established by Theorem \ref{thm:StabNHandle}. Thus we
suppose that $k > n$ and that Theorem \ref{thm:StabHigherHandle} holds
for elementary cobordisms of index $k-1$.  As in the case $k = n$ the
detailed proof is again cumbersome, but the strategy of the induction
step can be explained informally as follows.  Suppose first that there
exists an elementary cobordism $V$ of index $k-1$ such that the
composition $V \circ M = V \cup_Q M$ is defined and is diffeomorphic
to a trivial cobordism.  Then the composition $V \circ M$ is in
$\mathcal{W}$ and by induction $V$ is in $\mathcal{W}$ so by the
2-out-of-3 property we deduce that $M$ is in $\mathcal{W}$.  This
proves the theorem for elementary bordisms $M$ admitting a ``left
inverse'' elementary cobordism $V$ in this sense.  We shall then use a
simplicial resolution to reduce the general case to the case where $M$
admits a left inverse.

\subsection{Constructing auxiliary cobordisms}

Recall that to the data of an object $Q \in \cob$, an embedding
$\sigma: \bR^k \times \bR^{2n-k} \to [0,\infty) \times \bR^{\infty-1}$
with $\sigma^{-1}(Q) = \partial D^k \times \bR^{2n-k}$, and an
extension of a bundle map
$(\sigma\vert_{\partial D^k \times D^{2n-k}})^*\hat{\ell}_Q: T(D^k
\times D^{2n-k})\vert_{\partial D^k \times D^{2n-k}} \to
\theta^*\gamma$ to $T(D^k \times D^{2n-k})$ we have associated an
object $P_\sigma$ and a morphism $M_\sigma: P_\sigma \leadsto Q$ in
$\cob$.  This was explained in detail in
Construction~\ref{const:Trace}, but we remind the reader that
$P_\sigma$ was obtained by surgery along
$\sigma \vert_{\partial D^k \times \bR^{2n-k}}$ and $M_\sigma$ is the
associated trace of the surgery.  By Lemma~\ref{lem:PrepCob} it
suffices to prove Theorem~\ref{thm:StabHigherHandle} for elementary
bordisms $M = M_\sigma$ of this special type.  In particular, $M$ has
support in $\sigma(D^k \times D^{2n-k})$.

Let us consider $D^k$ as a subspace of $D^{k-1} \times D^1$, writing
coordinates as $(y,z)=(y_1, y_2, \ldots, y_{k-1}, z)$, and write
$\iota : D^{k-1} \to \partial D^k$ for the diffeomorphism onto the
lower hemisphere which is inverse to the stereographic projection
$(y_1, \dots, y_k) \mapsto \frac1{1-y_1}(y_2, \dots, y_k)$.  For each
$t \in (2,\infty)$ we have an embedding
\begin{align*}
  \mu_{t}: D^{2n-k} \times D^{k} & \lra  [0,2] \times \partial D^k \times \bR^{2n-k}\\
  (x;y,z) & \longmapsto  (\tfrac{3}{2}(1-\vert x \vert^2)(1+\tfrac{1}{3}z), \iota(y), t(1+\tfrac{1}{3}z)x).
\end{align*}
Composing $\mu_t$ with $[0,2] \times \sigma$ gives an embedding into
$[0,2] \times Q$, and this embedding when $t=3+3i$, $i \in \bN$, will
be important for us. We write
\begin{equation*}
  \hat{\ell}_i^\mathrm{triv} = (([0,2]\times\sigma) \circ \mu_{3+3i})^* \hat{\ell}_Q
\end{equation*}
for the $\theta$-structure which $\hat{\ell}_Q$ induces on
$D^{2n-k} \times D^{k}$ via these embeddings, and
\begin{equation*}
  \partial \psi_i: \partial D^{2n-k} \times D^k \lra Q
\end{equation*}
for the restriction of $([0,2] \times \sigma) \circ \mu_{3+3i}$ to
$\partial D^{2n-k} \times D^{k}$.

The following is analogous to Construction~\ref{const:Vp}, except that
we now do surgery on the spheres
$(\partial \psi_i)(\partial D^{2n-k} \times \{0\})$, which are
\emph{meridian} to $\sigma(\partial D^k \times \{0\})$, instead of the
spheres $\sigma(\partial D^n \times \{3(i+1)\})$ which were
\emph{parallel} to $\sigma(\partial D^n \times \{0\})$.

\begin{construction}\label{const:UV}
  For each $p \geq 0$ we will construct a pair of composable
  cobordisms $(1,U_p) : Q \leadsto R_p$ and $(1,V_p) : R_p \leadsto Q$
  such that $U_p$ has support outside of $\supp(M)$, and
  $V_p \circ U_p$ is an isomorphism.

  Let $(1,U_p) : Q \leadsto R_p$ be the cobordism obtained as the
  simultaneous (forward) trace of the surgeries along the disjoint
  embeddings
  $$\partial \psi_i : \partial D^{2n-k} \times D^k \lra Q \quad i=0,1,\ldots,p,$$
  with $\theta$-structure constructed using the extensions
  $\hat{\ell}^\mathrm{triv}_{i}$ of $(\partial
  \psi_i)^*\hat{\ell}_Q$. This is analogous to the cobordism
  $V_p : Q \leadsto S_p$ in Construction \ref{const:Vp}, except that
  we are now doing surgery on several spheres meridian to the core of
  $\sigma$ instead of several spheres parallel to the core of
  $\sigma$.  It is possible to arrange that the support of $U_p$ is
  disjoint from that of $M$, and we do so. We write
  $\psi_i : D^{2n-k} \times D^k \to U_p$ for the embedding of the
  $i$th handle.

  By construction, there is an embedding
  $i : U_p \hookrightarrow [0,2] \times Q$ relative to
  $\{0\} \times Q$, and a homotopy of $\theta$-structures
  $i^*\hat{\ell}_Q \simeq \hat{\ell}_{U_p}$ relative to $Q$. Using the
  isotopy extension theorem, and the homotopy extension property for
  bundle maps, this gives a $\theta$-cobordism
  $(1, V_p) : R_p \leadsto Q$ and a path from $V_p \circ U_p$ to
  $[0,2] \times Q$ in $\cob(Q,Q)$.
\end{construction}

The composable cobordisms $M : P \leadsto Q$ and
$U_p : Q \leadsto R_p$ have disjoint support by construction, so may
be subjected to interchange of support, giving composable cobordisms
$\mathcal{L}_{M}(U_p) : P \leadsto P'_p$ and
$\mathcal{R}_{U_p}(M) : P'_p \leadsto R_p$.

\begin{lemma}\label{lem:CancelHandle}
  The cobordism $V_p \circ \mathcal{R}_{U_p}(M) : P'_p \leadsto Q$ has
  a handle structure relative to $Q$ having handles of index $(k-1)$
  only.
\end{lemma}
\begin{proof}
  The cobordism $V_p$ consists of $(p+1)$ handles of index $(k-1)$
  relative to $Q$, and $\mathcal{R}_{U_p}(M)$ consists of a single
  handle of index $k$ relative to $R_p$. We claim that the handle of
  $\mathcal{R}_{U_p}(M)$ may be cancelled against one of the handles
  of $V_p$, leaving a cobordism with $p$ handles of index $(k-1)$. We
  shall explain the case $p=0$; the argument is the same in general,
  working only with the innermost handle.

  As an abstract manifold $R_0$ is the result of doing surgery on $Q$
  along
  $$\partial \psi_0 : \partial D^{2n-k} \times D^k \lra Q.$$
  The image of $\partial\psi_0$ is disjoint from the set
  $\sigma(\partial D^k \times D^{2n-k}) \subset Q$ onto which the
  handle of $M$ is attached, so we may consider
  $\sigma\vert_{\partial D^k \times D^{2n-k}}$ as the attaching map
  for the unique handle of $\mathcal{R}_{U_p}(M)$ as well. In $Q$,
  $\partial\psi_0$ is the embedding of a (thickened) meridian sphere
  of $\sigma(\partial D^k \times \{0\})$. The cobordism $V_0$ is the
  trace of a surgery along a $(2n-k)$-sphere, which in terms of the
  surgered manifold
  $$R_0 \approx (Q \setminus \partial\psi_0(\partial D^{2n-k} \times \mathrm{int} D^k)) \cup_{\partial D^{2n-k} \times \partial D^k} (D^{2n-k} \times \partial D^k)$$
  can be described as the union of the disc $D^{2n-k} \times \{\ast\}$
  with the meridian disc $\sigma(\{\ast\} \times 2D^{2n-k})$ for some
  $\ast \in \partial D^k$. In particular, it intersects
  $\sigma(\partial D^k \times \{0\})$ transversely in a single point,
  which shows that the handle of $\mathcal{R}_{U_p}(M)$ cancels the
  handle of $V_0$, as required.
\end{proof}

\subsection{A semi-simplicial resolution}

In this section we shall construct augmented semi-simplicial spaces
$\mathcal{Z}_j(P)_\bullet \to \mathcal{F}_j(P)$, for any $P \in \cob$
equipped with some auxiliary data. These semi-simplicial spaces will
play roles to those in Section~\ref{sec:ResNHandles}, although their
definition is different.

\begin{definition}
  Fix a $P \in \cob$, a $W = (s, W) \in \mathcal{F}_j(P)$ for some
  $j \geq 0$, an embedding
  $\chi : \partial D^k \times (\bR^{2n-k} \setminus D^{2n-k})
  \hookrightarrow P$ and a $\theta$-structure
  $\hat{\ell}^\mathrm{std}$ on
  $[0,2] \times \partial D^k \times \bR^{2n-k}$ which restricts to
  $\chi^*\hat{\ell}_P$ on
  $\{0\} \times \partial D^k \times (\bR^{2n-k} \setminus D^{2n-k})$.

  Let $Z(W)_0 = Z(W, \chi, \hat{\ell}^\mathrm{std})_0$ be the set of
  tuples $(t,c,\hat{L})$ consisting of a $t \in (2,\infty)$, an
  embedding $c : D^{2n-k} \times D^{k} \hookrightarrow W$, and a path
  of $\theta$-structures
  $\hat{L} \in \Bun^\theta(T(D^{2n-k} \times D^{k}))^I$ such that
  \begin{enumerate}[(i)]
  \item there is a $\delta>0$ such that
    $c(x,v) = \chi \circ \mu_t(\tfrac{x}{\vert x \vert}, v) + (1-\vert
    x \vert)\cdot e_0$ for $1-\vert x \vert < \delta$, where we have
    used that
    $\mu_t(\partial D^{2n-k} \times D^k) \subset \{0\} \times \partial
    D^k \times (\bR^{2n-k} \setminus D^{2n-k})$ to form
    $\chi \circ \mu_t$,
  \item the image of $c$ is disjoint from $[0,s] \times L$, and
    $c^{-1}(P) = \partial D^{2n-k} \times D^k$,
  \item $\hat{L}$ is a path from $c^*\hat{\ell}_W$ to
    $\mu_t^*\hat{\ell}^\mathrm{std}$ which is constant over
    $\partial D^{2n-k} \times D^k$.
  \end{enumerate}
  We topologise $Z(W)_0$ as a subspace of
  $$\bR \times \Emb(D^{2n-k} \times D^{k}, [0,\infty) \times \bR^\infty) \times \Bun^\theta(T(D^{2n-k} \times D^{k}))^I.$$
  Let
  $Z(W)_p = Z(W, \chi, \hat{\ell}^\mathrm{std})_p \subset (Z(W, \chi,
  \hat{\ell}^\mathrm{std})_0)^{p+1}$ be the subset consisting of
  tuples
  $(t_0, c_0, \hat{L}_0, t_1, c_1, \hat{L}_1, \ldots, t_p, c_p,
  \hat{L}_p)$ such that
  \begin{enumerate}[(i)]
  \item each $(t_i, c_i, \hat{L}_i)$ lies in $Z(W)_0$,
  \item\label{it:disj} the $c_i$ are disjoint,
  \item $t_0 < t_1 < \cdots < t_p$.
  \end{enumerate}
  We topologise $Z(W)_p$ as a subspace of
  $(Z(W, \chi, \hat{\ell}^\mathrm{std})_0)^{p+1}$. The collection
  $Z(W)_\bullet$ has the structure of a semi-simplicial space, where
  the $i$th face map is given by forgetting $(t_i, c_i, \hat{L}_i)$.
\end{definition}

We now combine all of the $Z(W)_\bullet$ into a single augmented
semi-simplicial space.

\begin{definition}\label{defn:CxL}
  Fix a $P \in \cob$, a $j \geq 0$, an embedding
  $\chi : \partial D^k \times (\bR^{2n-k} \setminus D^{2n-k})
  \hookrightarrow P$ and a $\theta$-structure
  $\hat{\ell}^\mathrm{std}$ on
  $[0,2] \times \partial D^k \times \bR^{2n-k}$ which restricts to
  $\chi^*\hat{\ell}_P$ on
  $\{0\} \times \partial D^k \times (\bR^{2n-k} \setminus D^{2n-k})$.

  Let
  $\mathcal{Z}_j(P)_p = \mathcal{Z}_j(P, \chi,
  \hat{\ell}^\mathrm{std})_p$ be the set of tuples $(s,W;x)$ with
  $(s,W) \in \mathcal{F}_j(P)$ and
  $x \in Z(W, \chi, \hat{\ell}^\mathrm{std})_p$. Topologise this set
  as a subspace of
  $$\mathcal{F}_j(P) \times (\bR \times \Emb(D^{2n-k} \times D^{k}, [0,\infty) \times \bR^\infty) \times \Bun^\theta(T(D^{2n-k} \times D^{k}))^I)^{p+1}.$$
  The collection $\mathcal{Z}_j(P)_\bullet$ has the structure of a
  semi-simplicial space augmented over $\mathcal{F}_j(P)$, where the
  $i$th face maps forgets $(t_i, c_i, \hat{L}_i)$, and the
  augmentation map just remembers the underlying $\theta$-manifold
  $(s,W)$.
\end{definition}

The main result concerning these semi-simplicial spaces is the
following, which is analogous to Theorem \ref{thm:ContractibilityK}.

\begin{theorem}\label{thm:ContractibilityL}
  If $k > n$, then for any data $(\chi, \hat{\ell}^\mathrm{std})$ as
  in Definition \ref{defn:CxL} the augmentation map
  $\vert \mathcal{Z}_j(P)_\bullet \vert \to \mathcal{F}_{j}(P)$ is a
  weak homotopy equivalence for all $j$.
\end{theorem}

The proof of this theorem is rather easier than the corresponding
Theorem \ref{thm:ContractibilityK}, which is suggested by the fact
that while that theorem only holds in the limit $j \to \infty$, this
theorem holds for finite $j$. The reason is that Theorem
\ref{thm:ContractibilityK} concerns $n$-dimensional submanifolds of a
$2n$-manifold, which were made disjoint using the infinite supply of
$W_{1,1}$'s available in the limit, whereas Theorem
\ref{thm:ContractibilityL} concerns $(2n-k)$-dimensional submanifolds
of a $2n$-manifold, which can be made disjoint merely by general
position since $2n-k < n$.

\begin{proof}[Proof of Theorem \ref{thm:ContractibilityL}]
  Just as in Lemma \ref{lem:qfib}, the map
  $\vert \mathcal{Z}_j(P)_\bullet\vert \to \mathcal{F}_j(P)$ is a
  quasifibration with fibre $\vert Z(W)_\bullet \vert$ over
  $(s,W) \in \mathcal{F}_j(P)$. Hence it will be enough to show that
  $\vert Z(W)_\bullet\vert$ is weakly contractible for each $W$.

  Let $\overline{Z}(W)_\bullet$ be defined analogously to
  ${Z}(W)_\bullet$ with the exception that for a tuple
  $(t_0, c_0, \hat{L}_0, t_1, c_1, \hat{L}_1, \ldots, t_p, c_p,
  \hat{L}_p)$ to span a $p$-simplex we require only the weaker
  condition
  \begin{enumerate}[(\ref{it:disj}$^\prime$)]
  \item the embeddings $c_i\vert_{D^{2n-k} \times \{0\}}$ are disjoint.
  \end{enumerate}
  The inclusion ${Z}(W)_\bullet \to \overline{Z}(W)_\bullet$ is a
  levelwise weak homotopy equivalence by the argument of Lemma
  \ref{lem:shrinking}, so it is enough to show that
  $\vert\overline{Z}(W)_\bullet\vert$ is weakly contractible for each
  $W$. It is easy to verify that $\overline{Z}(W)_\bullet$ is a
  topological flag complex in the sense of \cite[Definition
  6.1]{GR-W2}, and we claim that it satisfies the conditions of
  \cite[Theorem 6.2]{GR-W2}. (We recall these conditions below.)  This
  immediately implies the result.

  Condition (i), which says that the augmentation map has local
  sections, is vacuous because $\overline{Z}(W)_\bullet$ is augmented
  only over a point.

  Next we establish condition (ii): that the augmentation map is
  surjective, or in other words that $\overline{Z}(W)_0$ is not
  empty. Let $t \in (2,\infty)$ and consider the commutative diagram
  of bundle maps
  \begin{equation*}
    \xymatrix{
      T(D^{2n-k} \times D^{k})\vert_{\partial D^{2n-k} \times D^{k}} \ar[d]\ar[rr]^-{D \chi \circ \mu_t} & & TW\vert_P \ar[r]& TW \ar[d]^{\hat{\ell}_W}\\
      T(D^{2n-k} \times D^{k}) \ar@{-->}[rrru]^-{\hat{c}}\ar[rrr]^-{\mu_t^*\hat{\ell}^\mathrm{std}}& & & \theta^*\gamma.
    }
  \end{equation*}
  As $\ell_W : W \to B$ is $n$-connected, and
  $(D^{2n-k}, \partial D^{2n-k}) \times D^{k}$ only has relative cells
  of dimension strictly less than $n$, there is a dashed bundle map
  $\hat{c}$ making the top triangle commute, and the bottom triangle
  commute up to a homotopy of bundle maps which is constant over
  $\partial D^{2n-k} \times D^{k}$. By Smale--Hirsch theory, we may
  find an immersion $c : D^{2n-k} \times D^{k} \looparrowright W$
  extending
  $\chi \circ \mu_t\vert_{\partial D^{2n-k} \times D^k} : \partial
  D^{2n-k} \times D^k \to P$ and with differential homotopic to
  $\hat{c}$ relative to $\partial D^{2n-k} \times D^{k}$. By general
  position of the core $D^{2n-k} \times \{0\}$ and shrinking, this may
  be supposed to be an embedding, which then satisfies
  $c^* \hat{\ell}_W = \hat{\ell}_W \circ Dc \simeq \hat{\ell}_W \circ
  \hat{c} \simeq \mu_t^*\hat{\ell}^\mathrm{std}$.

  Finally we establish condition (iii): that any finite collection of
  vertices of $\overline{Z}(W)_\bullet$ each span a 1-simplex with
  some other common vertex. Let $\{(t_i, c_i, \hat{L}_i)\}_{j \in J}$
  be a finite collection of 0-simplices, and choose (as in part (ii))
  another $(t, c, \hat{L})$ having $t \ll t_i$. The embedding
  $c\vert_{D^{2n-k} \times \{0\}}$ may be perturbed relative to the
  boundary to make it disjoint from every
  $c_i\vert_{D^{2n-k} \times \{0\}}$, as these cores are
  $(2n-k)$-dimensional and $(2n-k)+(2n-k) < 2n$ as we have supposed
  that $k > n$. After changing $c$ in this way (using isotopy
  extension), the 0-simplex $(t, c, \hat{L})$ spans 1-simplex with
  each $(t_i, c_i, \hat{L}_i)$, as required.
\end{proof}

\subsection{Resolving composition with $M$}

We now come to the proof of Theorem \ref{thm:StabHigherHandle}
proper. We have a morphism $(1,M) : P \leadsto Q \in \cob$ which we
have supposed is of the form $M_\sigma$ for some
$\sigma : \bR^k \times \bR^{2n-k} \hookrightarrow [0,\infty) \times
\bR^{\infty-1}$ and some extension of
$(\sigma\vert_{\partial D^k \times D^{2n-k}})^*\hat{\ell}_Q$ to
$T(D^k \times D^{2n-k})$. Gluing on $M$ defines a map
$- \circ M : \mathcal{F}(Q) \to \mathcal{F}(P)$. The restricted
embedding
$\sigma\vert_{\partial D^k \times (\bR^{2n-k} \setminus D^{2n-k})}$
has image outside of the support of $M$, so may be considered as an
embedding into either $Q$ or $P$. Let $\hat{\ell}^\mathrm{std}$ be the
$\theta$-structure on $[0,2] \times \partial D^k \times \bR^{2n-k}$
given by
$$T([0,2] \times \partial D^k \times \bR^{2n-k}) = \epsilon^1 \oplus T(\partial D^k \times \bR^{2n-k}) \xrightarrow{\epsilon^1 \oplus D\sigma} \epsilon^1 \oplus TQ \overset{\hat{\ell}_Q}\lra \theta^*\gamma_{2n}.$$
Taking the limit $j \to \infty$ in Definition \ref{defn:CxL} gives
augmented semi-simplicial spaces
\begin{align*}
  \mathcal{Z}(Q)_\bullet = \mathcal{Z}(Q, \sigma\vert_{\partial D^k \times (\bR^{2n-k} \setminus D^{2n-k})}, \hat{\ell}^\mathrm{std})_\bullet &\lra \mathcal{F}(Q)\\
  \mathcal{Z}(P)_\bullet = \mathcal{Z}(P, \sigma\vert_{\partial D^k \times (\bR^{2n-k} \setminus D^{2n-k})}, \hat{\ell}^\mathrm{std})_\bullet &\lra \mathcal{F}(P),
\end{align*}
both of which become homotopy equivalences after geometric
realisation. We wish to cover
$- \circ M : \mathcal{F}(Q) \to \mathcal{F}(P)$ by a semi-simplicial
map
$(- \circ M)_\bullet : \mathcal{Z}(Q)_\bullet \to
\mathcal{Z}(P)_\bullet$, and can do this using the extrusion
construction of Definition \ref{defn:Extrusion}, via the formula
$$(W; t, c, \hat{L}) \longmapsto (W \circ M  ; t, \epsilon_1(c), \epsilon_1(\hat{L}))$$
on 0-simplices, and the analogous formula on higher simplices. This
commutes with face maps, and defines a semi-simplicial map
$(- \circ M)_\bullet : \mathcal{Z}(Q)_\bullet \to
\mathcal{Z}(P)_\bullet$.

\begin{proposition}\label{prop:pSxAbEqL}
  For each $p \geq 0$ the map
  $(- \circ M)_p : \mathcal{Z}(Q)_p \to \mathcal{Z}(P)_p$ is an
  abelian homology equivalence.
\end{proposition}
\begin{proof}
  The cobordism $U_p : Q \leadsto R_p$ provided by Construction
  \ref{const:UV} has embeddings $\psi_i : D^{2n-k} \times D^k \to U_p$
  for $i=0,1,\ldots,p$ extending $\partial\psi_i$, and
  $\psi_i^*\hat{\ell}_{U_p}$ is equal to
  $\hat{\ell}^\mathrm{triv}_{i}=\mu_{3+3i}^*\hat{\ell}^\mathrm{std}$. Hence,
  letting $\hat{L}_i = \hat{\ell}^\mathrm{triv}_{i}$ be the constant
  homotopy, we have a map
  \begin{align*}
    \mathcal{F}(R_p) &\lra \mathcal{Z}(Q)_p\\
    X &\longmapsto (X \circ U_p; 3,\psi_0,\hat{L}_0,6,\psi_1,\hat{L}_1,\ldots,3+3p,\psi_p,\hat{L}_p).
  \end{align*}
  This map is a weak homotopy equivalence, as may be proved in the
  same way as the corresponding step of Proposition
  \ref{prop:pSxAbEq}.

  If $\mathcal{L}_{M}(U_p) : P \leadsto P'_p$ and
  $\mathcal{R}_{U_p}(M) : P'_p \leadsto R_p$ are the cobordisms
  produced by the interchange of support, then $\mathcal{L}_{M}(U_p)$
  also contains the handles $\psi_i$ carrying the tangential
  structures $\hat{\ell}^\mathrm{triv}_{i}$. Thus we may form the
  analogous map $\mathcal{F}(P'_p) \to\mathcal{Z}(P)_p$, which is a
  weak homotopy equivalence by the same argument. Now consider the
  diagram
  \begin{equation}
    \begin{gathered}
      \xymatrix{
        {\mathcal{F}(Q)}\ar[r]^-{- \circ V_p} \ar@/_/[rd]_-{- \circ V_p \circ
          \mathcal{R}_{U_p}(M)} & 
        {\mathcal{F}(R_p)} \ar[d]^-{- \circ \mathcal{R}_{U_p}(M)} \ar[r]^-\simeq &
        {\mathcal{Z}(Q)_p} \ar[d]^-{(- \circ M)_p}\\
        &
        {\mathcal{F}(P'_p)} \ar[r]^-\simeq & 
        {\mathcal{Z}(P)_p}.
      }
    \end{gathered}
  \end{equation}
  The left triangle commutes by definition of the diagonal map. The
  cobordism $V_p$ is obtained from $Q$ by attaching $(k-1)$-handles,
  so lies in $\mathcal{W}$ by inductive hypothesis. By Lemma
  \ref{lem:CancelHandle} the cobordism
  $V_p \circ \mathcal{R}_{U_p}(M)$ is obtained from $Q$ by attaching
  $(k-1)$-handles, so also lies in $\mathcal{W}$ by inductive
  hypothesis. Hence by Lemma \ref{lem:W2outof3} the cobordism
  $\mathcal{R}_{U_p}(M)$ also lies in $\mathcal{W}$, so induces an
  isomorphism on homology with all abelian coefficient systems.

  Hence if the bottom square commutes up to homotopy then we have
  proved this proposition. But this follows by the argument for the
  analogous step of Proposition \ref{prop:pSxAbEq}.
\end{proof}

This proposition shows that the semi-simplicial map
$(- \circ M)_\bullet : \mathcal{Z}(Q)_\bullet \to
\mathcal{Z}(P)_\bullet$ is a levelwise abelian homology equivalence,
and as the augmentation maps for these semi-simplicial spaces both
become weak homotopy equivalences after geometric realisation, the
spectral sequence argument from the end of Section
\ref{sec:stable-stab} shows that
$- \circ M : \mathcal{F}(Q) \to \mathcal{F}(P)$ is an abelian homology
equivalence too. Thus $M \in \mathcal{W}$, which finishes the proof of
Theorem \ref{thm:StabHigherHandle}.

By the discussion in Section \ref{sec:Simplifications}, this finishes
the proof of Theorem \ref{thm:StabStab}.

%%% Local Variables: 
%%% mode: latex
%%% TeX-master: "stability2"
%%% End: 
 % Stability for handles of index n < k < 2n
\section{Stable homology and group-completion}\label{sec:StabHomAndGC}

In this section we shall explain how Theorem \ref{thm:Main} may be
combined with the ``surgery on morphisms'' part of \cite{GR-W2} to
re-prove and strengthen the main result of that paper.

For a pair of objects $A, B \in \mathcal{C}_\theta$, there is a map
$\mathcal{C}_\theta(A,B) \to \Omega_{[A,B]} B\mathcal{C}_\theta$ to
the space of paths in $B\mathcal{C}_\theta$ from the point represented
by $A$ to that represented by $B$. We shall consider the subspaces
$\mathcal{N}_n^\theta(P) \subset \mathcal{C}_\theta(\emptyset,P)$ of
Definition \ref{defn:NN}, and we shall establish a theorem which
describes the effect of the map
$$\mathcal{N}_n^\theta(P) \hookrightarrow \mathcal{C}_\theta(\emptyset, P) \lra \Omega_{[\emptyset,P]} B\mathcal{C}_\theta$$
on homology, after suitable stabilisation.  Our proof shall make use
of the group-completion theorem applied to the category
$\cob = \mathcal{C}_{\theta,\partial L}^{n-1}$ and the weak
equivalence $B\cob \simeq B\mathcal{C}_\theta$ from
Theorem~\ref{thm:morphism-surgery-from-acta-paper}.

\begin{definition}
  Analogously to Definition \ref{defn:ThetaEnd}, we say that a
  composable sequence of cobordisms
  $$K\vert_0 \overset{K\vert_{[0,1]}}\leadsto K\vert_1 \overset{K\vert_{[1,2]}}\leadsto K\vert_2 \overset{K\vert_{[2,3]}}\leadsto K\vert_3 \leadsto  \cdots$$
  in $\mathcal{C}_{\theta}$ is a \emph{$\theta$-end in $\mathcal{C}_{\theta}$} if each $K\vert_{[i,i+1]}$ satisfies
  \begin{enumerate}[(i)]
  \item\label{it:end2:1} it is $(n-1)$-connected relative to both
    $K\vert_i$ and $K\vert_{i+1}$, and
  \item\label{it:end2:2} it contains an embedded copy of $W_{1,1}$
    with standard $\theta$-structure.
  \end{enumerate}
  We shall often refer to such a $\theta$-end in $\mathcal{C}_\theta$
  by the non-compact $\theta$-manifold
  $K \subset [0,\infty) \times (-1,1)^\infty$ obtained by composing
  all of these cobordisms.
\end{definition}

\begin{remark}
  This definition of $\theta$-end generalises the notion of a
  universal $\theta$-end in \cite[Definition 1.7]{GR-W2}; in terms of
  the characterisation in \cite[Addendum 1.9]{GR-W2}, it omits
  conditions (i) and (ii).
\end{remark}

If $K$ is a $\theta$-end in $\mathcal{C}_\theta$ then there are
induced maps
$$K\vert_{[i,i+1]} \circ - : \mathcal{N}^\theta_n(K\vert_i) \lra \mathcal{N}^\theta_n(K\vert_{i+1}),$$
as $W \cup_{K\vert_i} K\vert_{[i,j]}$ has $n$-connected structure map
to $B$, by assumption (\ref{it:end2:1}) above. The following should be
considered as a strengthened version of \cite[Theorem 1.8]{GR-W2}.

\begin{theorem}\label{thm:StableHomologyBetter}
  Let $K$ be a $\theta$-end in $\mathcal{C}_\theta$ such that
  $\mathcal{N}^\theta_n(K\vert_0) \neq \emptyset$. Then the map
  $$\hocolim_{i \to \infty} \mathcal{N}^\theta_n(K\vert_i) \lra \hocolim_{i \to \infty} \Omega_{[\emptyset, K\vert_i]} B\mathcal{C}_\theta$$
  is acyclic.
\end{theorem}

Similarly to \cite[Section 7.4]{GR-W2}, the proof of Theorem
\ref{thm:StableHomologyBetter} will use the group-completion theorem
for categories. In fact, the group-completion theorem will be used to
prove a weaker preliminary result, which we state in
Proposition~\ref{prop:StableHomologyBad} below; we shall deduce
Theorem \ref{thm:StableHomologyBetter} from it by a further
application of Theorem \ref{thm:StabStab}.

We first make the following crucial observation: if the tangential
structure $\theta$ is such that $\mathcal{N}^\theta_n(K\vert_0)$ is
non-empty for some $K\vert_0$, then there is an $n$-connected map
$\ell_W : W \to B$ from a compact manifold. Cells of dimension at
least $(n+1)$ can be attached to $W$ in order to make this map a weak
homotopy equivalence, and hence, by \cite[Theorem A]{WallFin}, $B$
satisfies Wall's finiteness condition ($F_n$). We shall therefore
assume this property throughout this section.

\subsection{The group-completion argument}

To state the preliminary result, let
$L \subset (-1,0] \times \bR^{\infty-1}$ be such that
$(L, \partial L)$ is $(n-1)$-connected, and as in Definition
\ref{def:MainCat} consider the category
$\cob = \mathcal{C}_{\theta, \partial L}^{n-1} \subset
\mathcal{C}_{\theta, \partial L}$ consisting of those cobordisms which
are $(n-1)$-connected relative to their outgoing boundaries. As
$(L,\partial L)$ is $(n-1)$-connected, Remark \ref{rem:RemovingL}
shows that the natural map
$\mathcal{C}_{\theta, \partial L}^{n-1} \to \mathcal{C}_{\theta,
  L}^{n-1}$ is an isomorphism of categories.

\begin{definition}
  For $A, B \in \cob$, let $\cob_n(A,B) \subset \cob(A,B)$ be those
  path-components represented by cobordisms $W : A \leadsto B$ such
  that the structure map $\ell_{W}: W \to B$ is $n$-connected (recall
  that $W \in \cob(A,B)$ has had $[0,1] \times \mathrm{int}(L)$ cut
  out).
\end{definition}

The notation $\cob_n(A,B)$ should not be taken to imply that this
defines a subcategory $\cob_n$ of $\cob$: it does not; but
$\cob_n(A,-)$ is a subfunctor of $\cob(A,-)$.  The following is our
preliminary version of Theorem \ref{thm:StableHomologyBetter}.

\begin{proposition}\label{prop:StableHomologyBad}
  Suppose that $B$ satisfies Wall's finiteness condition ($F_n$), and
  let $K'$ be a $\theta$-end in $\mathcal{C}_{\theta, \partial L}$
  such that $\ell_{K'\vert_0}$ is $(n-1)$-connected. Then the map
  $$\hocolim_{i \to \infty} \cob_n(\overline{L}, K'\vert_i) \lra \hocolim_{i \to \infty} \Omega^\infty_{[\overline{L}, K'\vert_i]} B\cob$$
  is acyclic.
\end{proposition}

\begin{lemma}\label{lem:BetterEnd}
  Suppose that $B$ satisfies Wall's finiteness condition ($F_n$), and
  let $K'$ be a $\theta$-end in $\mathcal{C}_{\theta, \partial L}$
  such that $\ell_{K'\vert_0}$ is $(n-1)$-connected. Then there is
  another $\theta$-end $K''$ in $\mathcal{C}_{\theta, \partial L}$
  such that $K''\vert_0 = K'\vert_0$ and for all $i$ the structure map
  $\ell_{K''\vert_{[i,\infty)}} : K''\vert_{[i,\infty)} \to B$ is
  $n$-connected.
\end{lemma}
\begin{proof}
  It is enough to show that for any
  $P \in \mathcal{C}_{\theta, \partial L}$ such that $\ell_P$ is
  $(n-1)$-connected there is a morphism $W_P : P \leadsto P'$ in
  $\mathcal{C}_{\theta, \partial L}$ such that
  \begin{enumerate}[(i)]
  \item $\ell_{P'}$ is $(n-1)$-connected,
  \item $W_P$ is $(n-1)$-connected relative to either end,
  \item $W_P$ contains an embedded copy of $W_{1,1}$ with standard
    $\theta$-structure,
  \item $\ell_{W_P} : W_P \to B$ is $n$-connected.
  \end{enumerate}
  Then we can take $K''$ to be given by
  $W_{K'\vert_0}, W_{(K'\vert_0)'}, W_{(K'\vert_0)''}, \ldots$. Let us
  first suppose that $n \geq 3$, and we shall explain the necessary
  changes for small $n$ later.

  If we write $\pi = \pi_1(B)$ and $P^{(n-1)}$ for an $(n-1)$-skeleton
  of $P$, then it follows from the definition of Wall's condition
  ($F_n$) that $\pi_n(B, P^{(n-1)})$ is a finitely-generated
  $\bZ[\pi]$-module, and so from the long exact sequence of the triple
  $(B, P, P^{(n-1)})$ that $\pi_n(B, P)$ is finitely-generated
  too. Thus by the long exact sequence for the pair $(B, P)$ the
  $\bZ[\pi]$-module
  $$\Ker((\ell_P)_* : \pi_{n-1}(P) \to \pi_{n-1}(B))$$
  is also finitely generated. We can represent $\bZ[\pi]$-module
  generators of this kernel by finitely many maps
  $$\alpha_1, \alpha_2, \ldots, \alpha_k : S^{n-1} \lra P,$$
  and as $P$ has dimension $(2n-1)$ we may suppose that these maps are
  disjoint embeddings. Because they become nullhomotopic in $B$, and
  the tangent bundle of $P$ is pulled back from $B$, the embeddings
  $\alpha_i$ have stably trivial normal bundles, but as their normal
  bundles are of dimension $n > \mathrm{dim}(S^{n-1})$ they must in
  fact be unstably trivial. Thus we may upgrade the $\alpha_i$ to
  disjoint embeddings
  $$\hat{\alpha}_1, \hat{\alpha}_2, \ldots, \hat{\alpha}_k : S^{n-1} \times D^n \hookrightarrow \mathrm{int}(P).$$

  The trace of the (simultaneous) surgeries along these maps gives a
  cobordism $W' : P \leadsto P''$ in $\cob$. Consider the diagram
  \begin{equation*}
    \xymatrix{
      \pi_{n-1}(P) \ar[r] \ar[rd] & \pi_{n-1}(W') \ar[d]\\
      & \pi_{n-1}(B).
    }
  \end{equation*}
  As the diagonal map is surjective, the kernel of the diagonal map is
  contained in the kernel of the horizontal map, and the horizontal
  map is surjective, it follows that the vertical map is an
  isomorphism.

  Now $\pi_n(B, W')$ is a quotient of $\pi_n(B, P)$ (by the long exact
  sequence of the triple $(B,W',P)$ and the fact that
  $\pi_{n-1}(W', P)=0$) and so a finitely-generated
  $\bZ[\pi]$-module. The exact sequence
  $$\cdots \lra \pi_n(W') \lra \pi_n(B) \lra \pi_n(B,W') \lra 0$$
  shows that $\pi_n(B)$ is generated as a $\bZ[\pi]$-module by the
  image of $\pi_n(W')$ along with finitely many additional elements,
  $\beta_1, \beta_2, \ldots, \beta_l : S^n \to B$. For each $\beta_i$,
  the map $S^n \overset{\beta_i}\to B \overset{\theta}\to BO(2n)$ may
  be lifted to $\hat{\beta}_i : S^n \to BO(n)$, and the disc bundle
  $D^n \to E_i \overset{\pi}\to S^n$ classified by $\hat{\beta}_i$ has
  tangent bundle classified by
  $E_i \overset{\pi}\to S^n \overset{\beta_i}\to B \overset{\theta}\to
  BO(2n)$, and so is endowed with a $\theta$-structure by the lift
  $\beta_i \circ \pi$. The map
  $\ell_{E_i} : E_i \overset{\pi}\to S^n \overset{\beta_i}\to B$ hits
  $[\beta_i] \in \pi_n(B)$. Let $W_P : P \leadsto P'$ be the manifold
  obtained from $W' : P \leadsto P''$ by forming the boundary
  connect-sum at $P''$ with the $\theta$-manifolds
  $E_1, E_2, \ldots, E_l$, as well as with a copy of $W_{1,1}$.

  By construction $\pi_n(\ell_{W_P})$ is surjective, as its image
  contains the image of $\pi_n(\ell_{W'})$ and the $[\beta_i]$ and
  these generate $\pi_n(B)$. Furthermore,
  $\pi_{n-1}(W') \to \pi_{n-1}(W_P)$ is an isomorphism, as
  homotopically $W_P$ is obtained by wedging $n$-spheres on to
  $W'$. Thus properties (ii)--(iv) hold. For property (i) note that
  the composition
  $$P' \hookrightarrow W_P \overset{\ell_{W_P}}\lra B$$
  is $(n-1)$-connected as the first map is $(n-1)$-connected and the
  second is $n$-connected.

  Let us now explain the necessary changes for small $n$. If $n=2$ the
  first step should be interpreted as saying that
  $\Ker((\ell_P)_* : \pi_{1}(P) \to \pi_{1}(B))$ is normally finitely
  generated as a subgroup of $\pi_1(P)$. This no longer follows from
  Wall's condition ($F_2$), but is instead a standard exercise in
  group theory: a normal subgroup $N \lhd G$ of a finitely presented
  group is normally finitely generated if and only if $G/N$ is
  finitely presented. The technique described can then be used to kill
  the finitely many normal generators of this kernel, giving
  $W' : P \leadsto P''$, and the same argument shows that
  $\pi_1(\ell_{W'})$ is an isomorphism. Now we apply Wall's ($F_2$) to
  the map $(W')^{(2)} \to W' \to B$, and hence deduce that
  $\pi_2(B, W')$ is a finitely generated $\bZ[\pi]$-module. The
  argument is then concluded as above.

  If $n=1$ the first step should be interpreted as saying that
  $\pi_0(P)$ is a finite set, which is true as $P$ is compact. We can
  then perform finitely many 0-surgeries on it to obtain a connected
  cobordism $W' : P \leadsto P''$. As $B$ satisfies Wall's ($F_1$) the
  group $\pi$ is finitely generated, so in particular is generated by
  the image of $\pi_1(W')$ and finitely many additional elements,
  $\beta_1, \beta_2, \ldots, \beta_l : S^1 \to B$. Performing the
  construction described on these gives a cobordism having the
  required properties.
\end{proof}

\begin{lemma}\label{lem:UnivEndMakesNConnected}
  Suppose that $B$ satisfies Wall's finiteness condition ($F_n$), and
  let $K''$ be a $\theta$-end in $\mathcal{C}_{\theta, \partial L}$
  such that the structure map
  $\ell_{K''\vert_{[i,\infty)}} : K''\vert_{[i,\infty)} \to B$ is
  $n$-connected for all $i$. Then for each $X \in \cob$, the inclusion
  $$\hocolim_{i \to \infty} \cob_n(X, K''\vert_i) \hookrightarrow \hocolim_{i \to \infty} \cob(X, K''\vert_i)$$
  is a weak equivalence.
\end{lemma}
\begin{proof}
  Before taking homotopy colimits, the map of direct systems is a
  levelwise inclusion of a collection of path components, so it
  certainly induces an injection on $\pi_0$ on homotopy colimits. We
  shall show that it also induces a surjection on $\pi_0$ on homotopy
  colimits, from which it follows that it is a weak homotopy
  equivalence by considering the induced map on homotopy groups with
  all basepoints.
    
  To show that it is surjective on $\pi_0$, let
  $W : X \leadsto K''\vert_i \in \cob$. The commutative square
  \begin{equation*}
    \xymatrix{
      K''\vert_i \ar[d]_-{(n-1)\text{-connected}} \ar[rrr]^-{(n-1)\text{-connected}}&&& K''\vert_{[i,\infty)} \ar[d]^-{n\text{-connected}}\\
      W \ar[rrr]^-{\ell_W} &&& B
    }
  \end{equation*}
  shows that $\ell_W$ is $(n-1)$-connected, and we must glue on some
  $K''\vert_{[i,k]}$ to $W$ in order to make it $n$-connected. The
  proof of this is similar to that of Lemma \ref{lem:BetterEnd}, and
  uses all of the same techniques. We give it here for $n \geq 3$,
  leaving the necessary changes for $n \leq 2$ to the reader.

  Write $\pi = \pi_1(B)$. As $B$ is assumed to satisfy the finiteness
  condition ($F_n$) and $W$ is a finite CW-complex, as in the proof of
  Lemma \ref{lem:BetterEnd} we deduce that
  $$\Ker((\ell_W)_* : \pi_{n-1}(W) \to \pi_{n-1}(B))$$
  is also a finitely-generated $\bZ[\pi]$-module. As
  $\pi_{n-1}(K''\vert_i) \to \pi_{n-1}(W)$ is onto, we can represent
  $\bZ[\pi]$-module generators of this kernel by finitely many maps
  $\alpha_1, \alpha_2, \ldots, \alpha_k : S^{n-1} \to K''\vert_i$, and
  because these become nullhomotopic in $B$, they must also be
  nullhomotopic in $K''\vert_{[i,\infty)}$, and hence must in fact be
  nullhomotopic in $K''\vert_{[i,j]}$ for some $i \ll j$. Let
  $W' = W \cup_{K''\vert_i} K''\vert_{[i,j]}$. As in the proof of
  Lemma \ref{lem:BetterEnd} it follows that $\ell_{W'} : W' \to B$ is
  an isomorphism on $\pi_{n-1}$.
 
  Wall's condition ($F_n$) still implies that $\pi_n(B, W')$ is a
  finitely-generated $\bZ[\pi]$-module. The exact sequence
  $$\cdots \lra \pi_n(W') \lra \pi_n(B) \lra \pi_n(B,W') \lra 0$$
  shows that $\pi_n(B)$ is generated as a $\bZ[\pi]$-module by the
  image of $\pi_n(W')$ along with finitely many additional elements,
  $\beta_1, \beta_2, \ldots, \beta_l : S^n \to B$. As the map
  $K''\vert_{[j,\infty)} \to B$ is $n$-connected the $\beta$'s can be
  lifted to $K''\vert_{[j,\infty)}$, and hence to $K''\vert_{[j,k]}$
  for some $j \ll k$.

  Letting $W'' = W' \cup_{K''\vert_j}K''\vert_{[j,k]}$, the map
  $(\ell_{W''})_* : \pi_n(W'') \to \pi_n(B)$ is surjective, as its
  image contains both the image of $(\ell_{W'})_*$ and the elements
  $[\beta_i]$. The map $\ell_{W''}$ still induces an isomorphism on
  lower homotopy groups, so is $n$-connected: hence
  $W'' \in \cob_n(X, K''\vert_k)$, as required.
\end{proof}

\begin{proof}[Proof of Proposition \ref{prop:StableHomologyBad}]
  We shall use a version of the group-completion theorem for
  categories, specifically the version given as Theorem
  \ref{thm:generalized-group-completion} in the appendix.  For
  assumption~(\ref{item:fibrancy-condition}) of that theorem we
  require that the combined source/target map is a Serre fibration,
  which follows from the isotopy extension theorem (for deforming the
  underlying manifolds) and the homotopy extension property for the
  restriction map
  $\Bun^\theta(TW) \to \Bun^\theta(TW\vert_{\partial W})$ (for
  deforming $\theta$-structures).  Let $K''$ be a $\theta$-end in
  $\mathcal{C}_{\theta, \partial L}$ provided by Lemma
  \ref{lem:BetterEnd}. In order to apply Theorem
  \ref{thm:generalized-group-completion} we must show that any
  morphism in $\cob$ is sent to an abelian homology equivalence by
  applying
  $$\hocolim_{i \to \infty} \cob(-, K''\vert_i) : \cob^\mathrm{op} \lra \mathrm{Top}.$$
  But by Lemma \ref{lem:UnivEndMakesNConnected} we may replace this by
  $\hocolim\limits_{i \to \infty} \cob_n(-, K''\vert_i)$, and it is
  the content of Theorem \ref{thm:StabStab} that this sends every
  morphism to an abelian homology equivalence. Theorem
  \ref{thm:generalized-group-completion} then allows us to conclude
  that for any object $X \in \cob$,
  \begin{equation}\label{eq:CstatementUniv}
    \hocolim_{i \to \infty} \cob_n(X, K''\vert_i) \lra \hocolim_{i \to \infty} \Omega_{[X, K''\vert_i]} B\cob
  \end{equation}
  is an acyclic map.

  We must now pass from this result to the corresponding statement for
  $K'$. Inductively choose a sequence of composable cobordisms
  $$\cdots \leadsto \overline{L}\vert_{-2} \overset{\overline{L}\vert_{[-2,-1]}}\leadsto \overline{L}\vert_{-1} \overset{\overline{L}\vert_{[-1,0]}}\leadsto \overline{L}\vert_0 = \overline{L}$$
  by letting $\overline{L}\vert_{[-i-1,-i]}$ be obtained from
  $[0,1] \times \overline{L}\vert_{-i}$ by forming the boundary
  connect-sum with $W_{1,1}$ at $\{0\} \times \overline{L}\vert_{-i}$,
  and consider the diagram
  \begin{equation*}
    \xymatrix{
      \hocolim\limits_{i \to \infty} \cob_n(\overline{L}, K''\vert_i) \ar[r] & \hocolim\limits_{j \to \infty}\hocolim\limits_{i \to \infty} \cob_n(\overline{L}_{-j}, K''\vert_i)\\
      & \hocolim\limits_{j \to \infty} \cob_n(\overline{L}_{-j}, K'\vert_0) \ar[u] \ar[d]\\
      \hocolim\limits_{i \to \infty} \cob_n(\overline{L}, K'\vert_i) \ar[r] & \hocolim\limits_{j \to \infty}\hocolim\limits_{i \to \infty} \cob_n(\overline{L}_{-j}, K'\vert_i).
    }
  \end{equation*}
  Each of these four maps is an abelian homology equivalence: the
  horizontal ones by Theorem \ref{thm:StabStab}, and the vertical ones
  by the analogue of Theorem \ref{thm:StabStab} in which stabilisation
  is formed on the left. Comparing this diagram with the analogous one
  of homotopy colimits of path spaces of $B\cob$,
  \eqref{eq:CstatementUniv} implies that
  \begin{equation}\label{eq:Cstatement}
    \hocolim_{i \to \infty} \cob_n(\overline{L}, K'\vert_i) \lra \hocolim_{i \to \infty} \Omega_{[\overline{L}, K'\vert_i]} B\cob
  \end{equation}
  is an abelian homology equivalence, and hence acyclic.
\end{proof}

\subsection{Proof of Theorem \ref{thm:StableHomologyBetter}}

Let $K$ be a $\theta$-end in $\mathcal{C}_\theta$, pick a
self-indexing Morse function $f : K\vert_0 \to [0,2n-1]$, and let
$L = f^{-1}([0, n-\tfrac{1}{2}])$ with induced $\theta$-structure
$\hat{\ell}_L$. By construction, $L$ has a handle structure with
handles of index at most $(n-1)$; equivalently, it may be obtained
from $\partial L$ by attaching handles of index at least $n$. After
moving $K\vert_0 \subset \R^\infty$ by an isotopy, we may suppose that
$L = K\vert_0 \cap ((-\infty,0] \times \bR^{\infty-1})$ as
$\theta$-manifolds: $K\vert_0$ is then an object of
$\mathcal{C}_{\theta,L}$, giving an object $K\vert_0^\circ$ of
$\mathcal{C}_{\theta,\partial L}$.

\begin{lemma}\label{lem:StraightenEnd}
  The proper embedding
  $K \hookrightarrow [0,\infty) \times (-1,1)^\infty$ may be changed
  by an isotopy, and the bundle map $\hat{\ell}_K$ changed by a
  homotopy, relative to $K\vert_0$, so that $(K,\hat{\ell}_K)$ remains
  a $\theta$-end in $\mathcal{C}_\theta$ and so that
  $$K \cap ([0,\infty) \times (-\infty,0] \times \bR^{\infty-1}) = [0,\infty) \times L$$
  as $\theta$-manifolds.
\end{lemma}

\begin{proof}
  Supposing that
  $K\vert_i \cap ((-\infty,0] \times \bR^{\infty-1}) = L$ as
  $\theta$-manifolds, we shall show that $[i,i+1] \times L$ may be
  embedded into $K\vert_{[i,i+1]}$ relative to $\{i\} \times L$ such
  that $\{i+1\} \times L$ lies in $K\vert_{i+1}$. The claim will then
  follow, by the isotopy extension theorem and induction.

  As $L$ is a $(2n-1)$-manifold with a handle structure only having
  $(n-1)$-handles and smaller, any embedding of a $k$-sphere into
  $K\vert_i$ can be isotoped off of $L$ as long as $k \leq n-1$.

  Let us first suppose that $2n \geq 6$. As
  $(K\vert_{[i,i+1]}, K\vert_{i+1})$ is $(n-1)$-connected, by handle
  trading as in the proof of the $s$-cobordism theorem
  (\cite{KervaireSCobordism}) we can find a handle structure on
  $K\vert_{[i,i+1]}$ relative to $K\vert_i$ having only $n$-handles
  and lower: these are attached along spheres of dimension at most
  $n-1$, so their attaching maps can be isotoped off of $L$, and hence
  $K\vert_{[i,i+1]}$ contains an embedded $[i,i+1] \times L$. The same
  argument applies for $2n=2$.

  If $2n=4$ we must modify the argument slightly. The above goes
  through after perhaps changing $K\vert_{[i,i+1]}$ by connect-sum
  with finitely many copies of $S^2 \times S^2$ so that handle-trading
  becomes available (just as in the proof of Lemma
  \ref{lem:HandleStr}). As $K$ contains countably many
  $S^2 \times S^2$-summands by property (\ref{it:end2:2}) of
  $\theta$-ends in $\mathcal{C}_\theta$, we may realise this by
  sliding in enough such summands down from $K\vert_{[i+1,\infty)}$.
\end{proof}

Once $K$ has been prepared using this lemma, we may form
$K\vert_{[i,i+1]}^\circ$, giving a $\theta$-end in
$\mathcal{C}_{\theta, \partial L}$. Let us write
$$\mathcal{C}_{\theta, \partial L, n}(A,B) \subset \mathcal{C}_{\theta, \partial L}(A,B)$$
for the subspace of those cobordisms $(W, \hat{\ell}_W)$ such that
$\ell_W : W \to B$ is $n$-connected. The weak homotopy equivalence
$$\mathcal{C}_{\theta, \partial L}(\overline{L}, K\vert_i^\circ) \overset{- \cup L}\lra \mathcal{C}_\theta(D(L), K\vert_i) \overset{- \circ V_L}\lra \mathcal{C}_\theta(\emptyset, K\vert_i)$$
of Lemma \ref{lem:ChangeOfModel} thus restricts to a map
\begin{equation}\label{eq:ChangeModel}
  \mathcal{C}_{\theta, \partial L, n}(\overline{L}, K\vert_i^\circ) \lra \mathcal{N}_n^\theta(K\vert_i)
\end{equation}
which is a weak homotopy equivalence. 

By assumption
$\mathcal{N}^\theta_n(K\vert_0) \simeq \mathcal{C}_{\theta, \partial
  L, n}(\overline{L}, K\vert_0^\circ)$ is non-empty, so let
$W^\circ : \overline{L} \leadsto K\vert_0^\circ$ lie in
$\mathcal{C}_{\theta, \partial L, n}(\overline{L}, K\vert_0^\circ)$,
choose a self-indexing Morse function $f : W \to [0,2n]$, and let
$V_0^\circ = f^{-1}([n- \tfrac{1}{2},2n]) : P_0^\circ \leadsto
K\vert_0^\circ$. By construction, $V_0^\circ$ has a handle structure
relative to $P_0^\circ$ having all handles of index at least $n$. The
cobordism
$K\vert_{[0,1]}^\circ \circ V_0^\circ : P_0^\circ \leadsto
K\vert_1^\circ$ is thus $(n-1)$-connected relative to $P_0^\circ$ and
contains an embedded $W_{1,1}$ with standard $\theta$-structure, so
after changing it by a path in
$\mathcal{C}_{\theta, \partial L}(P_0^\circ, K\vert_1^\circ)$ we may
factorise it as
$$P_0^\circ \overset{{_{P_0^\circ} H}}\leadsto P_1^\circ \overset{V_1^\circ}\leadsto K\vert_1^\circ$$
for some cobordism $V_1^\circ$, which will be $(n-1)$-connected
relative to $P_1^\circ$. Continuing in this way, we obtain a homotopy
commutative diagram
\begin{equation}\label{eq:ChangeOfEnds}
  \begin{gathered}
    \xymatrix{
      P_0^\circ \ar[r]^-{{_{P_0^\circ} H}} \ar[d]^-{V_0^\circ}& P_1^\circ \ar[r]^-{{_{P_1^\circ} H}} \ar[d]^-{V_1^\circ}& P_2^\circ \ar[r]^-{_{P_2^\circ} H} \ar[d]^-{V_2^\circ}& \cdots\\
      K\vert_0^\circ \ar[r]^-{K\vert_{[0,1]}^\circ} & K\vert_1^\circ \ar[r]^-{K\vert_{[1,2]}^\circ} & K\vert_2^\circ \ar[r]^-{K\vert_{[2,3]}^\circ} & \cdots
    }
  \end{gathered}
\end{equation}
in the category $\mathcal{C}_{\theta, \partial L}$. 

The composition
$\ell_{P_0^\circ} : P_0^\circ \hookrightarrow W^\circ
\overset{\ell_W}\to B$ has the second map $n$-connected by assumption,
and $W^\circ$ is obtained from $P_0^\circ$ by attaching cells of
dimension $n$ and higher so $P_0^\circ \hookrightarrow W^\circ$ is
$(n-1)$-connected, and hence $\ell_{P_0^\circ}$ is
$(n-1)$-connected. The inclusions
$P_i^\circ \hookrightarrow {_{P_i^\circ} H} \hookleftarrow
P_{i+1}^\circ$ are both $(n-1)$-connected, so it follows that all
$\ell_{P_i^\circ}$ are $(n-1)$-connected.

Applying $\mathcal{C}_{\theta, \partial L, n}(\overline{L}, -)$ to
\eqref{eq:ChangeOfEnds}, we obtain a homotopy commutative diagram of
spaces, and choosing homotopies making each square commute gives a map
\begin{equation}\label{eq:ChangeOfP}
  \hocolim_{i \to \infty} \mathcal{C}_{\theta, \partial L, n}(\overline{L},P_i^\circ) \lra \hocolim_{i \to \infty} \mathcal{C}_{\theta, \partial L, n}(\overline{L},K\vert_i^\circ).
\end{equation}
It follows from Lemma \ref{lem:LadderHtpies} that the property of this
map being an abelian homology equivalence is independent of this
choice of homotopies.

\begin{lemma}
  The map \eqref{eq:ChangeOfP} is an abelian homology equivalence.
\end{lemma}
\begin{proof}
  Similarly to the proof of Proposition \ref{prop:StableHomologyBad}
  we consider the commutative square
  \begin{equation*}
    \xymatrix{
      \hocolim\limits_{i \to \infty} \mathcal{C}_{\theta, \partial L, n}(\overline{L},P_i^\circ) \ar[r] \ar[d] & \hocolim\limits_{i \to \infty} \mathcal{C}_{\theta, \partial L, n}(\overline{L},K\vert_i^\circ) \ar[d]\\
      \hocolim\limits_{j \to \infty} \hocolim\limits_{i \to \infty} \mathcal{C}_{\theta, \partial L, n}(\overline{L}_{-j}, P_i^\circ) \ar[r] & \hocolim\limits_{j \to \infty} \hocolim\limits_{i \to \infty} \mathcal{C}_{\theta, \partial L, n}(\overline{L}_{-j}, K\vert_i^\circ)
    }
  \end{equation*}
  where
  $\cdots \leadsto \overline{L}\vert_{-2}
  \overset{\overline{L}\vert_{[-2,-1]}}\leadsto \overline{L}\vert_{-1}
  \overset{\overline{L}\vert_{[-1,0]}}\leadsto \overline{L}\vert_0 =
  \overline{L}$ is a sequence of cobordisms formed by letting
  $\overline{L}\vert_{[-i-1,-i]}$ be obtained from
  $[0,1] \times \overline{L}\vert_{-i}$ by boundary connect-sum with
  $W_{1,1}$ at $\{0\} \times \overline{L}\vert_{-i}$.

  The claim then follows as the vertical maps are abelian homology
  equivalences by Theorem \ref{thm:StabStab}, and the lower map is an
  abelian homology equivalence by the analogue of Theorem
  \ref{thm:StabStab} in which stabilisation is formed on the left.
\end{proof}

If
$W^\circ \in \mathcal{C}_{\theta, \partial L,
  n}(\overline{L},P_i^\circ)$ then the map $\ell_{W^\circ}$ is
$n$-connected and $\ell_{P_i^\circ}$ is $(n-1)$-connected. Thus
$(W^\circ, P_i^\circ)$ is $(n-1)$-connected, and so
$W^\circ \in \cob_n(\overline{L},P_i^\circ)$, and hence
$\mathcal{C}_{\theta, \partial L, n}(\overline{L},P_i^\circ) =
\cob_{n}(\overline{L},P_i^\circ)$. The argument is now completed using
the commutative diagram
\begin{equation*}
  \xymatrix{
    \hocolim\limits_{i \to \infty} \mathcal{C}_{\theta, \partial L, n}(\overline{L},P_i^\circ) \ar[r]& \hocolim\limits_{i \to \infty} \mathcal{C}_{\theta, \partial L, n}(\overline{L},K\vert_i^\circ) \ar[dd] \ar[r]^-\simeq& \hocolim\limits_{i \to \infty} \mathcal{N}_{n}^\theta(K\vert_i) \ar[dd]\\
    \hocolim\limits_{i \to \infty} \cob_{ n}(\overline{L},P_i^\circ) \ar@{=}[u] \ar[d]&\\
    \hocolim\limits_{i \to \infty} \Omega_{[\overline{L}, P_i^\circ]} B\mathcal{C}_{\theta, \partial L}  \ar[r]^-\simeq& \hocolim\limits_{i \to \infty} \Omega_{[\overline{L}, K\vert_i^\circ]} B\mathcal{C}_{\theta, \partial L} \ar[r]^-\simeq& \hocolim\limits_{i \to \infty} \Omega_{[\emptyset, K\vert_i]} B\mathcal{C}_{\theta}.
  }
\end{equation*}
Consider the horizontal maps: the left hand top map is the abelian
homology equivalence \eqref{eq:ChangeOfP}; the right hand top map is
given levelwise by the weak homotopy equivalences
\eqref{eq:ChangeModel}; the left hand bottom map is given levelwise by
concatenation of paths, so is a weak homotopy equivalence; the right
hand bottom map is given by the weak homotopy equivalence
$B\mathcal{C}_{\theta, \partial L} \cong B\mathcal{C}_{\theta, L} \to
B\mathcal{C}_\theta$ and concatenation of paths. Finally, the left
hand vertical map is an abelian homology equivalence by Proposition
\ref{prop:StableHomologyBad} and the weak homotopy equivalence
$B\cob \to B\mathcal{C}_{\theta, \partial L}$ from
Theorem~\ref{thm:morphism-surgery-from-acta-paper}.  Thus, it is at
this point that we rely on a major technical result of
\cite{GR-W2}. It follows that the right hand vertical map is also an
abelian homology equivalence, which finishes the proof of Theorem
\ref{thm:StableHomologyBetter}.

Finally, we remind the reader of the main theorem of \cite{GMTW},
concerning the homotopy type of $B\mathcal{C}_\theta$: there is a weak
homotopy equivalence
\begin{equation*}
  B\mathcal{C}_\theta \simeq \Omega^{\infty-1} \MTtheta,
\end{equation*}
defined using a parametrised Pontryagin--Thom construction. This
allows us to translate Theorem \ref{thm:StableHomologyBetter} into the
form given in Theorem \ref{thm:StableHomology}.

%%% Local Variables: 
%%% mode: latex
%%% TeX-master: "stability2"
%%% End: 
 % Stable homology and group completion
\section{Finite genus and stability for closed manifolds}
\label{sec:finite-genus-stab}

In this section we shall combine Theorems \ref{thm:Main} and
\ref{thm:StableHomology} with our results from \cite{GR-W3} to prove
Corollaries \ref{cor:Main1} and \ref{cor:Main2}.

\subsection{Finite genus and homological stability}

By Remark \ref{rem:StdVsAdmissible}, the definition of $\theta$-genus
given in Section \ref{sec:intro:finitegenus} is equivalent to
$$g^\theta(W) = \max\left\{g \in \bN \,\,\bigg|\,\, \parbox{18em}{there are $g$ disjoint copies of $W_{1,1}$ in $W$,\\ each with standard $\theta$-structure}\right\},$$
and if $W$ has non-empty boundary then we have defined the
\emph{stable $\theta$-genus} to be
$$\bar{g}^\theta(W) = \max\{g^\theta(W \natural k(W_{1,1}))-k\,\,|\,\, k \in \bN \}.$$
If $W$ is a closed $\theta$-manifold, we now define
$\bar{g}^\theta(W)$ to be $\bar{g}^\theta(W \setminus D^{2n})$.

\begin{definition}
  A \emph{graded space} is a pair $(X,h_X)$ of a space $X$ and a
  continuous map $h_X : X \to \bZ$. We write $X_{h_X=n} = h_X^{-1}(n)$
  for the subspace of degree $n$. A (degree zero) map
  $f : (X, h_X) \to (Y, h_Y)$ of graded spaces is a continuous map
  $f : X \to Y$ such that $h_Y \circ f = h_X$. Similarly, a degree $k$
  map is a continuous map $f : X \to Y$ such that
  $h_Y \circ f = h_X + k$. The homology of a graded space $(X, h_X)$
  with any system of local coefficients $\mathcal{L}$ on $X$ acquires
  an extra grading
  $$H_i(X;\mathcal{L}) = \bigoplus_{n\geq 0} H_i(X;\mathcal{L})_{h_X=n},$$
  where $H_i(X;\mathcal{L})_{h_X=n} =
  H_i(X_{h_X=n};\mathcal{L})$. Maps of graded spaces respect this
  additional grading (a degree $k$ map induces a map with a shift of
  $k$).
\end{definition}

We may use the function
$\bar{g}^\theta : \mathcal{N}^\theta_n(P) \to \bZ$ to grade the space
$\mathcal{N}^\theta_n(P)$.  Recall that in Section \ref{sec:LeftRight}
we have defined for each $\theta$-manifold $P$ a bordism
${_P H}: P \to P'$ whose outgoing boundary $P'$ is diffeomorphic to
$P$ but possibly with a different $\theta$-structure.  The map
$- \circ {_P H} : \mathcal{N}^\theta_n(P) \to
\mathcal{N}^\theta_n(P')$ then has degree 1 with respect to this new
grading, and from \cite[Theorem 7.5]{GR-W3} we deduce the following.

\begin{theorem}\label{thm:UnstableStability}
  If $2n \geq 6$, $B$ is simply-connected, and $\mathcal{L}$ is an
  abelian local coefficient system on $\mathcal{N}^\theta_n(P')$, then
  the map
  $$(- \circ {_P H})_* :H_i(\mathcal{N}^\theta_n(P);(- \circ {_P H})^*\mathcal{L})_{\bar{g}^\theta=g} \lra H_i(\mathcal{N}^\theta_n(P');\mathcal{L})_{\bar{g}^\theta=g+1}$$
  is an epimorphism if $3i \leq {g-1}$, and an isomorphism if
  $3i \leq {g-4}$ (if $\mathcal{L}$ is constant and $\theta$ is
  spherical, it is an epimorphism if $2i \leq {g-1}$, and an
  isomorphism if $2i \leq {g-3}$).
\end{theorem}

\subsection{Proof of Corollary \ref{cor:Main1} for $P \neq \emptyset$}

We wish to combine Theorem \ref{thm:UnstableStability} with Theorem
\ref{thm:Main} in order to deduce that for $M : P \leadsto Q$ a
cobordism which is $(n-1)$-connected relative to $Q$ and has
$P \neq \emptyset$, the map
$$- \circ M : \mathcal{N}^\theta_n(Q) \lra \mathcal{N}^\theta_n(P)$$
induces an isomorphism in homology in a particular range of
homological degrees. However, it is easy to see that
$\bar{g}^\theta(- \circ M) - \bar{g}^\theta(-)$ need not be constant,
so we shall need another grading of $\mathcal{N}^\theta_n(P)$ for this
to be a graded map.

\begin{lemma}\label{lem:ModifiedGenus}
  Let $M : P \leadsto Q$ be a $\theta$-cobordism which is
  $(n-1)$-connected relative to $Q$. Then there is a function
  $$\bar{g}^\theta_M : \pi_0(\mathcal{N}^\theta_n(P)) \lra \bZ$$
  such that $\bar{g}^\theta_M(- \circ M) = \bar{g}^\theta(-)$,
  $\bar{g}^\theta_M(- \circ {_P H}) = \bar{g}^\theta_M(-)+1$, and
  $\bar{g}_M^\theta(-) \leq \bar{g}^\theta(-)$.
\end{lemma}
\begin{proof}
  Let us first suppose that $P \neq \emptyset$. Then Theorem
  \ref{thm:Main} applies, and the map
  $$- \circ M : \hocolim_{g \to \infty} \mathcal{N}^\theta_n(Q, \hat{\ell}_Q^{(g)}) \lra \hocolim_{g \to \infty} \mathcal{N}^\theta_n(P, \hat{\ell}_P^{(g)})$$
  is a homology equivalence so in particular a bijection on
  $\pi_0$. Furthermore, as
  $\bar{g}^\theta(W \natural W_{1,1}) = \bar{g}^\theta(W)+1$, there is
  an induced function
  $$\bar{g}^\theta: \colim_{g \to \infty} \pi_0(\mathcal{N}^\theta_n(Q, \hat{\ell}_Q^{(g)})) \lra \colim_{+1} \bZ \cong \bZ.$$
  We may therefore define $\bar{g}^\theta_M$ by
  $$\bar{g}^\theta_M : \pi_0(\mathcal{N}^\theta_n(P)) \lra \colim_{g \to \infty} \pi_0(\mathcal{N}^\theta_n(P, \hat{\ell}_P^{(g)})) \overset{\sim}\longleftarrow \colim_{g \to \infty} \pi_0(\mathcal{N}^\theta_n(Q, \hat{\ell}_Q^{(g)})) \overset{\bar{g}^\theta}\lra \bZ,$$
  which satisfies $\bar{g}^\theta_M(W \circ M) = \bar{g}^\theta(W)$ by
  construction. That
  $\bar{g}^\theta_M(- \circ {_P H}) = \bar{g}^\theta_M(-)+1$ follows
  immediately from the same property of $\bar{g}^\theta(-)$. By the
  definition of $\bar{g}^\theta_M$, when
  $W \circ k (_PH) = W' \circ M$ for some $W'$ then we have
  $$\bar{g}^\theta_M(W) = \bar{g}^\theta(W')-k \leq \bar{g}^\theta(W' \circ M)-k = \bar{g}^\theta(W \circ k (_PH))-k = \bar{g}^\theta(W).$$

  If $P=\emptyset$ then we factor $M$ as
  $M : \emptyset \overset{D^{2n}}\leadsto S^{2n-1}
  \overset{M'}\leadsto Q$, and define
  $$\bar{g}^\theta_M : \pi_0(\mathcal{N}^\theta_n(\emptyset))
  \underset{\sim}{\overset{- \circ D^{2n}}\longleftarrow}
  \pi_0(\mathcal{N}^\theta_n(S^{2n-1}))
  \overset{\bar{g}^\theta_{M'}}\lra \bZ,$$ which also satisfies
  $\bar{g}^\theta_M(W \circ M) = \bar{g}^\theta(W)$, by the way we
  have defined the stable $\theta$-genus of a closed manifold.
\end{proof}

If we use the function
$\bar{g}^\theta_M : \mathcal{N}^\theta_n(P) \to \bZ$ to grade
$\mathcal{N}^\theta_n(P)$, then by Lemma \ref{lem:ModifiedGenus} we
have an induced map
\begin{equation}\label{eq:GeneralStabMap}
  (- \circ M)_* : H_i(\mathcal{N}^\theta_n(Q);(- \circ M)^*\mathcal{L})_{\bar{g}^\theta=g} \lra H_i(\mathcal{N}^\theta_n(P);\mathcal{L})_{\bar{g}^\theta_M=g}
\end{equation}
for each local coefficient system $\mathcal{L}$ on $M$. The following
proves Corollary \ref{cor:Main1} in the case $P \neq \emptyset$.

\begin{proposition}\label{prop:cor:MainPrelim}
  If $2n \geq 6$ and $B$ is simply-connected, then for any system of
  abelian coefficients $\mathcal{L}$ on $\mathcal{N}_n^\theta(P)$ and
  any cobordism $M : P \leadsto Q$ which is $(n-1)$-connected relative
  to $Q$ and which has $P \neq \emptyset$, the map
  \eqref{eq:GeneralStabMap} is an isomorphism
  \begin{enumerate}[(i)]
  \item for $2i \leq {g-3}$ if $\mathcal{L}$ is constant and $\theta$ is spherical, or
  \item for $3i \leq {g-4}$ if $\mathcal{L}$ is constant, or
  \item\label{it:cor:MainPrelim:3} for $3i \leq {g-4}$ if
    $\mathcal{L}$ extends to
    $\hocolim\limits_{h \to \infty} \mathcal{N}^\theta_n(P,
    \hat{\ell}_P^{(h)})$.
  \end{enumerate}
  The extension condition in (\ref{it:cor:MainPrelim:3}) always holds
  for $g \geq 7$, and for $g \geq 5$ if $\theta$ is spherical.
\end{proposition}
\begin{proof}
  Consider the commutative diagram
  \begin{equation*}
    \xymatrix{
      {\mathcal{N}^\theta_n(Q)_{\bar{g}^\theta=g}} \ar[d]^-{- \circ M} \ar[r] & {\hocolim\limits_{h \to \infty} \mathcal{N}^\theta_n(Q, \hat{\ell}_Q^{(h)})_{\bar{g}^\theta=g+h}} \ar[d]^-{- \circ M}\\
      {\mathcal{N}^\theta_n(P)_{\bar{g}^\theta_M=g}} \ar[r] & {\hocolim\limits_{h \to \infty} \mathcal{N}^\theta_n(P, \hat{\ell}_P^{(h)})_{\bar{g}^\theta_M=g+h}},
    }
  \end{equation*}
  where the right-hand map is an abelian homology equivalence, by
  Theorem \ref{thm:Main}.

  By assumption $\mathcal{L}$ is pulled back from
  $\hocolim\limits_{h \to \infty} \mathcal{N}^\theta_n(P,
  \hat{\ell}_P^{(h)})$, so defines a coefficient system on every space
  in this commutative square. By Theorem \ref{thm:UnstableStability}
  the top horizontal map induces an isomorphism on
  $H_i(-;\mathcal{L})$ as long as $3i \leq g-4$ (or $2i \leq g-3$ if
  $\theta$ is spherical). As $\overline{g}_M^\theta=g$ implies that
  $\overline{g}^\theta \geq g$, Theorem \ref{thm:UnstableStability}
  also implies that the lower horizontal map induces an isomorphism on
  $H_i(-;\mathcal{L})$ in this range of degrees, and so the left-hand
  vertical map induces an isomorphism on homology in the same range.

  If $g \geq 7$, or if $\theta$ is spherical and $g \geq 5$, then by
  Theorem \ref{thm:UnstableStability} the lower horizontal map induces
  a bijection on path components and an isomorphism on $H_1(-;\bZ)$,
  and so the two spaces have the same collections of abelian local
  coefficient systems. In particular $\mathcal{L}$ is pulled back from
  $\hocolim\limits_{h \to \infty} \mathcal{N}^\theta_n(P,
  \hat{\ell}_P^{(h)})$.
\end{proof}

\subsection{Proof of Corollary \ref{cor:Main1} in general}

We will now finish the proof of Corollary \ref{cor:Main1} by
explaining how to remove the hypothesis $P \neq \emptyset$ from
Proposition \ref{prop:cor:MainPrelim}. In order to do so, it is enough
to prove the analogue of that proposition for the cobordism
$(D^{2n}, \hat{\ell}_{D^{2n}}) : \emptyset \leadsto S^{2n-1}$, as any
$M : \emptyset \leadsto Q$ can be factorised as
$M : \emptyset \overset{D^{2n}}\leadsto S^{2n-1} \overset{M'}\leadsto
Q$, and Proposition \ref{prop:cor:MainPrelim} applies to $M'$.

We will construct a resolution of the degree zero map of graded spaces
\begin{equation}\label{eq:DiscMap}
- \circ D^{2n} : \mathcal{N}^\theta_n(S^{2n-1}) \lra \mathcal{N}^\theta_n(\emptyset)
\end{equation}
such that the induced map on the space of $p$-simplices in this
resolution is of the form covered by Proposition
\ref{prop:cor:MainPrelim}. This is entirely analogous to the argument
given for surfaces in \cite[Section 11]{R-WResolution}. The resolution
is conceptually very similar to that used in
Section~\ref{sec:StabHomAndGC}, but is a little different in detail.

\begin{definition}
  Let $\mathcal{R}(P)_0$ be the set of tuples $(s, W ; c, \hat{L})$
  where $(s, W) \in \mathcal{N}_n^\theta(P)$,
  $c : D^{2n} \hookrightarrow (-\infty,0) \times \bR^\infty$, and
  $\hat{L} \in \Bun^\theta(TD^{2n})^I$ are such that
  \begin{enumerate}[(i)]
  \item $c$ has image in $W$,
  \item $\hat{L}$ is a path from $c^*\hat{\ell}_W$ to $\hat{\ell}_{D^{2n}}$.
  \end{enumerate}
  We topologise $\mathcal{R}(P)_0$ as a subspace of
  $$\mathcal{N}_n^\theta(P) \times \Emb(D^{2n}, (-\infty,0] \times \bR^\infty) \times \Bun^\theta(TD^{2n})^I.$$
  Let $\mathcal{R}(P)_p$ consist of tuples
  $(s,W; c_0, \hat{L}_0, \ldots, c_p, \hat{L}_p)$ such that
  \begin{enumerate}[(i)]
  \item each $(s,W; c_i, \hat{L}_i)$ lies in $\mathcal{R}(P)_0$,
  \item the $c_i$ are disjoint.
  \end{enumerate}
  We topologise $\mathcal{R}(P)_p$ as a subspace of the $(p+1)$-fold
  fibre product of the projection map
  $\mathcal{R}(P)_0 \to \mathcal{N}_n^\theta(P)$. The collection
  $\mathcal{R}(P)_\bullet$ has the structure of a semi-simplicial
  space augmented over $\mathcal{N}_n^\theta(P)$, where the $i$th face
  maps forgets $(c_i, \hat{L}_i)$, and the augmentation map just
  remembers $(s, W)$. Composition with
  $\bar{g}^\theta : \mathcal{N}_n^\theta(P) \to \bZ$ gives it the
  structure of a semi-simplicial graded space augmented over
  $\mathcal{N}_n^\theta(P)$.
\end{definition}

This semi-simplicial space replaces $\mathcal{L}_j(P)_\bullet$ in
Section \ref{sec:stable-stab2}, and the following theorem may be
proved in precisely the same way as Theorem
\ref{thm:ContractibilityL}.

\begin{theorem}\label{thm:ContractibilityR}
  The augmentation map
  $\vert \mathcal{R}(P)_\bullet \vert \to \mathcal{N}_n^\theta(P)$ is
  a weak homotopy equivalence (onto the components consisting of
  non-empty manifolds).
\end{theorem}

The map \eqref{eq:DiscMap} is covered by a semi-simplicial map
$(- \circ D^{2n} )_\bullet : \mathcal{R}(S^{2n-1})_\bullet \to
\mathcal{R}(\emptyset)_\bullet$ given on 0-simplices by
$(s, W ; c, \hat{L}) \mapsto ((s, W) \circ (1, D^{2n}) ; c + e_0,
\hat{L})$, and by the analogous formula on higher simplices.

\begin{proposition}\label{prop:pSxAbEqR}
  If $2n \geq 6$ and $B$ is simply-connected, the map of graded spaces
  $$(- \circ D^{2n})_p : \mathcal{R}(S^{2n-1})_{p} \lra \mathcal{R}(\emptyset)_{p}$$
  induces an isomorphism on homology in the range described in
  Proposition \ref{prop:cor:MainPrelim}.
\end{proposition}
\begin{proof}
  An argument completely analogous to that of Propositions
  \ref{prop:pSxAbEq} and \ref{prop:pSxAbEqL} identifies the map in
  question with the map of graded spaces
  $$(- \circ D^{2n}): \mathcal{N}_n^\theta(\sqcup^{p+2}S^{2n-1}) \lra \mathcal{N}_n^\theta(\sqcup^{p+1}S^{2n-1}).$$
  As $\sqcup^{p+1}S^{2n-1}$ is non-empty for all $p \geq 0$, this
  induces isomorphisms in the range claimed, by Proposition
  \ref{prop:cor:MainPrelim}.
\end{proof}

The full statement of Corollary \ref{cor:Main1} now follows from the
usual spectral sequence argument, namely that for any system
$\mathcal{L}$ of abelian local coefficients on
$\mathcal{N}_n^\theta(\emptyset)$ the map of spectral sequences
induced by $(- \circ D^{2n})_\bullet$ has the form
\begin{equation*}
  \xymatrix{
    E^1_{p,q} = H_{q}(\mathcal{R}(S^{2n-1})_p;\mathcal{L})_{\bar{g}^\theta=g} \ar[d]\ar@{=>}[r]& H_{p+q}(\vert \mathcal{R}(S^{2n-1})_\bullet\vert;\mathcal{L})_{\bar{g}^\theta=g} \ar[d]\\
    {\widetilde{E}^1_{p,q} = H_{q}(\mathcal{R}(\emptyset)_p;\mathcal{L})_{\bar{g}^\theta=g}} \ar@{=>}[r]& H_{p+q}(\vert \mathcal{R}(\emptyset)_\bullet\vert;\mathcal{L})_{\bar{g}^\theta=g},
  }
\end{equation*}
and by Theorem \ref{thm:ContractibilityR} its target can also be
identified with the map induced by \eqref{eq:DiscMap} on homology with
$\mathcal{L}$-coefficients. Proposition \ref{prop:pSxAbEqR} shows that
the map of $E^1$-pages is an isomorphism in a certain range of
degrees, and Corollary \ref{cor:Main1} follows.

\subsection{Stable homology and proof of Corollary \ref{cor:Main2}}

If $(W, \hat{\ell}_W)$ is a $\theta$-manifold with non-empty boundary
$(P, \hat{\ell}_P)$ and such that $\ell_W : W \to B$ is $n$-connected,
then we can consider the commutative diagram
\begin{equation*}
  \xymatrix{
    \mathcal{M}^\theta_n(W,\hat{\ell}_W)  \ar[r] \ar[rd] &  \mathcal{N}^\theta_n(P, \hat{\ell}_P) \ar[r] \ar[d] & \hocolim\limits_{g \to \infty} \mathcal{N}_n^\theta(P, \hat{\ell}_{P}^{(g)}) \ar[d]\\
    & \Omega_{[0,(P, \hat{\ell}_P)]}B\mathcal{C}_\theta \ar[r]^-\simeq & \hocolim\limits_{g \to \infty} \Omega_{[\emptyset, (P, \hat{\ell}_P^{(g)})]} B\mathcal{C}_\theta
  }
\end{equation*}
By Corollary \ref{cor:Main1} the composition along the top, restricted
to the path component which it hits, is an abelian homology
isomorphism in degrees satisfying
$3* \leq {\bar{g}^\theta(W,\hat{\ell}_W)-4}$, and is a homology
isomorphism with constant coefficients in degrees satisfying
$2* \leq {\bar{g}^\theta(W,\hat{\ell}_W)-3}$ if $\theta$ is
spherical. By Theorem \ref{thm:StableHomologyBetter} the right hand
vertical map is acyclic, and the target has the homotopy type of a
loop space so all coefficient systems on it are abelian. Therefore the
diagonal map is a homology equivalence (or an acyclic map) onto the
path component which it hits in the same range of degrees. This
establishes Corollary \ref{cor:Main2} in this case, after replacing
$B\mathcal{C}_\theta$ by the homotopy equivalent
$\Omega^{\infty-1} \MTtheta$.

If $(W, \hat{\ell}_W)$ is a $\theta$-manifold with empty boundary and
such that $\ell_W : W \to B$ is $n$-connected, let
$(W', \hat{\ell}_{W'}) : \emptyset \leadsto (P, \hat{\ell}_{W'})$ be a
$\theta$-cobordism obtained by subtracting a $2n$-disc
$D^{2n}=(D^{2n}, \hat{\ell}_{D^{2n}})$ from $W$. We obtain a homotopy
commutative diagram
\begin{equation*}
  \xymatrix{
    \mathcal{M}^\theta_n(W',\hat{\ell}_{W'})  \ar[r] \ar[d]^-{- \circ D^{2n}} & \mathcal{N}^\theta_n(P, \hat{\ell}_P) \ar[r] \ar[d]^-{- \circ D^{2n}} & \Omega_{[\emptyset, P]} B\mathcal{C}_\theta \ar[d]_-{\simeq}^-{- \circ D^{2n}}\\
    \mathcal{M}^\theta_n(W,\hat{\ell}_W)  \ar[r] & \mathcal{N}^\theta_n(\emptyset) \ar[r] & \Omega B\mathcal{C}_\theta
  }
\end{equation*}
where, by Corollary \ref{cor:Main1} and the case treated already, the
left hand vertical map and the top composition are both abelian
homology isomorphisms in degrees satisfying
$3* \leq {\bar{g}^\theta(W,\hat{\ell}_W)-4}$, and homology
isomorphisms in degrees satisfying
$2* \leq {\bar{g}^\theta(W,\hat{\ell}_W)-3}$ if $\theta$ is
spherical. Therefore the bottom composition has the same property,
which establishes Corollary \ref{cor:Main2}.

%%% Local Variables: 
%%% mode: latex
%%% TeX-master: "stability2"
%%% End: 
 % Finite genus and stability for closed manifolds
\section{General tangential structures}\label{sec:GeneralTheta}

So far we have only considered $2n$-manifolds $W$ equipped with
$\theta$-structures $\hat{\ell}_W$ such that the underlying map
$\ell_W : W \to B$ is $n$-connected. In this section we shall
formulate and prove a generalisation of Corollary \ref{cor:Main2}
which allows for manifolds equipped with general tangential
structures. We shall in particular prove Corollary
\ref{cor:GeneralThetaCalc}, but also extend it to manifolds with
non-empty boundary.

\begin{definition}
  Let $\pi : E \to B$ be a fibration and $i : P \to E$ be a
  cofibration, in the usual model structure on $\mathrm{Top}$
  (compactly generated weakly Hausdorff topological spaces, weak
  homotopy equivalences, Serre fibrations, induced cofibrations). Let
  $\hAut(\pi, i)$ be the space of maps $f: E \to E$ which are weak
  equivalences and such that $\pi \circ f = \pi$ and $f \circ i = i$.
\end{definition}

Composition of maps makes $\hAut(\pi, i)$ into a grouplike topological
monoid. If $\gamma$ is a vector bundle over $B$, and we let
$\pi^*\gamma$ be the pulled back vector bundle over $E$, then each $f$
induces an endomorphism $\hat{f} : \pi^*\gamma \to \pi^*\gamma$ by
sending $(e,v) \in \pi^*\gamma \subset E \times \gamma$ to
$(f(e), v) \in E \times \gamma$, which lies in $\pi^*\gamma$ as $f$
commutes with $\pi$. Furthermore, $\hat{f}$ is the identity over the
subspace $i(P) \subset E$.

To apply this to the situation at hand, let
$\theta : B \overset{u}\to B' \overset{\theta'}\to BO(2n)$ be a pair
of tangential structures with a map $u$ between them, which is a
fibration, and write
$\hat{u} : \theta^*\gamma_{2n} \to (\theta')^*\gamma_{2n}$ for the
induced bundle map. Let $W$ be a $2n$-dimensional manifold and
$\hat{\ell}_{\partial W} : TW\vert_{\partial W} \to
\theta^*\gamma_{2n}$ be a boundary condition with underlying map
$\ell_{\partial W} : \partial W \to B$ a cofibration. Then
$\hAut(u, \ell_{\partial W})$ acts on
$\Bun^{\theta}_{n}(TW, \hat{\ell}_{\partial W})$ via
$(f, \hat{\ell}) \mapsto \hat{f} \circ \hat{\ell}$, and this action
commutes with the map
\begin{align*}
  \Bun^{\theta}_{n}(TW, \hat{\ell}_{\partial W}) & \lra \Bun^{\theta'}(TW,\hat{u} \circ \hat{\ell}_{\partial W})\\
  \hat{\ell} & \longmapsto \hat{u} \circ \hat{\ell}
\end{align*}
because $\hat{u} \circ \hat{f} = \hat{u}$.

\begin{lemma}\label{lem:ChangeOfStr}
  Assume as above that $\ell_{\partial W}: \partial W \to B$ is a
  cofibration and that $u: B \to B'$ is a fibration.  Assume in
  addition that the map $u$ is $n$-co-connected. Then the induced map
  $$\Bun^{\theta}_{n}(TW, \hat{\ell}_{\partial W}) \hcoker \hAut(u, \ell_{\partial W}) \lra \Bun^{\theta'}(TW,\hat{u} \circ \hat{\ell}_{\partial W})$$
  is a homotopy equivalence onto the path components which it hits.
\end{lemma}
\begin{proof}
  Since $u$ is a fibration, so is the induced map
  \begin{equation}\label{eq:ChangeOfTheta}
    \Bun^{\theta}_{n}(TW, \hat{\ell}_{\partial W}) \lra \Bun^{\theta'}(TW,\hat{u} \circ\hat{\ell}_{\partial W}),
  \end{equation}
  and the fibre over $\hat{\ell}'_W$ may be identified with the space
  of relative lifts
  \begin{equation*}
    \xymatrix{
      \partial W \ar[r]^-{\ell_{\partial W}} \ar@{^(->}[d] & B \ar[d]^-{u}\\
      W \ar@{-->}[ru] \ar[r]^-{\ell'_W}& B'
    }
\end{equation*}
which are also $n$-connected maps. If such a lift exists, choose one
and call it $f : W \to B$. We shall consider this $f$ as a morphism in
the category $\partial W / \mathrm{Top} / B'$ of spaces over $B'$ and
under $\partial W$; this category has the structure of a simplicial
model category and $W$ is cofibrant and $B$ is both fibrant and
cofibrant, cf.\ e.g.~\cite[Example 1.7 and Definition
4.11]{GoerssSchemmerhorn}. Composition with $f$ defines a map
\begin{equation}\label{eq:1}
  \mathrm{Map}_{\partial W / \mathrm{Top} / B'}(B, B) \lra \mathrm{Map}_{\partial W / \mathrm{Top} / B'}(W, B)  
\end{equation}
between mapping simplicial sets, which can be seen to be a weak
equivalence by induction on cells in an approximation to $(B,W)$ by a
relative CW complex.  To give more details, let us write
$F: (\partial W / \mathrm{Top}/B')^{\mathrm{op}} \to \mathrm{sSet}$
for the functor represented by the object $B$.  Any representable
functor take pushout diagrams to pullback diagrams and by general
properties of (simplicial) model categories, $F$ sends cofibrations to
fibrations and sends weak equivalences between cofibrant objects to
weak equivalences.  In particular, if $X$ is a cofibrant object and
$Y = X \cup D^k$ is obtained by attaching a cell $D^k \to B'$ to $X$
along some map $\partial D^k \to X$ over $B'$, then there is the
(strict) pullback diagram
\begin{equation*}
  \xymatrix{
    F(Y) \ar[r] \ar[d] & \mathrm{Map}_{\mathrm{Top}/B'}(D^k,B) \ar[d]\\
    F(X) \ar[r] & \mathrm{Map}_{\mathrm{Top}/B'}(\partial D^k,B),
  }
\end{equation*}
in which the vertical maps are actually fibrations.  Up to homotopy, a
point in $\mathrm{Map}_{\mathrm{Top}/B'}(\partial D^k,B)$ is described
by a map from $\partial D^k$ to a homotopy fibre of $u$ and the space
of extensions to $D^k$ is a model for the homotopy fibre of the
horizontal maps in the square.  In particular the fibres are weakly
contractible when $k \geq n+1$ if the homotopy fibres of $u$ are
$(n-1)$-types, as we have assumed.  Hence by induction on cells
$F(Y) \to F(X)$ is a weak equivalence whenever $X$ and $Y$ are
cofibrant and $Y$ is obtained from $X$ by attaching cells of dimension
at least $n+1$.  Up to weak equivalence the morphism $W \to B$ is of
this form, and hence~(\ref{eq:1}) is a weak equivalence.

The simplicial sets in~(\ref{eq:1}) are the singular simplicial sets
of the associated mapping spaces, and composition with $f$ takes the
subspace of the mapping space
$\mathrm{map}_{\partial W / \mathrm{Top} / B'}(B, B)$ consisting of
weak equivalences onto the subspace of the mapping space
$ \mathrm{map}_{\partial W / \mathrm{Top} / B'}(W, B)$ consisting of
$n$-connected maps. Hence we have shown that the homotopy fibres of
\eqref{eq:ChangeOfTheta} are either empty or weakly equivalent to
$\hAut(u, \ell_{\partial W})$, in which case the equivalence is given
by acting on an element. Thus after forming the Borel construction by
$\hAut(u, \ell_{\partial W})$, the fibres of the map
\eqref{eq:ChangeOfTheta} become empty or contractible.
\end{proof}

We obtain a map
$\MM^\theta_n(W, \hat{\ell}_{\partial W})\hcoker \hAut(u,
\ell_{\partial W}) \to \MM^{\theta'}(W, \hat{u}
\circ\hat{\ell}_{\partial W})$ by taking the further Borel
construction by $\Diff_\partial(W)$, and this is also a weak homotopy
equivalence onto those path components which it hits. Therefore, in
terms of the point-set models introduced in Definition \ref{defn:NN},
the map
\begin{equation}\label{eq:555}
\mathcal{M}^\theta_n(W, \hat{\ell}_{\partial W})\hcoker \hAut(u, \ell_{\partial W}) \lra \mathcal{M}^{\theta'}(W, \hat{u} \circ\hat{\ell}_{\partial W})
\end{equation}
is also a weak homotopy equivalence onto those path components which it hits.

Each $f \in \hAut(u)$ gives a bundle map
$\hat{f} : \theta^*\gamma_{2n} \to \theta^*\gamma_{2n}$ which in turn
defines a continuous endofunctor of $\mathcal{C}_\theta$ (by
postcomposition on all $\theta$-structures), and hence a continuous
endomorphism of $B\mathcal{C}_\theta$. This defines an action of
$\hAut(u)$ on $B\mathcal{C}_\theta$. Moreover, if
$(\partial W, \hat{\ell}_{\partial W}) \in \mathcal{C}_\theta$ is an
object satisfying the assumption above that
$\ell_{\partial W} : \partial W \to B$ is a cofibration, then
$\hAut(u,\ell_{\partial W})$ preserves
$(\partial W, \hat{\ell}_{\partial W}) \in \mathcal{C}_\theta$. In
this case $\hAut(u,\ell_{\partial W})$ acts on both the morphism space
$\mathcal{C}_\theta(\emptyset, (\partial W, \hat{\ell}_{\partial W}))$
and on the path space
$\Omega_{[\emptyset, (\partial W, \hat{\ell}_{\partial W})]}
B\mathcal{C}_\theta$, and the map
$$\mathcal{C}_\theta(\emptyset, (\partial W, \hat{\ell}_{\partial W})) \lra \Omega_{[\emptyset, (\partial W, \hat{\ell}_{\partial W})]} B\mathcal{C}_\theta$$
is equivariant for this action.

\begin{definition}
  If $(W, \hat{\ell}_W)$ is a $\theta$-manifold, we write
  $\Omega_{(W, \hat{\ell}_W)}B\mathcal{C}_\theta$ for the path
  component of
  $\Omega_{[\emptyset, (\partial W, \hat{\ell}_{\partial
      W})]}B\mathcal{C}_\theta$ represented by the $\theta$-cobordism
  $(W, \hat{\ell}_W) : \emptyset \leadsto (\partial W,
  \hat{\ell}_{\partial W})$.

  If $(W, \hat{\ell}'_W)$ is a $\theta'$-manifold and
  $\ell'_W: W \overset{\ell_W}\to B \overset{u}\to B'$ is a
  Moore--Postnikov $n$-stage of $\ell'_W$, then we write
  $\Omega_{(W, \hat{\ell}'_W)}B\mathcal{C}_\theta$ for the union of
  those path components of
  $\Omega_{[\emptyset, (\partial W, \hat{\ell}_{\partial
      W})]}B\mathcal{C}_\theta$ represented by the
  $\pi_0\hAut(u, \ell_{\partial W})$-orbit of $(W, \hat{\ell}_W)$.
\end{definition}

\begin{lemma}\label{lem:ChangeOfThetaGenus}
  If $(W, \hat{\ell}'_W)$ is a $\theta'$-manifold and
  $\ell'_W: W \overset{\ell_W}\to B \overset{u}\to B'$ is a
  Moore--Postnikov $n$-stage of $\ell'_W$, then
  $\bar{g}^{\theta'}(W, \hat{\ell}'_W) = \bar{g}^\theta(W,
  \hat{\ell}_W)$.
\end{lemma}
\begin{proof}
  By the definition of the stable genus, it is enough to show that
  ${g}^{\theta'}(W, \hat{\ell}'_W) = {g}^\theta(W, \hat{\ell}_W)$. By
  definition of standard $\theta$- or $\theta'$-structure on
  $S^n \times D^n$, and hence on $W_{1,1}$, composing with $u$ sends
  standard $\theta$-structures to standard $\theta'$-structures, and
  so certainly
  ${g}^{\theta'}(W, \hat{\ell}'_W) \geq {g}^\theta(W,
  \hat{\ell}_W)$. Conversely, if
  $\hat{\ell} : T(S^n \times D^n) \to \theta^*\gamma_{2n}$ is a
  $\theta$-structure such that $\hat{u} \circ \hat{\ell}$ is a
  standard $\theta'$-structure, then $\hat{u} \circ \hat{\ell}$
  extends over the embedding $S^n \times D^n \hookrightarrow \bR^{2n}$
  described in Section \ref{sec:StdThetaStr}, so we have a commutative
  square
  \begin{equation*}
    \xymatrix{
      S^n \times D^n \ar[d] \ar[r]^-{\ell}& B \ar[d]^-{u}\\
      \bR^{2n} \ar[r] \ar@{-->}[ru]& B'.
    }
\end{equation*}
The map $u$ is $n$-co-connected, as it arises from the
Moore--Postnikov factorisation of $\ell'_W$. Therefore the square
admits a diagonal map, showing that $\hat{\ell}$ is a standard
$\theta$-structure.
\end{proof}

The following theorem implies Corollary \ref{cor:GeneralThetaCalc}
from the introduction, but also allows manifolds with non-empty
boundary.

\begin{theorem}\label{thm:ChangeOfTheta}
  Let $\theta' : B' \to BO(2n)$ be a tangential structure, and
  $(W, \hat{\ell}'_W)$ be a $\theta'$-manifold. Let
  $\ell'_W: W \overset{\ell_W}\to B \overset{u}\to B'$ be 
  a Moore--Postnikov $n$-stage of $\ell'_W$, i.e.\ a factorisation
  into an $n$-connected cofibration $\ell_W$ and an $n$-co-connected
  fibration $u$, and $\theta = \theta' \circ u$. Then there is a map
  $$\alpha^{\theta'}_W: \mathcal{M}^{\theta'}(W, \hat{\ell}'_{W}) \lra \left(\Omega_{[W, \hat{\ell}'_{W}]}B\mathcal{C}_\theta \right) \hcoker \hAut(u, \ell_{\partial W})$$
  which, if $2n \geq 6$ and $W$ is simply-connected, is acyclic in
  degrees satisfying $3* \leq \bar{g}^{\theta'}(W, \hat{\ell}'_W)-4$,
  and is a homology isomorphism in degrees satisfying
  $2* \leq \bar{g}^{\theta'}(W, \hat{\ell}'_W)-3$ if $\theta'$ is
  spherical.
\end{theorem}

\begin{proof}
  We have $\hAut(u, \ell_{\partial W})$-equivariant maps
  $$\alpha_W^\theta: \mathcal{M}^{\theta}_{n}(W; \hat{\ell}_{\partial W}) \lra \mathcal{C}_\theta(\emptyset, (\partial W, \hat{\ell}_{\partial W})) \lra \Omega_{[\emptyset, (\partial W, \hat{\ell}_{\partial W})]} B\mathcal{C}_\theta$$
  and so on taking Borel constructions and using \eqref{eq:555} we
  obtain a diagram
  \begin{equation*}
    \xymatrix{
      {\mathcal{M}^{\theta'}(W; \hat{u} \circ \hat{\ell}_{\partial W})}  & {\mathcal{M}^{\theta}_{n}(W; \hat{\ell}_{\partial W})\hcoker \hAut(u, \ell_{\partial W})} \ar[l] \ar[d]^-{\alpha_W^{\theta} \hcoker \hAut(u, \ell_{\partial W})}\\
      & {\big(\Omega_{[\emptyset,(\partial W, \ell_{\partial W)}]} B\mathcal{C}_\theta \big) \hcoker \hAut(u, \ell_{\partial W}).}
    }
  \end{equation*}
  The horizontal map is a weak homotopy equivalence onto the path
  components which it hits, and the
  $\hAut(u, \ell_{\partial W})$-orbit of $\hat{\ell}_W$ hits the path
  component containing $\hat{\ell}'_W$. Restricting to this component
  and inverting the homotopy equivalence, we obtain a map we define to
  be $\alpha_W^{\theta'}$.

  Any $\theta$-structure in the $\hAut(u, \ell_{\partial W})$-orbit of
  $\hat{\ell}_W$ becomes homotopic to $\hat{\ell}'_W$ after applying
  $u$, so has stable $\theta$-genus equal to
  $\bar{g}^{\theta'}(W, \hat{\ell}'_W)$ by Lemma
  \ref{lem:ChangeOfThetaGenus}. On each such path component, by
  Corollary \ref{cor:Main2} the map $\alpha_W^{\theta}$ is acyclic in
  degrees satisfying $3* \leq \bar{g}^{\theta'}(W, \hat{\ell}'_W)-4$,
  and a homology isomorphism in degrees satisfying
  $2* \leq \bar{g}^{\theta'}(W, \hat{\ell}'_W)-3$ if $\theta'$ is
  spherical ($\theta$ is spherical if and only if $\theta'$ is, as
  $u : B \to B'$ is $n$-co-connected), onto the path component which
  it hits. It therefore remains so after taking the Borel
  construction, and so the right-hand map has the same property.
\end{proof}

Finally, we observe that the maps used in establishing the main
theorem of \cite{GMTW},
$B\mathcal{C}_\theta \simeq \Omega^{\infty-1} \MTtheta$, are all
$\hAut(\theta)$-equivariant, which allows us to translate Theorem
\ref{thm:ChangeOfTheta} into the form given in the introduction.

%%% Local Variables:
%%% mode: latex
%%% TeX-master: "stability2"
%%% End:
 % General tangential structures
\appendix
\section{Homology equivalences and local coefficients}\label{sec:AbHEq}

Parts of this appendix are loosely based on one written by Johannes
Ebert for an early (uncirculated) draft of \cite{BER-W}. The authors
are grateful to him for allowing us to make use of it.

\subsection{Local coefficient systems}
\label{sec:loc-coef}

We shall take the point of view that a local coefficient system on a
space $X$ is a functor $\mathcal{L}: \pi_1(X) \to \mathrm{Ab}$ from
the fundamental groupoid of $X$ to the category of abelian groups.  It
is important that local systems form a \emph{category}, whose
morphisms $T: \mathcal{L} \to \mathcal{L}'$ are the natural
transformations.  In particular each coefficient system $\mathcal{L}$
has an endomorphism ring $R = \End(\mathcal{L})$, and in fact the
functor $\mathcal{L}: \pi_1(X) \to \mathrm{Ab}$ naturally factors
through the category of $R$-modules.  In particular the values
$\mathcal{L}(x)$ are naturally $\bZ[\pi_1(X,x)]$-$R$-bimodules.

Singular homology assigns to such a local coefficient system
$\mathcal{L}$ a chain complex $C_*(X;\mathcal{L})$ and homology groups
$H_*(X;\mathcal{L})$, and a natural transformation
$T: \mathcal{L} \to \mathcal{L}'$ induces maps of chains and of
homology $T: H_*(X;\mathcal{L}) \to H_*(X;\mathcal{L}')$.  If
$f: X \to Y$ is a continuous map and $\mathcal{L}: \pi_1(Y) \to \mathrm{Ab}$
is a coefficient system, then there is a pulled-back coefficient
system $f^*\mathcal{L}: \pi_1(X) \to\mathrm{Ab}$ and an induced map
\begin{equation*}
f_*:  H_*(X;f^*\mathcal{L}) \lra H_*(Y;\mathcal{L}).
\end{equation*}
For a fixed $\mathcal{L}$ the singular chains and the homology groups
$H_*(Y;\mathcal{L})$ are naturally modules over the ring
$R = \End(\mathcal{L})$, and the functoriality with respect to
$f: X \to Y$ is compatible with this structure.

For any space $X$ with coefficient system $\mathcal{L}$ there exists
for any $x \in H_k(X;\mathcal{L})$ a finite CW complex $K$ and a map
$f: K \to X$, such that $x$ is in the image of $f_*: H_*(K;f^*\mathcal{L})
\to H_*(X;\mathcal{L})$.

\subsubsection{Homology equivalences}\label{sec:homol-equiv}

\begin{definition}
  A local coefficient system $\mathcal{L}: \pi_1(X) \to \mathrm{Ab}$ is
  \emph{constant} if the action of $\pi_1(X,x)$ on $\mathcal{L}(x)$ is
  trivial for all $x \in X$.  The system is \emph{abelian} if the
  commutator subgroup of $\pi_1(X,x)$ acts trivially for all
  $x \in X$.
\end{definition}
\begin{definition}
  Let $f: X \to Y$ be a continuous map, and for each coefficient
  system $\mathcal{L}$ on $Y$ let
  \begin{equation}\label{eq:18}
    f_*: H_n(X; f^*\mathcal{L}) \lra H_n(Y;\mathcal{L})
  \end{equation}
  be the map induced by $f$.
  \begin{enumerate}[(i)]
  \item If~\eqref{eq:18} is an isomorphism for all $n$ and all
    $\mathcal{L}$ then $f$ is called an \emph{acyclic map}.
  \item If~\eqref{eq:18} is an isomorphism for all $n$ and all abelian
    $\mathcal{L}$ then $f$ is called an \emph{abelian equivalence}.
  \item If~\eqref{eq:18} is an isomorphism for all $n$ and all
    constant $\mathcal{L}$ then $f$ is called a \emph{homology
      equivalence}.
  \end{enumerate}
\end{definition}

We remark that the class of coefficient
systems $\mathcal{L}: \pi_1(Y) \to \mathrm{Ab}$ for
which~\eqref{eq:18} is an equivalence is closed under filtered
colimits and extensions.  Hence, if $f: X \to Y$ is an abelian
equivalence then~\eqref{eq:18} will be an isomorphism for any
coefficient system $\mathcal{L}$ obtained from abelian ones by
extensions and filtered colimits.  This class of local systems can
alternatively be described as those which for all $y \in Y$ satisfy
that $\mathcal{L}(y)$ admits no non-trivial homomorphism
$\mathcal{L}(y) \to W$ for any $\bZ[\pi_1(Y,y)]$-module $W$ with
$W^{[\pi_1(Y,y),\pi_1(Y,y)]} = 0$.

\subsubsection{Special abelian coefficient systems}

\begin{definition}
  A \emph{special abelian coefficient system} on a space $X$ consists
  of a commutative ring $R$ and a coefficient system
  $\mathcal{L}: \pi_1(X) \to \Mod_R$ such that
  $\{x \in X \mid \mathcal{L}(x) \neq 0\}$ is a path component of $X$,
  and for $x$ in this path component $\mathcal{L}(x)$ is a free
  $R$-module of rank 1.
\end{definition}

Note that a special abelian coefficient system is indeed an abelian
coefficient system, as $\pi_1(X,x)$ acts on $\mathcal{L}(x)$ via
$R^\times$, an abelian group.

\begin{lemma}\label{lem:special-abelian-suffices}
  If $f: X \to Y$ is a continuous map such that
  $f_*: H_*(X;f^*\mathcal{L}) \to H_*(Y;\mathcal{L})$ is an
  isomorphism for all special abelian coefficient systems
  $\mathcal{L}$ on $Y$, then $f$ is an abelian equivalence.
\end{lemma}
\begin{proof}
  If $\mathcal{L}: \pi_1(Y) \to \Mod_{\bF_2}$ is the unique special abelian
  coefficient system supported on the path component given by
  $y \in \pi_0(Y)$, then the dimension of $H_0(X;f^*\mathcal{L})$ as
  an $\bF_2$ vector space is the cardinality of the inverse
  image of $y$ under $\pi_0(X) \to \pi_0(Y)$.  Hence $f$ induces a
  bijection $\pi_0(X) \to \pi_0(Y)$.
  
  We can then restrict to the case where both $X$ and $Y$ are path
  connected and deduce that $f_*: H_1(X;\bZ) \to H_1(Y;\bZ)$ is an
  equivalence by using the constant coefficient system
  $\bZ: \pi_1(Y) \to \Mod_{\bZ}$, which is special abelian.  Then we
  pick a point $y \in Y$ and use the Hurewicz homomorphism
  $\pi_1(Y,y) \to A = H_1(Y)$ to define the special abelian coefficient system
  $\mathcal{L}_y: \pi_1(Y) \to \Mod_{\bZ[A]}$ by
  \begin{equation*}
    \mathcal{L}_y(z) = \bZ[A] \otimes_{\bZ[\pi_1(Y,y)]} \bZ[\pi_1(Y)(y,z)].
  \end{equation*}
  Then $H_*(Y;\mathcal{L}_y)$ and $H_*(X;f^*\mathcal{L}_y)$ calculate
  the homology of the universal abelian covers and if $f_*$ induces an
  isomorphism between these, then it induces an isomorphism for all
  abelian coefficient systems.
\end{proof}

The following lemma is the main advantage of special abelian systems
over all systems.  In general if $f: X \to Y$ and $\mathcal{L}$ is an
arbitrary coefficient system on $Y$, then the natural map
$\End(\mathcal{L}) \to \End(f^*\mathcal{L})$ can be very far from
surjective if the system is not abelian, even when $X$ and $Y$ are
path connected.

\begin{lemma}\label{lem:move-coeff-iso-to-target}
  Let $X$ and $Y$ be path connected.  If $f,g: X \to Y$ are two maps
  inducing equal homomorphisms $H_1(X;\bZ) \to H_1(Y;\bZ)$, then
  $f^*\mathcal{L}$ and $g^*\mathcal{L}: \pi_1(X) \to \Mod_R$ are
  isomorphic for any special abelian
  $\mathcal{L}: \pi_1(Y) \to \Mod_R$.

  Let $X$ be path connected and
  $\mathcal{L},\mathcal{L}': \pi_1(X) \to \Mod_R$ be isomorphic
  special abelian coefficient systems.  Then the $R$-module
  $\Hom_R(\mathcal{L},\mathcal{L}')$, consisting of natural
  transformations of functors $\pi_1(X) \to \Mod_R$, is free of rank
  one.

  Let $f: X \to Y$ be a map between path connected spaces and let
  $\mathcal{L}, \mathcal{L}': \pi_1(Y) \to \Mod_R$ be isomorphic
  special abelian coefficient systems.  Then the natural map
  \begin{equation*}
    \Hom_R(\mathcal{L},\mathcal{L}') \lra \Hom_R(f^*\mathcal{L}, f^*\mathcal{L}')
  \end{equation*}
  is an isomorphism.
\end{lemma}
\begin{proof}
  To verify the first claim we pick a basepoint $x \in X$.  Then
  $\mathcal{L}$ gives homomorphisms $\pi_1(Y,f(x)) \to R^\times$ and
  $\pi_1(Y,g(x)) \to R^\times$ given by equal elements of
  $H^1(Y;R^\times)$, and the isomorphism classes of $f^*\mathcal{L}$
  and $g^*\mathcal{L}$ are determined by the image of this element
  under the equal maps
  $f^* = g^*: H^1(Y;R^\times) \to H^1(X;R^\times)$.

  To verify the second claim we also use a basepoint $x \in X$.  It
  suffices to prove that the restriction map
  $\Hom_R(\mathcal{L},\mathcal{L}') \to
  \Hom_R(\mathcal{L}(x),\mathcal{L}'(x))$ is an isomorphism.  It is
  injective as $X$ is path connected, and any choice of an isomorphism
  $\mathcal{L} \cong \mathcal{L}'$ gives an element of
  $\Hom_R(\mathcal{L},\mathcal{L}')$ which restricts to a generator of
  the $R$-module $\Hom_R(\mathcal{L}(x),\mathcal{L}'(x)) \cong R$.

  The third claim now follows because the homomorphism is between free
  $R$-modules of rank one and sends a generator to a generator.
\end{proof}

\subsubsection{Homotopies}\label{sec:homotopies}

Finally, let us discuss homotopies.  The two inclusion
$i_0, i_1: X \to I \times X$ induce two functors
$\pi_1(i_0), \pi_1(i_1): \pi_1(X) \to \pi_1(I \times X)$, and there is
a canonical natural isomorphism $\pi_1(i_0) \Rightarrow \pi_1(i_1)$
given on the object $x$ by the path from $i_0(x) = (0,x)$ to
$i_1(x) = (1,x)$ along the interval.  If $R$ is an associative ring
and $\mathcal{L}: \pi_1(I \times X) \to \Mod_R$ is a coefficient
system, we may compose $\mathcal{L}$ with this natural transformation
to get an induced isomorphism
$T: i_0^*\mathcal{L} \to i_1^*\mathcal{L}$ of functors into $\Mod_R$.
For any homotopy $H: I \times X \to Y$ from $f$ to $g$ and coefficient
system $\mathcal{L}: \pi_1(Y) \to \Mod_R$ we therefore have an induced
isomorphism
\begin{equation*}
  T_H: f^*\mathcal{L} \lra g^*\mathcal{L}
\end{equation*}
of functors $\pi_1(X) \to \Mod_R$, fitting into a commutative
diagram
\begin{equation}\label{eq:20}
  \begin{aligned}
    \xymatrix{ H_*(X;f^*\mathcal{L}) \ar[d]_{T_H} \ar[r]^{f_*} &
      H_*(Y;
      \mathcal{L}) \ar@{=}[d]\\
      H_*(X;g^*\mathcal{L}) \ar[r]_{g_*} & H_*(Y; \mathcal{L}).
    }
  \end{aligned}
\end{equation}

\subsection{Invariance properties for abelian coefficients}
\label{sec:invariants}
The following observation is immediate from diagram~\eqref{eq:20}.
\begin{proposition}\label{prop:homotopic-maps}
  Let $f,g: X \to Y$ be homotopic maps and let
  $\mathcal{L}: \pi_1(Y) \to \Mod_R$ be any coefficient
  system.
  \begin{enumerate}[(i)]
  \item For any $n$,
    $f_*: H_n(X;f^*\mathcal{L}) \to H_n(Y;\mathcal{L})$ is surjective
    if and only if $g_*: H_n(X;g^*\mathcal{L}) \to H_n(Y;\mathcal{L})$
    is surjective.
  \item For any $n$,
    $f_*: H_n(X;f^*\mathcal{L}) \to H_n(Y;\mathcal{L})$ is injective
    if and only if $g_*: H_n(X;g^*\mathcal{L}) \to H_n(Y;\mathcal{L})$
    is injective.
  \end{enumerate}
  In particular, $f$ is a homology equivalence, abelian equivalence,
  or acyclic map, if and only if $g$ has that property.\qed
\end{proposition}

It turns out that in the case of abelian equivalences, the assumption
in the above proposition that $f$ and $g$ be homotopic can be replaced
by the following weaker assumption, called ``weakly homotopic'' by
\cite{MR0283788}.  \newcommand{\weakly}{almost}
\begin{definition}
  Two parallel maps
  \begin{equation*}
    \xymatrix{ X \ar@/^/[r]^f \ar@/_/[r]_g & Y}
  \end{equation*}
  are \emph{\weakly{} homotopic} if for all finite CW complexes $K$
  and maps $i: K \to X$ we have $i \circ f \simeq i \circ g$.
\end{definition}

\begin{proposition}
  \label{prop:almost-homotopic} Let $X$ and $Y$ be path connected, let
  $f,g: X \to Y$ be \weakly{} homotopic and let
  $\mathcal{L}: \pi_1(Y) \to \Mod_R$ be a special abelian coefficient
  system.
  \begin{enumerate}[(i)]
  \item If $f_*:H_k(X;f^*\mathcal{L}) \to H_k(Y;\mathcal{L})$ is
    surjective, then
    $g_*:H_k(X;g^*\mathcal{L}) \to H_k(Y;\mathcal{L})$ is also
    surjective.
  \item If $f_*:H_k(X;f^*\mathcal{L}) \to H_k(Y;\mathcal{L})$ is
    injective, then $g_*:H_k(X;g^*\mathcal{L}) \to H_k(Y;\mathcal{L})$
    is also injective.
  \end{enumerate}
\end{proposition}
\begin{proof}
  Without loss of generality we may assume $X$ is a CW complex.  The
  surjectivity part in fact holds true for any coefficient system,
  with essentially the same proof as in
  Proposition~\ref{prop:homotopic-maps}: any homology class in
  $H_n(Y;\mathcal{L})$ is by assumption the image of some
  $x \in H_n(X;f^*\mathcal{L})$ under $f_*$, and any such $x$ is
  supported on some finite subcomplex $K \subset X$, but the map
  \begin{equation*}
    (g\vert_{K})_*: H_n(K;g\vert_{K}^*\mathcal{L}) \to H_n(Y;\mathcal{L})
  \end{equation*}
  has the same image as $(f\vert_{K})_*$, by
  diagram~\eqref{eq:20}.

  Injectivity is not true for general coefficient systems, but when
  $\mathcal{L}: \pi_1(Y) \to \Mod_R$ is special abelian we may argue
  as follows.  Suppose
  $f_*: H_n(X;f^*\mathcal{L}) \to H_n(Y;\mathcal{L})$ is injective and
  suppose $x \in \Ker(g_*) \subset H_n(X;g^*\mathcal{L})$.  Again $x$
  is supported on some finite subcomplex $K \subset X$ and by
  assumption we may choose a homotopy
  $H_K: f\vert_{K} \simeq g\vert_{K}$, inducing an isomorphism of
  coefficient systems
  $T_{H_K}: f\vert_{K}^* \mathcal{L} \to g\vert_{K}^* \mathcal{L}$ and
  a commutative diagram
  \begin{equation*}
    \xymatrix{ H_*(K;f\vert_{K}^*\mathcal{L}) \ar[d]_{T_{H_K}} \ar[r]^-{(f\vert_{K})_*} &
      H_*(Y;
      \mathcal{L}) \ar@{=}[d]\\
      H_*(K;g\vert_{K}^*\mathcal{L}) \ar[r]^-{(g\vert_{K})_*} & H_*(Y; \mathcal{L}).
    }
  \end{equation*}
  Since any class in $H_1(X;\bZ)$ is supported on a finite subcomplex
  of $X$, we see that $f_* = g_*: H_1(X;\bZ) \to H_1(Y;\bZ)$.  If we
  assume, as we may, that $K$ is path connected, the special
  abelianness of $\mathcal{L}$ implies by
  Lemma~\ref{lem:move-coeff-iso-to-target} that
  \begin{equation*}
    \Hom_R(f^*\mathcal{L}, g^*\mathcal{L}) \lra
    \Hom_R(f\vert_{K}^* 
    \mathcal{L}, g\vert_{K}^*\mathcal{L}) 
  \end{equation*}
  is an isomorphism of $R$-modules, and in particular the isomorphism
  $T_{H_K}$ is induced by an isomorphism
  $T \in \Hom_R(f^*\mathcal{L},g^*\mathcal{L})$. We may therefore form
  a diagram
  \begin{equation*}
    \xymatrix{ H_*(K;f\vert_{K}^*\mathcal{L}) \ar[d]_{T_{H_K}} \ar[r] &
      H_*(X;
      f^*\mathcal{L}) \ar[d]_{T} \ar[r]^{f_*} & H_*(Y;\mathcal{L}) \ar@{=}[d]\\
      H_*(K;g\vert_{K}^*\mathcal{L}) \ar[r] & H_*(X; g^*\mathcal{L})
      \ar[r]^{g_*} & H_*(Y;\mathcal{L}),
    }
  \end{equation*}
  whose horizontal maps are induced by the maps of underlying spaces
  and whose vertical maps are induced by isomorphisms of coefficient
  systems, in which the left-hand square and the outer rectangle
  commute. Now the element $x \in H_n(X;g^*\mathcal{L})$ lifts to an
  $x' \in H_n(K;g\vert_{K}^*\mathcal{L})$ by assumption, and by the
  commutativity of the outer rectangle we have
  $(f\vert_{K})_*T_{H_K}^{-1}(x') = (g\vert_{K})_*(x') = g_*(x)=0$.
  As $f_*$ is injective and the left-hand square commutes, it follows
  that $x=0$, as required.
\end{proof}

\begin{corollary}\label{cor:AlmHtcImpliesAbEq}
  Let $f,g: X \to Y$ be \weakly{} homotopic maps.  Then $f$ is an
  abelian equivalence if and only if $g$ is.
\end{corollary}
\begin{proof}
  It is easy to reduce to the case of path connected spaces, and by
  Lemma~\ref{lem:special-abelian-suffices} it suffices to consider
  special abelian coefficient systems, where the claim follows from
  Proposition~\ref{prop:almost-homotopic}.
\end{proof}

\subsection{Ladder diagrams and abelian equivalences}

A \emph{ladder diagram} is a diagram of spaces of the shape
\begin{equation}\label{long-ladder-diagram}
  \begin{aligned}
    \xymatrix{
      X_0 \ar[r]^{g_1} \ar[d]^{f_0} &  X_1 \ar[r]^{g_2} \ar[d]^{f_1} & X_2\ar[d]^{f_2}\ar[r] & \cdots \\
      Y_0 \ar[r]^{h_1}  &  Y_1 \ar[r]^{h_2}  & Y_2\ar[r] & \cdots. \\
    } 
  \end{aligned}
\end{equation}
A commutative ladder diagram is one in which each square commutes, and
a homotopy commutative ladder diagram is one in which each square
commutes up to homotopy.

For a ladder diagram as above we shall write $X_\infty$ for the
telescope (homotopy colimit) of the top row and $Y_\infty$ for the
bottom row.  If the diagram commutes, there is an induced map
$X_\infty \to Y_\infty$.  If the diagram homotopy commutes, then a
choice of homotopies $H_i: h_i \circ f_{i-1} \simeq f_i \circ g_i$
induces a map $X_\infty \to Y_\infty$.  The homotopy class of the map
of telescopes in general depends on the choice of homotopies $H_i$,
but we have the following result.

\begin{lemma}\label{lem:LadderHtpies}
  For a homotopy commutative ladder
  diagram~\eqref{long-ladder-diagram}, any two choices of homotopies
  $H_i: h_i \circ f_{i-1} \simeq f_i \circ g_i$ induce \weakly{}
  homotopic maps $X_\infty \to Y_\infty$.  In particular, the property
  of whether the induced map $X_\infty \to Y_\infty$ is an abelian
  equivalence depends only on the underlying commutative ladder
  diagram in the homotopy category, not on the choices of homotopies
  $H_i$.
\end{lemma}
\begin{proof}
  The induced maps become equal when restricted to each
  $X_i \subset X_\infty$ and hence they become homotopic when
  restricted to the finite telescope of $(X_0 \to \dots \to X_i)$ for
  all $i$.  Any map from a finite CW complex will factor through one of
  these finite telescopes.
\end{proof}

\begin{definition}
  A map $f: X \to Y$ is an \emph{almost homotopy
    equivalence} if there exists a map $g: Y \to X$ such that
  $f \circ g$ and $g \circ f$ are each \weakly{} homotopic to the
  identity.
\end{definition}

The following is immediate from Corollary \ref{cor:AlmHtcImpliesAbEq}
and this definition.

\begin{corollary}
  An almost homotopy equivalence is an abelian equivalence.\qed
\end{corollary}

The main technical result of this subsection is as follows.

\begin{proposition}\label{acyclicity-key-proposition}
  A homotopy commutative ladder diagram as
  in~\eqref{long-ladder-diagram} induces an almost homotopy
  equivalence $X_\infty \to Y_\infty$, and hence an abelian
  equivalence, provided that the following condition holds: for each $n$,
  there exists a $k>n$, a map $i: Y_n \to X_k$ such that
  $f_{k} \circ i \simeq h_k \circ \dots \circ h_{n+1}$, as well as a
  map $j: Y_n \to X_k$ such that
  $j \circ f_{n} \simeq g_k \circ \dots \circ g_{n+1}$.
\end{proposition}
\begin{proof}
  After restricting to cofinal subsequences and reindexing, we may
  assume that we have maps $i_n, j_n: Y_{n-1} \to X_n$, making each
  triangle in the diagram
  \begin{equation*}
    \xymatrix{
      X_{n-1} \ar[r]^{g_n}\ar[d]_{f_{n-1}} & X_{n} \ar[d]^{f_n}\\
      Y_{n-1} \ar@/^/[ur]^-{j_n} \ar@/_/[ur]_{i_n} \ar[r]_{h_n} & Y_{n},
    }
  \end{equation*}
  homotopy commute, but possibly $i_n$ and $j_n$ are not homotopic.
  If we write
  $k_n = j_{n} \circ f_{n-1} \circ i_{n-1} : Y_{n-2} \to X_n$, we may
  calculate
  \begin{align*}
    f_n \circ k_n = f_n \circ j_n \circ f_{n-1} \circ i_{n-1} \simeq
                    f_n \circ g_n \circ i_{n-1} \simeq h_n \circ f_{n-1} &\circ i_{n-1}
                     \simeq h_n \circ h_{n-1},\\
    k_n \circ f_{n-2}  = j_n \circ f_{n-1} \circ i_{n-1} \circ f_{n-2}
                        \simeq j_n \circ h_{n-1} \circ f_{n-2} & \simeq j_n \circ f_{n-1}
                        \circ g_{n-1} \\
    &= g_n \circ g_{n-1}.
  \end{align*}
  In other words, after passing to a further subsequence and
  reindexing we have a single diagonal map
  \begin{equation*}
    \xymatrix{
      X_{n-1} \ar[r]^{g_n}\ar[d]_{f_{n-1}} & X_{n} \ar[d]^{f_n}\\
      Y_{n-1} \ar[ur]^{k_n} \ar[r]_{h_n} & Y_{n},
    }
  \end{equation*}
  such that both triangles homotopy commute.  The resulting diagram
  may be reinterpreted as a homotopy commutative diagram
  \begin{equation*}
    \xymatrix{
      X_0 \ar[r]^{g_1} \ar[d]^{f_0} &  X_1 \ar[r]^{g_2} \ar[d]^{f_1} &
      X_2\ar[d]^{f_2}\ar[r] & \cdots \\
      Y_0 \ar[r]^{h_1}\ar[d]^{k_1}  &  Y_1 \ar[r]^{h_2}\ar[d]^{k_2}
      & Y_2\ar[r]\ar[d]^{k_3} & \cdots \\ 
      X_1 \ar[r]^{g_2} &  X_2 \ar[r]^{g_3} &
      X_3\ar[r] & \cdots \\
    }
  \end{equation*}
  such that the vertical maps from the top row to the bottom row are
  homotopic to $g_n: X_{n-1} \to X_n$.  Any choice of homotopies in
  the squares now induce maps
  $k \circ f: X_\infty \to Y_\infty \to X_\infty$, whose composition
  $X_\infty \to X_\infty$ is \weakly{} homotopic to the ``shift map'' of
  the telescope, and hence \weakly{} homotopic to the identity map of
  $X_\infty$.  By a similar argument we see that
  $f \circ k: Y_\infty \to X_\infty \to Y_\infty$ is \weakly{} homotopic
  to the identity, and hence $f: X_\infty \to Y_\infty$ is an almost
  homotopy equivalence, as claimed.
\end{proof}

\subsection{The generalised group completion theorem with local coefficients}

Let $\mathcal{C}$ be a (possibly non-unital) topological category, and
write $N_{\bullet} \mathcal{C}$ for the nerve, so that
$N_0 \mathcal{C}$ is the space of objects and $N_1 \mathcal{C}$ is the
space of morphisms. Let $B \mathcal{C}$ be the geometric realisation
of the nerve as a \emph{semi-simplicial} space. The space of morphisms
from $a$ to $b$ is denoted $\mathcal{C}(a,b)$.

Fix a sequence of objects and morphisms
$$c_0 \stackrel{f_1}{\lra} c_1 \stackrel{f_2}{\lra} c_2 \stackrel{f_3}{\lra}\cdots $$
in $\mathcal{C}$. Then, for objects $a$ and $b$ of $\mathcal{C}$ and a
morphism $g: a \to b$, we obtain a commutative ladder diagram
\begin{equation*}
  \begin{gathered}
    \xymatrix{
      {\mathcal{C}(b,c_0)} \ar[d]^{- \circ g}\ar[r]^{f_1 \circ -} & {\mathcal{C}(b,c_1)} \ar[d]^{- \circ g} \ar[r]^{f_2 \circ -} & {\mathcal{C} (b,c_2 )} \ar[d]^{- \circ g}\ar[r]& \cdots \\
      {\mathcal{C}(a,c_0)} \ar[r]^{f_1 \circ -} & {\mathcal{C}(a,c_1)}  \ar[r]^{f_2 \circ -} & {\mathcal{C} (a,c_2 )} \ar[r]& \cdots. \\}
  \end{gathered}
\end{equation*}

The generalised version of the group completion theorem is then as follows.

\begin{theorem}\label{thm:generalized-group-completion}
  Let $\mathcal{C}$ be a (perhaps non-unital) topological category and
  let $c_n$ and $f_n$ ($n \geq 0$) be as above.  Suppose in addition
  that
  \begin{enumerate}[(i)]
  \item\label{item:fibrancy-condition} The map
    $(d_0, d_1): N_1\mathcal{C} \to (N_0\mathcal{C})^2$ is a
    fibration.
  \item\label{item:asdf} For any morphism $g \in \mathcal{C}(a,b)$,
    the induced map
    \begin{equation}\label{eq:5}
      \hocolim_{i \to \infty} \mathcal{C} (b,c_i) \to \hocolim_{i \to \infty} \mathcal{C} (a,c_i)
    \end{equation}
    is an abelian homology equivalence.
  \end{enumerate}
  Then the natural map 
  \begin{equation}\label{eq:3}
    \hocolim_{i \to \infty} \mathcal{C} (a,c_i) \lra \hocolim_{i \to
      \infty}\Omega_{[a, c_i]} B \mathcal{C}
  \end{equation}
  is acyclic.
\end{theorem}

\begin{proof}[Proof sketch]
  If we replace ``abelian homology equivalence'' and ``acyclic'' in
  the statement by ``homology equivalence'', this is a standard
  argument based on the notion of homology fibration from
  \cite{McDuff-Segal}, cf.\ \cite[Theorem 3.2]{Tillmann}, \cite[\S
  7]{GMTW}, or \cite[\S 7.4]{GR-W2}, with a small modification to deal
  with the possible lack of units.  Let us outline the argument.  For
  $i \geq 0$ let $E^i_\bullet$ be the semi-simplicial space with
  $E^i_p \subset N_{p+1}\mathcal{C}$ the subspace consisting of
  sequences of $(p+1)$ composable morphisms ending in $c_i$, i.e.\ the
  inverse image of $c_i$ under
  $d_0^{p+1}: N_{p+1}\mathcal{C} \to N_0 \mathcal{C}$.  Postcomposing
  the last morphism with $f_{i+1} : c_i \to c_{i+1}$ defines a
  semi-simplicial map $F^i_\bullet: E^i_\bullet \to E^{i+1}_\bullet$,
  and instead appending $f_{i+1}$ to the sequence of composable
  morphisms defines maps $E^i_p \to E^{i+1}_{p+1}$ which supply a
  homotopy between
  $\vert F^i_\bullet \vert: |E^i_\bullet| \to |E^{i+1}_\bullet|$ and
  the constant map to the vertex $f_{i+1} \in E_0^{i+1}$. Thus the
  mapping telescope $\hocolim\limits_{i \to \infty} |E^i_\bullet|$ is
  contractible.

  The map $d_{p+1}: N_{p+1}\mathcal{C} \to N_p\mathcal{C}$ restricts
  to a fibration $E^i_p \to N_p \mathcal{C}$, by assumption
  (\ref{item:fibrancy-condition}), which induces a quasi-fibration
  \begin{equation*}
    \hocolim\limits_{i \to \infty} E^i_p \lra N_p\mathcal{C}.
  \end{equation*}
  By \cite[Propositions 4 and 5]{McDuff-Segal} this becomes a homology
  fibration after taking geometric realisation, using
  assumption~(\ref{item:asdf}).  Hence the inclusion of the fiber over
  $a \in B\mathcal{C}$ into the homotopy fiber is a homology
  equivalence, and since the domain is contractible this inclusion may
  be identified with~(\ref{eq:3}).

  When the assumption is strengthened to~(\ref{eq:5}) being an abelian
  equivalence, the conclusion may be strengthened to~(\ref{eq:3})
  being an abelian equivalence as well.  (See \cite{MillerPalmer} for
  a treatment of the main results of \cite{McDuff-Segal} in this
  setting.)  Since the codomain has abelian fundamental group, this
  implies the map is acyclic.
\end{proof}

%%% Local Variables: 
%%% mode: latex
%%% TeX-master: "stability2"
%%% End: 
 % Homology equivalences and local coefficients
\bibliographystyle{amsalpha}
\bibliography{biblio}

\end{document}